\documentclass[UKenglish,a4paper,reqno]{amsart}
\usepackage{amsmath,centernot,cases,amsfonts,csquotes,amssymb,esint,empheq,etoolbox,geometry,url}
\catcode`@=11 
\def\blfootnote{\xdef\@thefnmark{}\@footnotetext}
\catcode`@=12 
\usepackage{tabularx}
\theoremstyle{plain}
\usepackage[font=small,format=hang,labelfont={sf,bf}]{caption}
\usepackage{hyperref}
\hypersetup{					
			colorlinks=true,
			linkcolor=blue,
                        linktoc=page,
			anchorcolor=black,
			citecolor=blue,
			urlcolor=blue,
			pdftitle={},
}

\mathchardef\emptyset="001F

\numberwithin{equation}{section}
\newcommand{\e}{\varepsilon}
\newcommand{\be}{\begin{equation}}
\newcommand{\ee}{\end{equation}}

\newcommand{\E}{\mathfrak{E}}
\newcommand{\R}{{\mathbb R}}
\newcommand{\Sn}{{\mathbb S}}
\newcommand{\N}{{\mathbb N}}

\newcommand{\Z}{{\mathbb Z}}
\newcommand{\Rk}{{\R}^k}
\newcommand{\Rd}{{\R}^d}
\newcommand{\Rkd}{\R^{k\times d}}
\newcommand{\Rdk}{\R^{k\times d}}
\newcommand{\ClosureE}{\mathfrak{E}^{\alpha,\vartheta}_{w}}
\newcommand{\hd}{{\mathcal H}^{d-1}}
\newcommand{\Ld}{{\mathcal{L}}^d}
\newcommand{\meno}{\textup{-}}
\newcommand{\jump}[1]{J_{#1}}
\newcommand{\mres}{\mathbin{\vrule height 1.6ex depth 0pt width
0.13ex\vrule height 0.13ex depth 0pt width 1.3ex}}

\newtheorem{theorem}{Theorem}[section]
\newtheorem{corollary}[theorem]{Corollary}
\newtheorem{lemma}[theorem]{Lemma}
\newtheorem{proposition}[theorem]{Proposition}

\theoremstyle{definition}
\newtheorem{remark}[theorem]{Remark}
\newtheorem{definition}[theorem]{Definition}

\address{SISSA, Via Bonomea 265, 34136 Trieste}
\author[G. Dal Maso]{Gianni Dal Maso}
\author[D.Donati]{Davide Donati}
\title[Homogenisation of Vectorial Free-discontinuity problems]{Homogenisation of vectorial free-discontinuity functionals with cohesive type surface terms}

\begin{document}

\thanks{Preprint SISSA 18/2024/MATE}
\begin{abstract}
  The results on  $\Gamma$-limits of sequences of free-discontinuity functionals with bounded cohesive surface terms are extended to the case of vector-valued functions. In this framework, we prove an integral representation result for the $\Gamma$-limit, which is then used to study   deterministic and stochastic homogenisation problems for this type of functionals.
\end{abstract}
\maketitle
\vspace{-0.6cm}


\section{Introduction}
Free-discontinuity problems are minimisation problems for functionals of the form
\begin{equation}\label{intro:free disco}
    \int_{A}f(x,\nabla u)\, dx+\int_{J_u}g(x,[u],\nu_u)\, d\hd, 
\end{equation}
where 
\begin{itemize}
    \item $A$ is a bounded open set in $\Rd$,
    \item $f$ and $g$ are two given scalar functions,
    \item the unknown $u$ is a function defined in $A$ with values in $\Rk$,
    \item $J_u$ is the  $(d-1)$-dimensional essential discontinuity set of $u$,  whose location and size are unknown,
    \item $\nabla u$ is the gradient of $u$ in $A\setminus J_u$, 
    \item $\nu_u$ is the measure theoretical unit normal to $\jump{u}$,
    \item $[u]=u^+-u^-$, where $u^+ $ and $u^-$ are the traces of $u$ on both sides of $J_u$,
    \item $\hd$ is the $(d-1)$-dimensional Hausdorff measure.
\end{itemize} 
Since the function $u$ exhibits an essential discontinuity set, a suitable framework to study these problems is the space $BV(A;\Rk)$ of functions of bounded variation. Problems of this type in $BV(A;\Rk)$ have been extensively studied (see for instance \cite[Chapter 4.6]{AmbFuscPall} and \cite{BraidesApprox}).

In his pioneering work \cite{Griffith}, Griffith  introduced the idea that the crack growth  in an elastic material is determined by the competition between the stored elastic energy and the energy spent to open a new portion of the crack.  Adopting this point of view, Francfort and Marigo proposed in \cite{francfort1998revisiting} a variational model to study crack growth which includes the requirement that at each time $t>0$ the displacement $u(t)$  of the elastic body minimises a free-discontinuity functional of the form \eqref{intro:free disco}, where $g$ also depends on the cracks present before $t$. In these models, the crack at time $t$ is the union of $J_{u(s)}$ for $s\leq t$, the  volume integral represents the stored elastic energy, while the surface integral is related the energy spent to produce the crack. For an overview on this subject we refer to \cite{BourFrancMari}.

In cohesive models of fracture mechanics, it is natural to assume that the function $g$ satisfies the following growth conditions
\begin{equation}\label{intro: growth for g}
    c(|\zeta|\land 1)\leq g(x,\zeta,\nu)\leq C(|\zeta|\land1),
\end{equation}
for suitable constants $0<c\leq C$, where for $s,t\in\R$, $s\land t=\min\{s,t\}$. In particular, the Dugdale model \cite{Dugdale1960YieldingOS} can be reformulated in the language of \cite{BourFrancMari}, using $g(\zeta):= a(|\zeta|\land b)$, for suitable constants $a,b\in[0,+\infty)$. Unfortunately, under hypotheses \eqref{intro: growth for g} the functional \eqref{intro:free disco} is never lower semicontinuous in  $BV(A;\Rk)$ with respect to the weak$^*$ convergence. Under suitable hypotheses (see for instance \cite[Theorem 3.1]{Braides1995} and \cite{bouchitte1998global}), its lower semicontinuous envelope has the form 
\begin{equation}\label{intro: reference problem}
     \int_{A} f(x,\nabla u)\, dx+\int_A f^\infty\Big(x,\frac{dD^cu}{d|D^cu|}\Big)\, d|D^cu|+\int_{\jump{u}\cap A} g(x, [u],\nu_u)\,d\hd,
\end{equation}
for different functions $f$ and $g$, where $f$ satisfy the growth conditions 
\begin{equation}\label{intro growth for f}
    c(|\xi|-1)\leq f(x,\xi)\leq C(|\xi|+1) \quad \text{for every $x\in\Rd$ and  $\xi\in\Rkd$},
\end{equation} and $f^\infty$ is its recession function with respect $\xi$. Here and in the rest of the paper  $D^cu$ is the Cantor part of the measure $Du$ (see \cite[Definition 3.91]{AmbFuscPall}), and $dD^cu/d|D^cu|$ is the Radon-Nikod\'ym derivative of $D^cu$ with respect to its total variation. 

Since condition \eqref{intro: growth for g} does not allow us to control the total variation of $|Du|(A)$ along a minimising sequence, to gain coerciveness it is convenient to extend the functional to a space larger than $BV(A;\Rk)$, where all terms of \eqref{intro: reference problem} can still be defined. We choose to extend these functionals to the space $GBV_\star(A;\Rk)$, studied by the second author in \cite{donati2023new}, and defined using the scalar version $GBV_\star(A)$ introduced in \cite{DalToa22}. This is a vector space and, although $Du$ is not defined for every $u\in GBV_\star(A;\Rk)$, one can always define  $\nabla u$ and $D^cu$ in a convenient way. Moreover, the bounds for a minimising sequence obtained from \eqref{intro: growth for g} are enough to apply a suitable compactness result with respect to convergence $\Ld$-a.e. (see \cite[Theorem 5.4]{DalToa22}, \cite[Theorem 7.13]{DalToa23}, and \cite[Theorem 4.8]{donati2023new}). 

In this paper, we study the $\Gamma$-convergence with respect to the convergence in measure of sequences of functionals of the form \eqref{intro: reference problem}, where $f$ and $g$ satisfy \eqref{intro growth for f} and  \eqref{intro: growth for g}.
We first prove a compactness result (Theorem \ref{thm:compactness for E}),  which shows that under suitable hypotheses (see Definitions \ref{def:good bulk integrands} and \ref{def:good surface integrands}) on the integrands $f_n$ and $g_n$  corresponding to a  sequence of functionals $(E_n)_n$, there always exists a subsequence $\Gamma$-converging to a functional $E$, whose volume and jump terms can be written as  integral functionals associated to some functions $f$ and $g$ (see Theorem \ref{thm:rappresentazione}).

To obtain a complete integral representation of $E$ as in \eqref{intro: reference problem} it remains to deal with the term  depending on the Cantor part. This requires stronger hypotheses on $f_n$ and $g_n$, studied in Section \ref{sec:smallerclass}, concerning quantitative estimates on the behaviour of $f_n$ and $g_n$ near $\infty$ and near $0$, respectively. Under these assumptions, we show that $f$ and $g$ can be obtained by taking suitable limits of the infima of some minimisation problems for $E_n(\cdot,Q)$ on suitable small cubes $Q$ and that, if $f$ is independent of $x$, we have the  integral representation  
\begin{equation}\nonumber\label{intro sezione 8}
    E(u,A)= \int_{A} f(\nabla u)\, dx+\int_A f^\infty\Big(\frac{dD^cu}{d|D^cu|}\Big)\, d|D^cu|+\int_{\jump{u}\cap A} g(x, [u],\nu_u)\,d\hd.
\end{equation}

These results are then applied to the case of  homogenisation, where the functionals $E_n$ are given by
\begin{equation}\label{eq:stochastic integrals}
    E_n(u,A)= \int_{A} f\Big(\frac{x}{\e_n},\nabla u\Big)\, dx+\int_A f^\infty\Big(\frac{x}{\e_n},\frac{D^cu}{|D^cu|}\Big)\, d|D^cu|+\int_{\jump{u}\cap A} g\Big(\frac{x}{\e_n}, [u],\nu_u\Big)\,d\hd,
\end{equation}
for a sequence $\e_n\to 0^+$. We determine general conditions on $f$ and $g$ which imply that the $\Gamma$-limit of these sequence of functionals exists and has the form 
\begin{equation}\label{intro:homogenised}
    E_{\rm hom}(u,A)= \int_{A} f_{\rm hom}(\nabla u)\, dx+\int_A f_{\rm hom}^\infty\Big(\frac{dD^cu}{d|D^cu|}\Big)\, d|D^cu|+\int_{\jump{u}\cap A} g_{\rm hom}( [u],\nu_u)\,d\hd,
\end{equation}
for some functions $f_{\rm hom}$ and $g_{\rm hom}$ independent of $x$. These functions $f_{\rm hom}$ and $g_{\rm hom}$ are obtained by taking the limits of infima of certain minimisation problems for $E_n$, with $\e_n=1$, on cubes whose sides tend to infinity.

Thanks to these properties,  we can apply the previous results to stochastic homogenisation problems, where $f$ and $g$ are random integrands satisfying suitable properties (see Definition \ref{def:stochastic integrands}). Under these assumptions, we show that the sequence $E_n$ $\Gamma$-converges to $E$ given by \eqref{intro:homogenised} almost surely.

These results were obtained in the scalar case in \cite{DalToa23,DalToa23b}. 
In these papers, the truncation $u^{(m)}:=(u\land m)\lor(-m)$  and the corresponding estimates for the functionals are frequently used.  The main difficulty in the vector-valued setting is that these truncations have to be replaced by a sort of {\it smooth} truncations of the form $\psi_{m}^i\circ u$, $i\in\{1,...,m\}$ for suitable functions $\psi_m^i\in C^{\infty}_c(\Rk;\Rk)$ satisfying 
\begin{equation*}
    \psi_m^i(y)=y \quad \text{for $|y|\leq R_m$ \quad and \quad  } \psi_m^i(y)=0 \quad \text{ for $|y|\geq \sigma^mR_m$},
\end{equation*}
where  the sequence $R_m\to +\infty$ and the constant $\sigma>2$ are prescribed. These types of truncations were already considered in previous works (see, for instance, \cite{BraDefVit,CagnettiDetFree,  CeladaDalM, FonsecaMuller,FuscoHutchinson}). The main difficulty here is that an estimate for $E_n({\psi_m^i\circ u},A)$ in terms of $E_n(u,A)$ cannot be obtained for every $i\in\{1,...,m\}$, but only for a suitable choice of $i$, depending on $n,m, u,$ and $A$. This requires deep changes in many technical results of \cite{DalToa23} and \cite{DalToa23b}, which introduce new terms depending on $m$ and which cannot be neglected (see, for instance, Lemma \ref{lemma:teo troncature} and Proposition \ref{prop:g is not in gtheta}).

Our results can be considered as a preliminary step for the study of the asymptotic behaviour of the crack growth in a heterogeneous  cohesive material, when the size of the grains tends to zero. Indeed, according to the approach of \cite{MielkeRoubStef2008} (see also \cite{MainikMielke2009,NumerApprox,mielke2015rate,MielTimof}), in the variational theory of rate-independent problems  the convergence of the quasi-static evolutions requires the $\Gamma$-convergence of the corresponding energy-dissipation functionals and the existence of a joint recovery sequence. Our results give a complete answer to the first part of this program in the case of cohesive models for fracture. In this paper we do not address the construction of a joint recovery sequence.

The problem of homogenisation of free-discontinuity functionals has already been addressed in a wide variety of cases. In the now classical work \cite{BraDefVit}, Braides, Defranceschi, and Vitali first dealt with the deterministic periodic case for integrands $f$ and $g$ satisfying 
\begin{eqnarray}
    \label{intro:pgrowth}&\displaystyle c|\xi|^p\leq f(x,\xi)\leq C(1+|\xi|^p)\quad \text{for every $x\in\Rd$ and $\xi\in\Rkd$},\\
    &\displaystyle\label{intro:g bradefrvit} c(1+|\zeta|)\leq g(x,\zeta,\nu)\leq C(1+|\zeta|)\quad \text{for every }x\in\Rd, \zeta\in\Rk,\nu\in\Sn^{d-1}
\end{eqnarray}
for some suitable constants $c,C>0$ and for $p>1$. Under these growth conditions one can use the compactness results on $SBV$ by Ambrosio \cite[Theorem 2.1]{AmbroExist} (see also \cite{AmbrosioComp1,AmbrosioANew})  and the functionals to homogenise are of type \eqref{intro:free disco}. The main result of their paper is a periodic  homogenisation theorem for such functionals. They also show that, in this case, the bulk integrand $f_{\rm hom}$ of the $\Gamma$-limit  is determined only by $f$ and  the surface integrand $g_{\rm hom}$ is determined only by $g$. 

Later Giacomini and Ponsiglione studied the scalar case $k=1$ in \cite{GiacominiPonsiglione}, assuming that $f$ satisfies \eqref{intro:pgrowth} and that $g$ does not depend on $\zeta$ and
\[
c\leq g(x,\nu)\leq C \quad \text{ for every } x\in\Rd \text{ and } \nu\in\mathbb{S}^{d-1}.
\] 
Note that in their result no periodicity assumption on $f$ and $g$ is made. As in the periodic case studied by Braides et al. \cite{BraDefVit},  under these hypotheses they are able to show that $f_{\rm hom}$ dependes only on $f$ and that $g_{\rm hom}$ depends only on $g$.  

More recently, Cagnetti et al. considered the vectorial case $k\geq 1$ in \cite{CagnettiDetFree} with non-periodic integrands $f$ and $g$, where $f$ satisfies \eqref{intro:pgrowth} and  $g$ satisfies 
\begin{equation}\label{intro:growthCagnetti} 
    c\leq g(x,\zeta,\nu)\leq C(1+|\zeta|)\quad \text{for every $x\in\Rd$, $\zeta\in\Rk$, and $\nu\in\Sn^{d-1}$}.
\end{equation}

In a subsequent work \cite{CagnettiStocFree}, the results of  \cite{CagnettiDetFree} are then employed to deal with the stochastic homogenisation of free-discontinuity integral functionals satisfying  growth conditions \eqref{intro:pgrowth} and \eqref{intro:growthCagnetti}. Under the standard assumptions of stochastic homogenisation they prove an almost sure  $\Gamma$-convergence result for functionals of type \eqref{intro:free disco}.

Building on the techniques devised in \cite{CagnettiDetFree,CagnettiStocFree}, the same authors tackled in \cite{CagnettiGlobal} the problem of deterministic and stochastic homogenisation for sequences of type \eqref{eq:stochastic integrals}  under the hypotheses that both $f$ and $g$ have linear growth. In this case, the underlying function space becomes $BV(A;\Rk)$ and the integral depending on the Cantor part of the derivative has to be considered in \eqref{eq:stochastic integrals}. 

Recently, the problem of homogenisation of free-discontinuity functionals was also addressed in the context of functions of bounded deformation.
In \cite{FriedrichPeruSolo}, Friedrich, Peruguini, and Solombrino tackle in dimension $d=2$ the $\Gamma$-convergence with respect to the convergence in measure of functionals 
$E_n$  given by 
\begin{equation}
    \label{intro: linearised homogenisation}
E_n(u,A)=\int_{A}f\Big(\frac{x}{\e_n},\mathcal{E}u\Big)\, dx+\int_Ag\Big(\frac{x}{\e_n},[u],\nu_u\Big)\, \hd,
\end{equation}
where $A\subset\R^2$ is open and bounded, $u\in GSBD^p(A;\R^2)$, and 
$\mathcal{E}u=(\nabla u+\nabla u^T)/2$ is the approximate symmetric gradient of $u$, assuming that the integrand $f\colon \R^2\times \R^{2\times 2}_{\rm sym}\to [0,+\infty)$ has $p$-growth and that $g$ is bounded from below by a positive constant.

In  a recent work \cite{DonFri23}, the stochastic   homogenisation of functionals of type \eqref{intro: linearised homogenisation} restricted to  piecewise rigid functions, i.e., functions $u(x)=\sum_{i=1}^\infty (R_ix+b_i)\chi_{E_i}$ for $R_i\in    \R^{d\times d}$ skew-symmetric, $b_i\in\R^d,$ and $E_i$ of finite perimeter, was tackled in dimensions $d\geq 2$.

\medskip

Our paper is organised as follows. In Section \ref{sec:notationpreliminaries} we fix the notation and lay out the basic tools used throughout the paper.
We then introduce in Section \ref{sec:classE} the collections of volume integrands $\mathcal{F}$ and surface integrands $\mathcal{G}$ which will be the object of our study; we also introduce a class of abstract functionals $\E$ which contains the integral functionals corresponding to integrands in $\mathcal{F}$ and $\mathcal{G}$.

Section \ref{seq:compactness} is devoted to proving a compactness result for sequences of integral functionals. In Section \ref{sec:partial representation} we investigate the properties of the class $\E$, by proving an integral representation result for $\E$, showing that if $E\in\E$ is lower-semicontinuous with respect to the convergence in measure then its \lq\lq  absolutely continuous part\rq\rq $E^a$ and its \lq\lq jump part\rq\rq $E^j$  can be represented by integral functionals, with integrands $f\in\mathcal{F}$ and $g\in\mathcal{G}$, respectively.

In Section \ref{sec:smallerclass} we introduce two smaller collections of integrands $\mathcal{F}^\alpha$ and $\mathcal{G}^\vartheta$ and study the $\Gamma$-limits of sequences of integral functionals associated with them. We then prove in Section \ref{sec:repr} that under some suitable hypotheses these $\Gamma$-limits can be fully represented as the sum of three integral functionals as in \eqref{intro: reference problem}, including the term depending on the Cantor part. 

In Section \ref{sec:integrands in the limit}
we exploit the representation result of Section \ref{sec:repr} to give a necessary sufficient condition for the $\Gamma$-convergence of sequences of functionals of the form \eqref{intro: reference problem} with $f\in\mathcal{F}^\alpha$ and $g\in\mathcal{G}^\vartheta$. Finally, Section \ref{sec:homogenisation} is devoted to the study of the $\Gamma$-limit of functionals of type \eqref{eq:stochastic integrals}. As an application, these results are then employed in the final part of Section \ref{sec:homogenisation} to deal with stochastic homogenisation.
 
\section{Notation and preliminaries}\label{sec:notationpreliminaries}

In this Section we fix the notation and introduce the basic tools used in the rest of the paper.
\begin{enumerate}
    \item [(a)] Let $n\in\N$. The scalar product in $\mathbb{R}^n$ is denoted by $\cdot$ and the  Euclidean norm of $\mathbb{R}^n$ is denoted by $|\,\,|$. Given $x\in\mathbb{R}^n$,  the $i$-th component of $x$  is denoted by $x_i$. For every $\rho>0$ and $x\in\R^n$ the open ball of radius $\rho$ and center $x$ is denoted by $B_\rho(x)$.
    \item [(b)] We fix once and for all two positive integer numbers $d\geq 1$, $k\geq 1$. the unit spere in $\Rd$ is denote by $\mathbb{S}^{d-1}:=\{\nu\in\Rd: |\nu|=1\}$, . We also set $\mathbb{S}^{d-1}_{\pm}:=\{\nu\in\mathbb{S}^{d-1}: \pm \nu_{i(\nu)}>0\}$, where $i(\nu)\in\{1,...,d\}$ is the largest index such that $\nu^{i(\nu)}\neq 0$.
    \item [(c)]  Vectors in $\Rd$ are identified with $1\times d$ matrices, while $\R^{k\times d}$ is identified with the space of all $k\times d$ matrices. 
    For $\xi\in\Rkd$ and $x\in\Rd$ $\xi x\in\Rk$ is defined by the usual rules of matrix multiplication and  $\xi^i$ is the $i$-th row of $\xi$.
    Given a matrix $\xi=(\xi_{ij})\in\Rkd$, its Frobenius norm  is defined by
    $$|\xi|:=\Big(\sum_{i=1}^k\sum_{j=1}^d\xi_{ij}^2\Big)^{1/2}.$$

    \item [(d)]For $\rho>0$ we set $Q(\rho):=\{y\in\Rd: |y\cdot e _i|<\rho /2\}$, where $(e_i)_{i=1}^d$ is the canonical basis of $\Rd$. Given $x\in\Rd$, we set $Q(x,\rho):=x+Q(\rho)$.  
    \item [(e)]  For $n\in\N$  the space of all $n\times n$ orthonormal matrices $R$ with $\text{det}(R)=1$ is denoted by  $SO(n)$.  For every $\nu\in\mathbb{S}^{d-1}$ we fix once and for all an element $R_\nu\in SO(d)$ such that $R_{\nu}(e_d)=\nu$. We suppose that $R_{e_d}=I$, the identity matrix, that the restrictions of $\nu\to\mathbb{S}^{d-1}$ to $\mathbb{S}^{d-1}_{\pm}$ is continuous, and that $R_\nu(Q(\rho))=R_{-\nu}(Q(\rho))$ for every $\nu\in\mathbb{S}^{d-1}$ (see \cite[Example A.1]{CagnettiStocFree} for the proof of the existence of such $\nu\mapsto R_\nu$).

    \item [(f)] For $x\in\Rd$, $\nu\in\mathbb{S}^{d-1}$, $\lambda\geq 1$, and $\rho>0$, we consider the rectangle 
    \begin{equation*}
Q^\lambda_\nu(x,\rho):=x+R_\nu\Big(\Big(-\frac{\lambda\rho}{2},\frac{\lambda\rho}{2}\Big)^{d-1}\times\Big(-\frac{\rho}{2},\frac{\rho}{2}\Big)\Big);   
    \end{equation*}
     we omit the indication of $\lambda$ when $\lambda=1$.
    \item [(g)] Given an open set $\Omega\subset \Rd$, $\mathcal{A}(\Omega)$ (resp. $\mathcal{B}(\Omega)$) is the collection of all open (resp. Borel) sets $A\subset\Omega$. Given $A,B\in\mathcal{A}(\Omega)$, $A\subset\subset B$ means that $A$ is relatively compact in $B$. We set $\mathcal{A}_c(\Omega):=\{A\in\mathcal{A}(\Omega)\colon \,A\subset\subset\Omega\}$.

    \item [(h)]

For every $x\in\Rd$, $\xi\in\Rkd$, $\zeta\in\Rk$, and $\nu\in\mathbb{S}^{d-1}$  the two functions $\ell_\xi\colon \Rd \to\Rk$ and  $u_{x,\zeta,\nu}\colon \Rd\to\Rk$ are defined for every $y\in\Rd$ by
    \begin{eqnarray}
        &\displaystyle \nonumber \ell_\xi(y):=\xi y,\\
        &\displaystyle\nonumber u_{x,\zeta,\nu}(y):=\begin{cases}
            \zeta \quad \text{ if }(y-x)\cdot \nu\geq 0,\\
            0 \quad \text{ if }(y-x)\cdot \nu< 0.
        \end{cases}
    \end{eqnarray}
    Moreover, we set $\Pi^\nu_x:=\{x\in\Rd\colon (y-x)\cdot \nu=0\}$.

\item[(i)]Given $A\in\mathcal{A}(\Rd)$,  the space of $\R^n$-valued bounded Radon measures on $A$ is denoted by $\mathcal{M}_b(A;\mathbb{R}^n)$. If $n=1$ we omit the indication of $\R^n$. 
 If $\mu \in\mathcal{M}_b(A;\Rdk)$ and $\lambda\in \mathcal{M}_b(A)$ is non-negative, $d\mu/d\lambda$ denotes the Radon-Nikod\'ym derivative of $\mu$ with respect to $\lambda$.
The Lebesgue measure is denoted by $\Ld$ and the $(d-1)$-dimensional Hausdorff measure is denoted by $\hd$. 
For $\mu\in\mathcal{M}_b(A;\Rkd)$  the total variation $|\mu|$ is computed with respect to the Frobenius norm.
        
\item [(j)] For every $A\in\mathcal{A}(\Rd)$,  $L^0(A;\R^n)$ is the space of all $\Ld$-measurable  functions $u\colon A\to\Rk$ with the topology induced by the convergence in measure. We recall that such topology is metrisable and separable. When $n=1$ the indication $\R^n$ is omitted.

 \item [(k)] Given an $\Ld$-measurable set $E\subset \Omega$, a point $x\in \Omega$  such that 
 \begin{equation*}
     \limsup_{\rho\to 0^+}\frac{\Ld(E\cap B_\rho(x))}{\rho^d}>0,
 \end{equation*} and an $\Ld$-measurable function $u\colon E\to\Rk$, we say that $a\in \Rk$ is the approximate limit of $u$ at  $x$, in symbols 
\begin{equation*}
    \underset{\underset{y\in E}{y\to x}}{\text{ap}\lim \,}u(y)=a,
\end{equation*}
if for every $\varepsilon>0$ we have
\begin{equation*}
    \lim_{\rho\to 0^+}\frac{\Ld(\{|u-a|>\varepsilon\}   \cap B_\rho(x))}{\rho^d}=0,
\end{equation*}
where $\{|u-a|>\e\}:=\{y\in E\colon |u(y)-a|>\e\}$.

\item [(l)] Given $A\in\mathcal{A}(\Rd)$  and an $\Ld$-measurable function $u\colon A\to\Rk$,  the jump set $\jump{u}$ is the set of all points $x\in A$ such that 
there exists a triple $(u^+(x),u^-(x),\nu_u(x))\in\Rk\times\Rk\times\mathbb{S}^{d-1}$, with $u^+(x)\neq u^-(x)$, such that, setting 
\begin{equation*}
    H^+:=\{y\in A\colon \, (y-x)\cdot \nu_u(x)>0\}\,\, \text{ and }\,\,H^-:=\{y\in A\colon \, (y-x)\cdot \nu_u(x)<0\},
\end{equation*}
we have 
\begin{equation*}
      \underset{\underset{y\in H^+}{y\to x}}{\text{ap}\lim \,}u(y)=u^+(x)\text{ \,\,\, and }  \,\,\,\underset{\underset{y\in H^-}{y\to x}}{\text{ap}\lim \,}u(y)=u^-(x).
\end{equation*}
The triple $(u^+(x),u^-(x),\nu_u(x))$ is well-defined up to interchanging the roles of $u^+(x)$ and $u^-(x)$ and swapping  the sign of $\nu_u(x)$.

\item [(m)] Given $A\in\mathcal{A}(\Rd)$ the symbol $BV(A;\Rk)$ denotes the space of $\Rk$-valued functions with bounded variation on $A$. We refer the reader to \cite{AmbFuscPall} (see also \cite{BraidesApprox, Evans2015,giusti1984minimal}) for an exhaustive introduction to this function space. 
We recall that if $u\in BV(A; \Rk)$ then for $\hd$-a.e. $x\in A\setminus\jump{u}$ there exists 
\[ \Tilde{u}(x):=\underset{\underset{y\in A}{y\to x}}{\text{ap}\lim \,}u(y).\]
 We recall also that if $u\in BV(A;\Rk)$, then  $\jump{u}$ is a $(d-1)$-countably rectifiable set, and for $\hd$-a.e. $x\in\jump{u}$ the vector $\nu_u(x)$ is a measure theoretical normal to $\jump{u}$.  For every $x\in\jump{u}$  we set 
\[[u](x):=u^+(x)-u^-(x).\]
 A change of sign of $\nu_u(x)$ obviously implies a change of sign in $[u](x)$. 
 
 \item [(n)]Given $u\in BV(A;\Rk)$, its distributional derivative  $Du$, which is by definition a bounded  $\Rkd$-valued Radon measure, can be decomposed as

\[Du=\nabla u\Ld +D^cu+[u]\otimes\nu_u\hd\mres\jump{u},\]
where 
\begin{itemize}
    \item $\nabla u\in L^1(A;\Rdk)$ is the approximate gradient of $u$, that is, the only $\Rkd$-valued function such that for $\Ld$-a.e $x\in A$ we have 
\begin{equation}\label{eq:Approximate Differentaibility}
    \underset{\underset{y\in A}{y\to x}}{\text{ap}\lim \,}\frac{u(y)-\Tilde{u}(x)-\nabla u(x)(y-x)}{|y-x|}=0,
\end{equation}
\item $D^cu$, called the Cantor part of $Du$, is a measure singular with respect to $\Ld$ and vanishing on all Borel sets $B\in\mathcal{B}(A)$ with $\hd(B)<+\infty$,
    \item $\otimes$ denotes the tensor product defined by $(a\otimes b)_{ij}=a_ib_j$ for $a\in\Rk$, $b\in\Rd$, $\hd\mres\jump{u}$ is the Borel measure on $A$ defined by $\hd(B):=\hd(B\cap \jump{u})$ for every $B\in\mathcal{B}(A)$, and $[u]\otimes\nu_u\hd\mres\jump{u}$ denotes the measure with density $[u]\otimes\nu_u$ with respect to $\hd\mres\jump{u}$.  
\end{itemize}

\end{enumerate}

\bigskip

We briefly recall the definition and the main properties of the space $GBV_\star(A;\Rk)$, introduced in the scalar setting by the first author and Toader in \cite{DalToa22} and in the vectorial setting by the second author in \cite{donati2023new}. For $s,t\in\R$ we set $s\land t=\min\{s,t\}$ and $s\lor t=\max\{s,t\}$. For every $t\in\R$, $a\in\Rk$ and $m>0$, we set  $t^{(m)}:=(t\land m)\lor (-m)$ and $a^{(m)}:=(a_1^{(m)},...,a_k^{(m)})$. In the rest of the section, $A\subset\Rd$ will always be a bounded open set.
\begin{definition}\label{def: GBVstar}
   Let $u\in L^0(A;\Rk)$. Then $u\in GBV_\star(A;\Rk)$ if and only $u^{(m)}\in BV(A;\Rk)$ for every $m>0$ and there exists $M>0$ such that 
   \begin{equation}\label{eq:boundsinglecomponent}
\sup_{m>0}\int_A|\nabla u^{(m)}|\, dx+|D^cu^{(m)}|(A)+\int_{\jump{u^{(m)}}}|[u^{(m)}]|\land 1\,d\mathcal{H}^{d-1}\leq M.
\end{equation}
$GBV_\star(A)$ is defined similarly for scalar functions.
\end{definition}

\begin{remark}\label{re:def componentwise gbv}
    It follows immediately from the definition that $u$ belongs to $GBV_\star(A;\Rk)$ if and only if each component $u_i$ belongs to  $GBV_\star(A)$. By \cite[Theorem 3.9]{DalToa23} this implies   that $GBV_\star(A;\Rk)$ is a vector space.
\end{remark}

To characterise $GBV_\star(A;\Rk)$ by means of smooth truncations we introduce the following functions.

Given a positive constant $\sigma>2$, we fix a  smooth radial function $\psi\in C^\infty_c(\Rk;\Rk)$ satisfying
\begin{equation}\label{eq:defpsi}
    \begin{cases}
        \psi(y)=y&\text{ if }|y|\leq 1,\\
        \psi(y)=0&\text{ if }|y|\geq \sigma,\\
        |\psi(y)|\leq \sigma,\\
        \textup{Lip}(\psi)=1.
    \end{cases}
\end{equation}

\noindent It is not difficult to construct such a function  (see for instance \cite[Section 4]{CagnettiDetFree}).

Given $R>0$, we set
\begin{equation}\label{def:psiR}\psi_{R}(y):=R \psi\Big(\frac{y}{R}\Big)\quad \text{ for every $y\in\Rk$}.
\end{equation}
Note that $\psi_R$ satisfies
\begin{equation}\label{eq:proppsiR1}
    \begin{cases}
        \psi_R(y)=y&\text{ if }|y|\leq R,\\
        \psi_R(y)=0&\text{ if }|y|\geq \sigma R,\\
        |\psi_R(y)|\leq \sigma R,\\
        \textup{Lip}(\psi_R)=1.
    \end{cases}
\end{equation}
The following proposition characterises $GBV_\star(A;\Rk)$ in terms of smooth truncations. 
\begin{proposition}[{\cite[Proposition 3.8]{donati2023new}}]\label{prop:characterisation GBVstar}
   For every  $u\in GBV_\star(A;\Rk)$  there exists a constant $C_u>0$ such that for every Lipschitz function $\phi$ with compact support the function  $v:=\phi\circ u$ belongs to $BV(A;\Rk)$ and satisfies the inequality
\begin{equation}\label{eq:def GBVstar}
    \int_A|\nabla v|\, dx+|D^cv|(A)+\int_{J_{v}}|[v]|\land 1\,d\mathcal{H}^{d-1}\leq C_u(\textup{Lip}(\phi)\lor 1).
\end{equation}
Conversely, if $u\in L^0(A;\Rk)$ and there exists a constant $C_u>0$ such that for every integer $R>0$ \eqref{eq:def GBVstar} holds with $v=\psi_R\circ u$ and $\phi=\psi_R$ then $u\in GBV_\star (A;\Rk)$.
\end{proposition}
In the following proposition we recall the fine properties of functions in $GBV_\star(A;\Rk)$. 

\begin{proposition}\label{prop: Properties of GBVsvector}
    Let $u\in GBV_\star(A;\Rk)$. Then
    \begin{itemize}
        \item [(a)] for $\hd$-a.e $x\in A\setminus\jump{u}$ there exists 
        \[\Tilde{u}(x):=\underset{{y\to x}}{\textup{ap}\lim \,}u(y);\]
        \item [(b)] there exists a Borel function $\nabla u\in L^1(A;\Rdk)$ such that for $\Ld$-a.e. $x\in A $ formula \eqref{eq:Approximate Differentaibility} holds true; moreover, for every $R>0$  we have
        \begin{equation*} 
         \nabla u(x)=\nabla (\psi_R\circ u)(x) \quad \text{for $\Ld$-a.e. }x\in \{|u|\leq R\};
        \end{equation*}
        \item [(c)] 
        there exists a unique Radon measure $D^cu\in\mathcal{M}_b(A;\Rkd)$ such that for every $R>0$  we have $D^cu(B)=0$ for every $B\in\mathcal{B}(A)$ with $\hd(B\setminus\jump{u})=0$ and  $D^cu(B)=D^c(\psi_R\circ u)(B)$ for every $B\subset\{x\in A\setminus\jump{u}\colon \, \widetilde{u}(x) \text{ exists and }|\widetilde{u}(x)|\leq R\}$; moreover,  for every $B\in\mathcal{B}(A)$ we have 
        \begin{eqnarray}
        &\displaystyle\label{eq:limite cantor}
       D^cu(B)=\lim_{R\to +\infty} D^c(\psi_{R}\circ u)(B),\\\label{eq:limite cantor modulo}
       &\displaystyle |D^cu|(B)=\lim_{R\to +\infty} |D^c(\psi_{R}\circ u)|(B);
        \end{eqnarray}
        
        \item [(d)] for every  $R>0$  we have $\jump{\psi_R\circ u}\subset \jump{u}$ up to an $\hd$-negligible set  and $|[\psi_R\circ u]|\leq|[u]|$ on $\jump{\psi_R\circ u}\cap\jump{u}$. Moreover, for $\hd$-a.e $x$ in $\jump{u}$ and every $R>|u^+(x)|\lor|u^-(x)|$  we have $|[\psi_R\circ u](x)|=|[u](x)|$.
          \end{itemize}
\end{proposition}
For the proof of these facts we refer the reader to \cite[Proposition 3.7]{donati2023new}.

We conclude this section recalling some  useful facts related to the Cantor part of compositions with smooth function with compact support.
\begin{proposition}\label{prop:Properties of derivatives of GBVstar}
    Let $A\in\mathcal{A}_c(\Rd)$, $u\in GBV_\star(A;\Rk)$ and $\phi\in C^1_c(\Rk;\Rk)$. Then 
    \begin{itemize}
        \item[(i)] $\nabla(\phi\circ u)=\nabla \phi(\widetilde{u})\nabla u \quad \text{$\Ld$-a.e. in $A$};$
        \item[(ii)]  $D^c(\phi \circ u)=\nabla \phi(\widetilde{u})D^cu$ as Radon measures on $A$;

        \item [(iii)]  for every $R>0$ we have 
        \begin{equation*}
            \frac{dD^c(\psi_{R}\circ u)}{d|D^c(\psi_{R}\circ u)|}=\frac{dD^cu}{d|D^cu|} \quad \text{$|D^cu|$-a.e. in $\ A^R_{u,0}$},
        \end{equation*} 
          where $A^R_{u,0}:=\{x\in A\colon \widetilde{u}(x)\text{ exists and } |\widetilde{u}(x)|\leq R\}$.
        As a consequence we have
        \begin{equation*}
           \lim_{R\to+\infty} \frac{dD^c(\psi_{R}\circ u)}{d|D^c(\psi_{R}\circ u)|}=\frac{dD^cu}{d|D^cu|} \quad\text{ $|D^cu|$-a.e. in $ A$}.
        \end{equation*}
        
    \end{itemize}
\end{proposition}
For the proof we refer the reader to \cite{donati2023new}.
\
In accordance with the notation of \cite{DalToa23,DalToa23b} and \cite{donati2023new}, we introduce two functionals on $L^0(\Rd;\Rk)$ closely related to the space $GBV_\star(A;\Rk)$.

\begin{definition}\label{def:functional V}
    Given $u\in L^0(\Rd;\Rk)$, with components $u_1,...,u_k$, for every  $A\in\mathcal{A}_c(\Rd)$ we define 
 \begin{equation*}
  V(u,A):=\sum_{i=1}^k\Big(\int_A|\nabla u_i|\, dx+|D^cu_i|(A)+\int_{\jump{u_i}\cap A}|[u_i]|\land 1\,d \hd\Big),\\ 
 \end{equation*}
  if $u|_A \in GBV_\star(A)$ and we set $V=+\infty$ otherwise. 
 The  definition is then extended  to $A\in\mathcal{A}(\Rd)$ by setting
  \begin{equation*}
      V(u,A):=\sup\{V(u,A'):\,\, A'\subset A, \,\,A'\in\mathcal{A}_c(\Rd)\}
  \end{equation*}
  and then to $B\in\mathcal{B}(\Rd)$ by setting
  \begin{equation*}
     V(u,B):=\inf\{V(u,A):\,\, B\subset A, \,\, A\in\mathcal{A}(\Rd)\}.
  \end{equation*}

\end{definition}
\begin{definition}\label{def:functional Vfrob}
 Given $u\in L^0(\Rd;\Rk)$, for every $A\in\mathcal{A}_c(\Rd)$ we define 
    \begin{equation} \nonumber V_2(u,A):=\int_A|\nabla u|\, dx+|D^cu|(A)+\int_{\jump{u}\cap A}|[u]|\land 1\,d \hd,
    \end{equation}
    if $u\in GBV_\star(A;\Rk)$ and we set $V_2(u,A)=+\infty$ otherwise. The definition is then extended to every Borel set as in Definition \ref{def:functional V}.
\end{definition}

\begin{remark}\label{re:domainForV}
    Let $A\in\mathcal
A_c(\Rd)$ and let $u\in L^0(\Rd;\Rk)$. It follows immediately from Remark \ref{re:def componentwise gbv} that
\begin{equation*}
   u|_A\in GBV_\star(A;\Rk) \text{ \,\,if and only if\,\, } V(u,A)<+\infty.
\end{equation*}
\end{remark}

\begin{remark}\label{re:lowersemicontinuity of Veta}
    It follows from \cite[Theorem 2.1]{Braides1995} and \cite[Theorem 3.11]{DalToa22}  that the functional 
    $V$ of Definition \ref{def:functional V} is lower semicontinuous with respect to the topology of $L^0(\Rd;\Rk)$. We don't know whether this property holds for the functional $V_2$ of Definition \ref{def:functional Vfrob}.
\end{remark}

\section{Volume and Surface integrands}\label{sec:classE}

Throughout the rest of the paper we fix  six constants $c_1,...,c_6\geq 0$  and a bounded continuous function $\tau\colon [0,+\infty)\to[0,+\infty)$. We assume that 
\begin{align}
    &\label{eq:hp on constants}\displaystyle0<c_1\leq 1\leq c_3\leq c_5,\,\,\,\, c_6\geq (c_3/c_1)k^{3/2},
    \\
    &\displaystyle \nonumber \tau(0)=0\,\,\,\,\text{ and\,\,\,\, }\tau(t)\geq c_3(t\land1)\,\,\,\text{ for every $t\geq 0$}.
\end{align}

We now introduce the collection of volume integrands considered in this paper.

\begin{definition}\label{def:good bulk integrands}
    Let  $\mathcal{F}$ be the space of  functions $f\colon\Rd\times\Rdk\to [0,+\infty)$ satisfying the following conditions:
    \begin{itemize}
        \item [(f1)] $f$ is Borel measurable;
        \item [(f2)]  $c_1\sum_{i=1}^k|\xi_i|-c_2\leq f(x,\xi)$ \text{ for every }$x\in\Rd$ \text{and }$\xi\in \Rdk$;
        \item  [(f3)] $f(x,\xi)\leq c_3\sum_{i=1}^k|\xi_i|+c_4$  \text{ for every }$x\in\Rd$ \text{and  }$\xi\in \Rdk$;
        \item [(f4)] $|f(x,\xi)-f(x,\xi)|\leq c_5|\xi-\eta|$  \text{ for every }$x\in\Rd$ \text{and  }$\xi,\eta\in \Rdk$.
    \end{itemize}
\end{definition}
\begin{remark}
    It follows from the inequalities $|\xi|\leq\sum_{i=1}^k|\xi_i|\leq k^{1/2}|\xi|$ that
    \begin{itemize}
     \item[(f2$'$)]  $c_1|\xi|-c_2\leq f(x,\xi)$ \text{ for every }$x\in\Rd$ \text{ and  }$\xi\in \Rdk,$
     \item[(f3$'$)] $f(x,\xi)\leq c_3k^{1/2}|\xi|+c_4$ \text{ for every }$x\in\Rd$ \text{ and  }$\xi\in \Rdk$.
    \end{itemize}
\end{remark}

The following definition introduces the collection of surface integrands considered in this paper.
\begin{definition}\label{def:good surface integrands}
    Let $\mathcal{G}$ be the space of functions $g:\Rd\times\Rk\times\mathbb{S}^{d-1}\to[0,+\infty) $ that satisfy the following conditions:
    \begin{itemize}
        \item [(g1)] $g$ is Borel measurable;
        \item [(g2)] $c_1\sum_{i=1}^k(|\zeta_i|\land 1)\leq g(x,\zeta,\nu)$ \text{ for every }$x\in\Rd$, $\zeta\in \Rk$, and $\nu\in\mathbb{S}^{d-1}$;
        \item  [(g3)] $g(x,\zeta,\nu)\leq c_3\sum_{i=1}^k(|\zeta_i|\land 1)$  \text{ for every }$x\in\Rd$, $\zeta\in \Rk$, and $\nu\in\mathbb{S}^{d-1}$;
        \item [(g4)] $|g(x,\zeta,\nu)-g(x,\theta,\nu)|\leq \tau(|\zeta-\theta|)$  \text{ for every }$x\in\Rd$, $\zeta,\theta\in \Rk$, and $\nu\in\mathbb{S}^{d-1}$;
        \item [(g5)] $g(x,-\zeta,-\nu)=g(x,\zeta,\nu)$ \text{ for every }$x\in\Rd$, $\zeta\in \Rk$, and $\nu\in\mathbb{S}^{d-1}$;
        \item [(g6)] for every $\zeta,\theta\in\Rk$ with $c_6k|\zeta|\leq |\theta|$ we have $g(x,\zeta,\nu)\leq g(x,\theta,\nu)$ for every $x\in\Rd$ and  $\nu\in\mathbb{S}^{d-1}$.
    \end{itemize}
\end{definition}
\begin{remark}
     A variant of property (g6) was already used in \cite{CagnettiDetFree}. Combining (g2) and (g3), it is easy to show that $g(x,\zeta,\nu)\leq c_3/c_1g(x,\theta,\nu)$ whenever $\sum_{i=1}^k|\zeta_i|\leq\sum_{i=1}^k|\theta_i|$.  Arguing as in \cite[Remark 3.2]{CagnettiDetFree} we can show that this  property and (g6) are weaker than monotonicity in $|\zeta|$ of $g$.
\end{remark}

\begin{definition}
    For every $f\colon\Rd\times\Rdk\rightarrow[0,+\infty)$, the recession function $f^\infty\colon\Rd\times\Rdk\rightarrow[0,+\infty]$ (with respect to $\xi$) is the function defined by
    \begin{equation}\label{eq:defrecession}
        f^\infty(x,\xi):=\limsup_{t\to +\infty}\frac{f(x,t\xi) }{t}
    \end{equation}
   for every $x\in\Rd$ and for every $\xi\in\Rdk$.
\end{definition}

\begin{remark}\label{re:bounds finfty}
    For every $x\in\Rd$ the function $\xi\mapsto f^\infty(x,\xi)$ is positively $1$-homogeneous. If for every $x\in\Rd$ the function $\xi\mapsto f(x,\xi)$ is convex, the  $\limsup$ in \eqref{eq:defrecession} is  a limit. If $f$ satisfies (f2) and (f3), then $f^\infty$ satisfies 
    \begin{equation}\label{eq:Finfty bounds}
         c_1|\xi|\leq c_1\sum_{i=1}^k|\xi_i|\leq f^\infty(x,\xi)\leq c_3k^{1/2}|\xi|\quad \text{for every $x\in\Rd$ and  $\xi\in\Rdk$}.
    \end{equation}
If $f$ satisfies (f4), then $f^\infty$ satisfies 
\begin{equation*}\label{eq:Finfty Lipschitz}
|f^\infty(x,\xi)-f^\infty(x,\eta)|\leq c_5|\xi-\eta|\quad\text{for every  $x\in\Rd$ and  $\xi,\eta\in\Rdk$}.
\end{equation*}
\end{remark}

 In Section \ref{sec:integrands in the limit} we will consider also integrands $g^0$ that do not belong to $\mathcal{G}$. For this reason, in the following definition we do not assume $g\in\mathcal{G}$.
\begin{definition}\label{def:Functionals Efg}
 Given $f\in\mathcal{F}$ and  a Borel function $g\colon \Rd\times \Rk\times \mathbb{S}^{d-1}\to[0,+\infty)$ satisfying (g5), the functional  $E^{f,g}:L^0(\Rd;\Rk)\times\mathcal{B}(\Rd)\to[0,+\infty]$ is  the functional defined for  $A\in\mathcal{A}_c(\Rd)$ 
 by
 \begin{equation*}\label{eq:Def Efg}
     E^{f,g}(u,A):=\int_Af(x,\nabla u)\, dx+\int_Af^\infty\Big(x,\frac{D^cu}{|D^cu|}\Big)\,d|D^cu|+\int_{\jump{u}\cap A}g(x,[u],\nu_u)\, \hd,
 \end{equation*}
  if $u|_A\in GBV_\star(A;\Rk)$, and by $E^{f,g}(u,A)=+\infty$ otherwise. 
 The definition is then extended to $A\in\mathcal{A}(\Rd)$ by setting
  \begin{equation*}
      E^{f,g}(u,A):=\sup\{E^{f,g}(u,A'):\,\, A'\subset A, \,\,A'\in\mathcal{A}_c(\Rd)\},
  \end{equation*}
  and then to $B\in\mathcal{B}(\Rd)$ by setting
  \begin{equation*}
      E^{f,g}(u,B):=\inf\{E^{f,g}(u,A):\,\, B\subset A, \,\, A\in\mathcal{A}(\Rd)\}.
  \end{equation*}
\end{definition}

To study the integral representation of $\Gamma$-limits of sequences of functionals of the form $E^{f,g}$ with $f\in\mathcal{F}$ and $g\in\mathcal{G}$, it is convenient to study the properties of functionals $E^{f,g}$ that pass to the $\Gamma$-limit. This leads to define  the abstract space $\E$ of functionals defined on $L^0(\Rd;\Rk)\times \mathcal{B}(\Rd)$. We shall prove that the $\Gamma$-limit of a sequence of functionals in $\E$ belongs to $\E$ (see Theorem \ref{thm:compactness for E}) and that for every lower semicontinuous functional in $\E$  the volume and surface part admit and integral representation of the form 
\begin{equation}\nonumber 
    \int_Af(x,\nabla u)\, dx\quad \text{ and } \quad \int_{J_u} g(x,[u],\nu_u)\, \hd
\end{equation}
with 
integrands $f\in\mathcal{F}$ and $g\in\mathcal{G}$ (see Theorem \ref{thm:rappresentazione}).

 For technical reasons, in the definition of the space $\E$, we consider the behaviour of the functionals with respect to some smooth truncations. To this aim,  for every $R>0$ and every integer $i\geq 1$ we consider 
   the function $\psi_R^i$ defined by \eqref{def:psiR} with $R$ replaced by $\sigma^{i-1}R$ and $\sigma=c_6k+2$. Note that by \eqref{eq:proppsiR1}  for every integer $i\geq 1$ we have  
\begin{equation}\label{eq:properties psiRi}
    \begin{cases}
        \psi^i_R(y)=y&\text{ if }|y|\leq \sigma^{i-1}R,\\
        \psi^i_R(y)=0&\text{ if }|y|\geq \sigma^i R,\\
        |\psi^i_R(y)|\leq \sigma^i R,\\
        \textup{Lip}(\psi_R)=1.
    \end{cases}
\end{equation}

We are now in a position to introduce the space of abstract functionals $\E$.
\begin{definition}\label{def:space of functionals E}
   Let  $\E$ be the space of functionals $E\colon L^0(\Rd;\Rk)\times\mathcal{B}(\Rd)\to [0,+\infty]$ satisfying  the following conditions:
\begin{itemize}
  \item[(a)] for every $A\in\mathcal{A}(\Rd)$ and for every $u,v\in L^0(\Rd;\Rk)$ with $u=v$ $\Ld$-a.e. in $A$, we have $E(u,A)=E(v,A)$;
\item [(b)] for every $u\in L^0(\Rd;\Rk)$ the set function $B\mapsto E(u,B)$ is a non-negative Borel measure on $\Rd$ and for every $B\in\mathcal{B}(\Rd)$ we have
\begin{equation*}
    E(u,B)=\inf\{E(u,A)\colon  A\in\mathcal{A}(\Rd)\text{ and }  B\subset A \};
\end{equation*}
    \item[(c1)] for every $u\in L^0(\Rd;\Rk)$, and  $B\in\mathcal{B}(\Rd)$ we have
    \begin{equation}\nonumber 
        c_1V(u,B)-c_2\Ld(B)\leq E(u,B);
    \end{equation}
    \item[(c2)] for every $u\in L^0(\Rd;\Rk)$ and  $B\in\mathcal{B}(\Rd)$ we have
     \begin{equation}\nonumber 
        E(u,B)\leq c_3V(u,B)+c_4\Ld(B);
    \end{equation}
    \item[(d)] for every $u\in L^0(\Rd;\Rk)$, $B\in\mathcal{B}(\Rd)$, and $a\in\Rk$ we have
    \[E(u+a,B)=E(u,B); \]
    \item [(e)] for every $u\in L^0(\Rd;\Rk)$, $\xi\in\Rkd$, $B\in \mathcal{B}(\Rd)$  we have
    \begin{equation*}
        E(u+\ell_\xi,B)\leq E(u,B)+ c_5|\xi|\Ld(B);  \end{equation*}
    \item[(f)]for every $u\in L^0(\Rd;\Rk)$, $B\in\mathcal{B}(\Rd)$, $x\in\Rd $, $\zeta\in\Rk$, and $\nu\in\mathbb{S}^{d-1}$ we have
    \begin{equation*}
        E(u+u_{x,\zeta,\nu},B)\leq E(u,B)+\tau(|\zeta|)\hd(B\cap\Pi_x^\nu);
    \end{equation*}
    \item[(g)] for every $m\in\N$, $u\in L^0(\Rd;\Rk)$, $B\in\mathcal{B}(\Rd)$, $w\in W^{1,1}_{\rm loc}(\Rd;\Rk)$, and $R>0$ we have
    \begin{align}
        \nonumber \frac{1}{m}\sum_{i=1}^mE(w+\psi^i_R\circ (u-w),B)& \leq E(u,B) + c_3k^{1/2}\int_{B^R_{u,w}}|\nabla w|\, dx+c_4\Ld(B^R_{u,w}) \\
      & \label{eq: g} \quad +\frac{C}{m}\Big(E(u,B) + \Ld(B) + \int_B|\nabla w|\, dx \Big),
        \end{align}
       where $B^R_{u,w}:=\{x\in B\colon |u(x)-w(x)|\geq R\}$ and  $C:=\max\{9c_3k/c_1,2c_3k^{1/2},c_4\}$;
    \item [(h)] for every $0<\lambda\leq 1/(c_6k)$, $R\in SO(k)$, $B\in\mathcal{B}(\Rd)$, and $u\in L^0(\Rd;\Rk)$ we have
    \begin{equation}\nonumber
        E(\lambda Ru,B)\leq E(u,B)+(c_4+c_2)\Ld(B).
    \end{equation}
\end{itemize}
The subspace of all functionals $E\in\E$ such that for every $A\in\mathcal{A}(\Rd)$ the functional $E(\cdot,A)$ is lower semicontinuous with respect to the topology  of $L^0(\Rd;\Rk)$ is denoted by $\E_{\rm sc}$.
\end{definition}

\begin{remark}\label{ref:locality}
 Let $E\in\E$, $A\in\mathcal{A}(\Rd)$, and $u\in L^0(A;\Rk)$.  We can define $E(u,B)$ for every $B\in\mathcal{B}(A)$ by extending the function $u$ to a function $v\in L^0(\Rd;\Rk)$ and setting $E(u,B):=E(v,B)$. The value of $E(u,B)$ does not depend on the chosen extension thanks to the locality property (a) of Definition \ref{def:space of functionals E}.
\end{remark}

\begin{remark}\label{re:new c1 and c2}
    The inequalities (c1) and (c2) imply that
    \begin{equation}
         \tag{c1$'$} c_1V_2(u,B)-c_2\Ld(B)\leq E(u,B)
        \end{equation}
        \begin{equation}
        \tag{c2$'$}E(u,B)\leq c_3kV_2(u,B)+c_4,
    \end{equation}
    for every $u\in L^0(\Rd;\Rk)$ and $B\in\mathcal{B}(\Rd)$. 
   This follows from the elementary inequalities $|\xi|\leq\sum_{i=1}^k|\xi_i|\leq k^{1/2}|\xi|$ and $|\zeta|\land 1\leq \sum_{i=1}^k|[\zeta_i]|\land 1\leq k(|\zeta|\land 1)$, where $\xi_i$ is the $i$-th row of the matrix $\xi$ and $\zeta_i $ is the $i$-th component of $\zeta$.
\end{remark}

\begin{remark}
The technical condition (g) replaces the simpler condition (g) of \cite[Definition 3.1]{DalToa23}, which in the particular case $w_1=-m$ and $w_2=m$ reads as
\begin{equation}\label{eq:estimate for truncations in dal maso and toader}
    E(u^{(m)},A)\leq E(u,A)+\Ld(\{|u|>m\}).
\end{equation}
 Although this truncation procedure is still available in the vector-valued case  (see \cite[Definition 3.1, Proposition 3.8]{donati2023new}), acting componentwise, an estimate of the form \eqref{eq:estimate for truncations in dal maso and toader}
does not hold for $k>1$, even for the prototypical integral functional $V_2$. This is one of the main difficulties in generalising these results to the vector-valued case.  The new condition (g) of Definition \ref{def:space of functionals E} is crucial to introduce a different type of truncation operators with good estimates. Similar smooth truncations were already considered in  \cite{BraDefVit,CagnettiDetFree,  CeladaDalM, FonsecaMuller,FuscoHutchinson}.
\end{remark}

\begin{remark}\label{re:Domain of function spaces}
    It follows directly from  Remark \ref{re:domainForV} and from Definition \ref{def:space of functionals E}   that if $A\in\mathcal{A}_c(\Rd)$ and $u\in L^0(\Rd;\Rk)$ then
    \[
    E(u,A)<+\infty\,\,\,\text{ if and only if }\,\,\, u|_A\in GBV_\star(A;\Rk).
    \]
\end{remark}

The following proposition shows that functionals of type $E^{f,g}$ belong to $\E$.

\begin{proposition}\label{prop: integral functionals are in E}
    Let $f\in\mathcal{F}$ and $g\in\mathcal{G}$. Then the functional $E^{f,g}$ belongs to $\E$ and for every $A\in\mathcal{A}_c(\Rd)$. Moreover, $u\in GBV_\star(A;\Rk)$ we have
    \begin{equation}\label{eq: identity Efg sui borel}
        E^{f,g}(u,B)=\int_Bf(x,\nabla u)\, dx+\int_Bf^\infty\Big(x,\frac{D^cu}{|D^cu|}\Big)\,d|D^cu|+\int_{\jump{u}\cap B}g(x,[u],\nu_u)\, \hd
    \end{equation}
    for every $B\in\mathcal{B}(A)$.
\end{proposition}

\begin{proof}
   The proof of the fact that $E^{f,g}$ satisfies \eqref{eq: identity Efg sui borel} and (a)-(f) of Definition \ref{def:space of functionals E}  may be deduced  from the proof of  \cite[Proposition 3.11]{DalToa23}, with minor changes.
   
   We are left with showing that $E^{f,g}$ enjoys properties (g) and (h) of Definition \ref{def:space of functionals E}. We begin proving (g). Since $E^{f,g}$ satisfies properties (a) and (b), it is enough to prove \eqref{eq: g} for every  $A\in\mathcal{A}_c(\Rd)$.  Without loss of generality, we may assume that $u|_A\in GBV_\star(A;\Rk)$. Let $w\in W^{1,1}_{\rm loc}(\Rd;\Rk)$, let  $R>0$, let $m\in\N$, and let $i\in\{1,...,m\}$. We set $v^i_R:=w+\psi^i_R\circ (u-w)$ and note that $v_R^i|_A\in GBV_\star(A;\Rk)$, thanks to Proposition \ref{prop:characterisation GBVstar} and to the fact that $GBV_\star(A;\Rk)$ is a vector space.
   
   Consider the set $A_{\rm reg}:=\{x\in A\colon \widetilde{u}(x),\,\widetilde{w}(x)\,\,\, \text{exist}\}$. We observe that by Proposition \ref{prop: Properties of GBVsvector}  $A_{\rm reg}$ and $\jump{u}$ are Borel sets and that $\hd(A\setminus (A_{\rm reg}\cup \jump{u}))=0$. We set
   \begin{eqnarray*}
     &&\displaystyle A_{\rm in}^i:=\{x\in A_{\rm reg}\colon|\widetilde{u}(x)-\widetilde{w}(x)|\leq \sigma^{i-1}R\},\\
     &&\displaystyle A^i:=\{x\in A_{\rm reg}\colon\sigma^{i-1}R<|\widetilde{u}(x)-\widetilde{w}(x)|< \sigma^{i}R\},\\
     &&\displaystyle A^i_{\rm out}:=\{x\in A_{\rm reg}\colon|\widetilde{u}(x)-\widetilde{w}(x)|\geq \sigma^{i}R\}.
   \end{eqnarray*}
   Thanks to (i) of Proposition \ref{prop:Properties of derivatives of GBVstar}, for every $i\in\{1,...,m\}$ we have that 
   \begin{eqnarray}
       &\nonumber \displaystyle \nabla v^i_R= \nabla u  \quad \Ld\meno \text{a.e. in } A^i_{\rm in}, \\
       &\displaystyle \label{eq:nabla sulla strip}\nabla v^i_R=\nabla w +\nabla\psi^i_R\circ(u-w)(\nabla u-\nabla w)\quad \Ld\meno \text{a.e. in } A^i,\\
       &\nonumber \displaystyle \nabla v^i_R=\nabla w \quad \Ld\meno \text{ a.e. in } A_{\rm out}^i.
   \end{eqnarray}
   Therefore,
   \begin{equation}\label{eq:split on AC concrete}
       \int_Af(x,\nabla v^i_R)\,dx=\int_{A^i_{\rm in}}f(x,\nabla u)\,dx+\int_{A^i_{\rm str}}f(x,\nabla v^i_R)\,dx+\int_{A^i_{\rm out}}f(x,\nabla w)\,dx.
   \end{equation}
  
   \noindent Exploiting (f2$'$), (f3$'$), and \eqref{eq:nabla sulla strip}, and recalling that $\text{Lip}(\psi_R^i)= 1$,  we obtain
   \begin{eqnarray}
\nonumber\int_{A^i}f(x,\nabla v^i_R)\,dx&\leq& 2c_3k^{1/2}\int_{A^i}|\nabla w|\, dx+c_3k^{1/2}\int_{A^i}|\nabla u|\, dx+c_4\Ld(A^i)\\
        &\leq& \nonumber\label{eq:stimestripac}2c_3k^{1/2}\int_{A^i}|\nabla w|\, dx+\frac{c_3k^{1/2}}{c_1}\int_{A^i}f(x,\nabla u)\,dx+c_4\Ld(A^i),\\\nonumber
        \int_{A^i_{\rm out}}f(x,\nabla w)\,dx&\leq&c_3k^{1/2}\int_{A^R_{u,w}}|\nabla w|\, dx+c_4\Ld(A_{u,w}^R).
        \end{eqnarray}
        These inequalites and \eqref{eq:split on AC concrete} lead to 
 \begin{eqnarray}
   \nonumber \sum_{i=1}^m\int_{A}f(x,\nabla v^i_R)\, dx&\leq& m\Big(\int_A f(x,\nabla u)\, dx +c_3k^{1/2}\int_{A_{u,w}^R}|\nabla w|\, dx+c_4\Ld(A_{u,w}^R)\Big)\\
     &&+\label{eq:contofatto su inner}\frac{c_3k^{1/2}}{c_1}\int_{A}f(x,\nabla u)\, dx+
     2c_3\int_{A}|\nabla w|\, dx+c_4\Ld(A).
 \end{eqnarray}
 To estimate the term depending on $D^cu$ in $E^{f,g}$, we use Proposition \ref{prop:Properties of derivatives of GBVstar}(ii) to rewrite
   \begin{multline}\nonumber
       \int_{A}f^\infty\Big(x,\frac{D^cv^i_R}{|D^cv^i_R|}\Big)\,d|D^cv^i_R|=\int_{A_{\rm in}^i}f^\infty\Big(x,\frac{dD^cu}{d|D^cu|}\Big)\,d|D^cu|\\+\int_{A^i}f^\infty\Big(x,\frac{\nabla\psi_R^i(\widetilde{u}-\widetilde{w})}{|\nabla\psi_R^i(\widetilde{u}-\widetilde{w})|}\frac{dD^cu}{d|D^cu|}\Big)\,|\nabla\psi_R^i(\widetilde{u}-\widetilde{w})| d|D^cu|+\int_{A^i_{\rm out}}f^\infty(x,0)\,d|D^cu|\\
       \leq \int_{A}f^\infty\Big(x,\frac{dD^cu}{d|D^cu|}\Big)\,d|D^cu|+\int_{A^i}f^\infty\Big(x,\frac{\nabla\psi_R^i(\widetilde{u}-\widetilde{w})}{|\nabla\psi_R^i(\widetilde{u}-\widetilde{w})|}\frac{dD^cu}{d|D^cu|}\Big)\,|\nabla\psi_R^i(\widetilde{u}-\widetilde{w})| d|D^cu|,
   \end{multline}
    where we have used that $w\in W^{1,1}_{\rm loc}(\Rd;\Rk)$ and that $f^\infty(x,0)=0.$ Taking advantage once again of \eqref{eq:Finfty bounds} and of $\text{Lip}(\psi_R^i)=1$, we infer that
    \begin{eqnarray*}\nonumber
&& \int_{A^i}f^\infty\Big(x\frac{\nabla\psi_R^i(\widetilde{u}-\widetilde{w})}{|\nabla\psi_R^i(\widetilde{u}-\widetilde{w})|}\frac{dD^cu}{d|D^cu|}\Big)\,|\nabla\psi_R^i(\widetilde{u}-\widetilde{w})| d|D^cu|\\&&\qquad \qquad\qquad\qquad
        \leq c_3k^{1/2}\int_{A^i}|\nabla\psi_R^i(\widetilde{u}-\widetilde{w})|d|D^cu|\leq \frac{c_3k^{1/2}}{c_1}\int_{A^i}f^\infty\Big(x,\frac{dD^cu}{|D^cu|}\Big)\,d|D^cu|.
    \end{eqnarray*}
    Together with the previous inequalities this leads to
    \begin{eqnarray} \nonumber\sum_{i=1}^m\int_{A}f^\infty\Big(x,\frac{D^cv^i_R}{|D^cv^i_R|}\Big)\,d|D^cv^i_R|&\leq& m\int_{A}f^\infty\Big(x,\frac{dD^cu}{d|D^cu|}\Big)\,d|D^cu|\\\label{eq:stime cantor concrete} 
    &&+\frac{c_3k^{1/2}}{c_1}\int_{A}f^\infty\Big(x,\frac{dD^cu}{|D^cu|}\Big)\,d|D^cu|.
    \end{eqnarray}
    
     We now estimate the surface term in $E^{f,g}$. To this scope, we first remark that $\jump{v^i_R}\subset A\setminus A_{\rm reg}$ and that for $\hd$-a.e $x\in\jump{v^i_R}$ we have $\nu_{v^i_R}(x)=\nu_{u}(x)$. Moreover, since $w\in W^{1,1}_{\rm loc}(\Rd;\Rk)$, for $\hd$-a.e. $x\in A$  the approximate limit $\widetilde{w}(x)$ exists, so that for $\hd$-a.e. every $x\in \jump{v^i_R}$ we have 
     \begin{eqnarray}
         &&\label{eq:v+}\displaystyle (v^i_R)^+(x)=\widetilde{w}(x)+\psi_R^i(u^+(x)-\widetilde{w}(x)),\\
         &&\label{eq:v-}\displaystyle (v^i_R)^-(x)=\widetilde{w}(x)+\psi_R^i(u^-(x)-\widetilde{w}(x)).
     \end{eqnarray}
     For $i\in\{1,...,m\}$ we introduce the following partition of  $\widetilde{J}_{v^i_R}:= \{x\in J_{v^i_R}\colon \widetilde w(x) \text{ exists}\} $:
    \begin{eqnarray*}
        &\displaystyle S^i_1:=\{x\in \widetilde{J}_{v^i_R}\colon |u^+(x)-\widetilde{w}(x)|<\sigma^{i-1}R,\,|u^{-}(x)-\widetilde w(x)|<\sigma^{i-1}R\},\\
        &\displaystyle S^i_2:=\{x\in \widetilde{J}_{v^i_R}\colon |u^+(x)-\widetilde{w}(x))|,|u^{-}(x)-\widetilde{w}(x)|\in[\sigma^{i-1}R,\sigma^{i}R]\},\\
        &\displaystyle S^i_3:=\{x\in \widetilde{J}_{v^i_R}\colon |u^+(x)-\widetilde{w}(x)|>\sigma^{i}R,\,|u^{-}(x)-\widetilde{w}(x)|>\sigma^{i}R\}\},\\
        &\displaystyle S^i_4:=\{x\in \widetilde{J}_{v^i_R}\colon |u^+(x)-\widetilde{w}(x)|<\sigma^{i-1}R,\,|u^{-}(x)-\widetilde{w}(x)|>\sigma^{i}R\},\\
       & \displaystyle S^i_5:=\{x\in \widetilde{J}_{v^i_R}\colon |u^+(x)-\widetilde{w}(x)|>\sigma^{i}R,\,|u^{-}(x)-\widetilde{w}(x)|<\sigma^{i-1}R\},
        \\&\displaystyle S^i_6:=\{x\in \widetilde{J}_{v^i_R}\colon |u^+(x)-\widetilde{w}(x)|\in[\sigma^{i-1}R,\sigma^iR],\,\,  |u^-(x)-\widetilde{w}(x)|\notin[\sigma^{i-1}R,\sigma^iR]\},\\
    &\displaystyle  S^i_7:=\{x\in  \widetilde{J}_{v^i_R}\colon \,|u^{+}(x)-\widetilde{w}(x)|\notin[\sigma^{i-1}R,\sigma^iR], \,\, 
       |u^-(x)-\widetilde{w}(x)|\in [\sigma^{i-1}R,\sigma^iR]\}.\\
    \end{eqnarray*}
    \noindent Note that for $\ell\in\{2,6,7\}$
    \begin{equation*}
    \displaystyle \text{if } S^i_\ell\cap S^j_\ell\neq \emptyset \text{\,\,\,\,then\,\, $|i-j|\leq 1$},
    \end{equation*} 
    so that 
    \begin{equation}
  \label{eq:intersection S267}\sum_{i=1}^m(\chi_{S^i_2}+\chi_{S^i_6}+\chi_{S^i_7})\leq 9.
    \end{equation}
    By definition of $S^i_1$ and of $S^i_3$, recalling \eqref{eq:v+} and \eqref{eq:v-}, we have 
\begin{eqnarray}
    &\displaystyle \label{eq:A1}\int_{S_1^i}g(x,[v^{i}_R],\nu_{v^i_R})\,d\hd=\int_{S_1^i}g(x,[u],\nu_u)\,d\hd,\\
    &\displaystyle \label{eq:A3}\int_{S_3^i}g(x,[v^{i}_R],\nu_{v^i_R})\,d\hd=\int_{S_3^i}g(x,0,\nu_u)\,d\hd=0.
\end{eqnarray}
    Since  $\text{Lip}(\psi_R^i)=1$, we have $|[\psi_R^i(u-w)]|\leq |[u-w]|=|[u]|$ in $\widetilde{J}_{{v^i_R}}$. Hence, by  (g2) and (g3) we  deduce that
    \begin{eqnarray}
        \nonumber
         \int_{S_2^i\cup S_6^i\cup S_7^i}g(x,[v^{i}_R],\nu_{v^i_R})\,d\hd&\leq&
         c_3k  \int_{S_2^i\cup S_6^i\cup S_7^i}|[\psi_R^i(u-w)]|\land 1\,d\hd \\  \leq c_3k  \int_{S_2^i\cup S_6^i\cup S_7^i}|[u]|\land 1\,d\hd
         &\leq&
         \label{eq:A2A5} 
         \frac{c_3k}{c_1}  \int_{S_2^i\cup S_6^i\cup S_7^i}g(x,[u],\nu_u)\,d\hd.
    \end{eqnarray}
    
    We are left with estimating the surface integral over $S_4^i$ and $S^i_5$. By definition, if $x\in S^i_4$ then  $|u^+(x)-\widetilde{w}(x)|<\sigma^{i-1}R$ and  $|u^{-}(x)-\widetilde{w}(x)|>\sigma^{i}R.$ 
    Therefore, 
    \begin{eqnarray}\nonumber\label{eq:dis a4 1}
        |[u](x)|=|u^+(x)-u^-(x)|&\geq&|u^-(x)-\widetilde{w}(x)|-|u^+(x)-\widetilde{w}(x)|\\
        &\geq&\sigma^iR-\sigma^{i-1}R=\sigma^{i-1}R(\sigma-1)\geq c_6k\sigma^{i-1}R,
   \end{eqnarray}
   where we have used that $\sigma>c_6k+1$. On the other hand, $[v^i_R](x)=u^+(x)-\widetilde{w}(x)$ so that
   \begin{equation}\nonumber\label{eq:dis a4 2}
       c_6k|[v^i_R]|= c_6k|u^+(x)-\widetilde{w}(x)|\leq c_6k\sigma^{i-1}R.  
   \end{equation}
   From these inequalities we deduce that $c_6k|[v^i_R]|\leq |[u]|$, which by (g6) implies 
    \begin{equation}\nonumber
        g(x,[v^i_R],\nu_{v^i_R})\leq  g(x,[u],\nu_{v^i_R})=g(x,[u],\nu_u) \quad \text{ $\hd$-a.e. in $S^{i}_4$}.
    \end{equation} 
    The same argument shows that $\hd$-a.e. in $S^{i}_5$  it holds
    \begin{equation}\nonumber
        g(x,[v_R^i],\nu_{v^i_R})\leq g(x,[u],\nu_u).
    \end{equation}
    From these last two inequalities we get 
   \begin{equation}\label{eq:A4}
        \int_{S_4^i\cup S^i_5}g(x,[v^{i}_R],\nu_{v^i_R})\,d\hd\leq  \int_{S_4^i\cup S^i_5}g(x,[u],\nu_{u})\,d\hd.
   \end{equation}
  
   Combining \eqref{eq:A1}-\eqref{eq:A4}, we obtain
   \begin{equation*}\label{eq:estimate on surface concrete}
       \int_{\jump{v^i_R}}g(x,[v^i_R],\nu_{v^{i}_R})\,d\hd\leq \int_{J_u}g(x,[u],\nu_u)\,d\hd+\frac{c_3k}{c_1}\int_{S^i_2\cup S^i_6\cup S^i_7 }g(x,[u],\nu_u)\,d\hd,
   \end{equation*}
which, in light of \eqref{eq:intersection S267}, implies that
\begin{equation}\label{eq:stimastrisce salti}
\sum_{i=1}^m\int_{\jump{v^i_R}}g(x,[v^i_R],\nu_{v^{i}_R})\,d\hd\leq \Big(m+\frac{9c_3k}{c_1}\Big)\int_{J_u}g(x,[u],\nu_u)\,d\hd.
\end{equation}
   Finally, from  \eqref{eq:contofatto su inner},  \eqref{eq:stime cantor concrete}, and \eqref{eq:stimastrisce salti} we deduce that
     \begin{eqnarray*}
       \frac1m\sum_{i=1}^mE^{f,g}(v^i_R,A) &\leq& E^{f,g}(u,A) + c_3k^{1/2}\int_{A^R_{u,w}}|\nabla w|\, dx+c_4\Ld(A^R_{u,w})\\
       &&+\frac{C}{m}\left( E^{f,g}(u,A) + \Ld(A) +\int_A|\nabla w|\, dx \right),
\end{eqnarray*}
where $C:=\max\{9c_3k/c_1,2c_3k^{1/2},c_4\}$.  This shows that $E^{f,g}$ satisfies (g).

To see that $E^{f,g}$ satisfies (h) it is enough to verify it for $A\in\mathcal{A}_c(\Rd)$ and for $u|_A\in GBV_\star(A;\Rk)$. Let $\lambda \leq 1/c_6$ and let $R\in SO(k)$. Note that by \eqref{eq:hp on constants} we also have $\lambda \leq (c_1/c_3)k^{-1/2}$. We set $v:=\lambda Ru$. By (f2'), (f3), and (g6) we can estimate
\begin{eqnarray*}
    E^{f,g}(v,A)&=&\int_{A}f(x,\nabla v)\, dx+\int_{A}f^\infty\Big(x,\frac{dD^cv}{d|D^cv|},\Big)\, d|D^cv|+\int_{A\cap \jump{u}}g(x,[v],\nu)\, d\hd\\&\leq&
c_3k^{1/2}\lambda\Big( \int_A|\nabla u|\, dx+|D^cu|(A)\Big)+\int_{A}g(x,[u],\nu)\,d\hd+c_4\Ld(A)\\
&\leq&
E^{f,g}(u,A)+(c_4+c_2)\Ld(A),
\end{eqnarray*}
concluding the proof.
\end{proof}

\begin{remark}
    It follows from Remark \ref{re:lowersemicontinuity of Veta} and Proposition \ref{prop: integral functionals are in E} 
    that the functional $V$ of Definition \ref{def:functional V}  belongs to $\E$.
\end{remark}
 The following lemma provides an estimate for smooth truncations of the form $\psi_R\circ u$, which can be considered as an extension to the vector-valued case of the simpler inequality \eqref{eq:estimate for truncations in dal maso and toader}, valid in the scalar case. The proof 
heavily relies on property (g) of Definition \ref{def:space of functionals E}. More refined versions of the following result will be presented in the forthcoming sections.

\begin{lemma}\label{lemma:monotic truncations}
    Let $A_1,A_2\in\mathcal{A}_c(\Rd)$, let $u_1\in L^0(A_1;\Rk)$ and $u_2\in L^0(A_2;\Rk)$, and for $j=1,2$ let $E(\cdot,A_j)\colon L^0(A_j;\Rk)\to [0,+\infty]$ be a a lower semicontinuous functional satisfying \textup{(c2$'$)} and \textup{(g)} of Definition \ref{def:space of functionals E} with $B=A_1$ and $B=A_2$. Then there exists a strictly increasing sequence $R_m>0$, with $R_m\to +\infty$ as $m\to+\infty$, such that
\begin{equation}\label{eq:claim lemma monotonic}
    \lim_{m\to+\infty}E(\psi_{R_m}\circ u_j,A_j)=E(u_j,A_j)
\end{equation}
for $j=1,2$.
\end{lemma}
\begin{proof}
For every sequence $R_m\to+\infty$ as $m\to+\infty$,  we have that $\psi_{R_m}\circ  u\to u_j$ in $L^0(A;\Rk)$ as $m\to+\infty$, so that the lower semicontinuity of $E(\cdot,A_j)$ with respect to the topology of $L^0(A_j;\Rk)$ implies that
\begin{equation}\label{eq:liminf nel lemma 3.14}
    E(u_j,A_j)\leq \liminf_{m\to\infty}E(\psi_{R_m}\circ u_j,A_j)
\end{equation}
for $j=1,2.$

Assume now that  $E(u_1,A_1)$  and $E(u_2,A_2)$ are both finite.
 We can choose a sequence $r_m$ converging to $+\infty$ as $m\to+\infty$ such that 
\begin{equation*}
c_4\Ld(A_{u_j,0}^{r_m})\leq \frac{1}{m}\quad \text{ for every }m\in\N \text{ and for $j=1,2$}.
\end{equation*}
From (g) of Definition \ref{def:space of functionals E} applied with $w=0$, we obtain that there exists $i(m)\in\{1,...,m\}$, such that
\begin{align}\nonumber
    E(\psi_{r_m}^{i(m)}\circ u_1,A_1)+ E(\psi_{r_m}^{i(m)}\circ u_2,A_2)&\leq E(u_1,A_1)+E(u_2,A_2)\\&\quad\nonumber 
+\frac{2+C(E(u_1,A_1)+E(u_2,A_2))+2c_4C\Ld(A))}{m}.
\end{align}
We now set $R_m:=\sigma^{i(m)}r_m$, so that $\psi^{i(m)}_{r_m}=\psi_{R_m}$. Hence, the previous inequality gives
\begin{equation*}
    \limsup_{m\to\infty}(E(\psi_{R_m}\circ u_1,A_1)+E(\psi_{R_m}\circ u_2,A_2))\leq E(u_1,A_1)+E(u_2,A_2).
\end{equation*}
Combining this last inequality with \eqref{eq:liminf nel lemma 3.14}, we obtain \eqref{eq:claim lemma monotonic} in the case where   $E(u_1,A_1)$  and $E(u_2,A_2)$ are both finite. A simpler proof yields the result when in the other cases.
\end{proof}

\section{A compactness result }\label{seq:compactness}
The main result of this section is a compactness theorem for the class $\E$. The strategy here adopted is the one of \cite[Lemma 3.24]{DalToa23}. As already mentioned, the main difference in the proof this compactness result in the vectorial case lies in the different type of truncation used, which is dealt with  in Lemma \ref{lemma: convergenza delle troncature nel gammalimsup} below.

In what follows, given a sequence $(E_n)_n\subset\E$ and $A\in\mathcal{A}(\Rd)$, we set 
\begin{eqnarray}\label{eq:definition Gammaliminf e Gammalimsup}
       &\displaystyle E'(\cdot,A):=\Gamma\meno\liminf_{n\to\infty}E_n(\cdot,A)\,\,\text{ and }\,\,E''(\cdot,A):=\Gamma\meno\limsup_{n\to\infty}E_n(\cdot,A),\\
       \label{eq:definition inner regularization}&\displaystyle E'_-(\cdot,A):=\sup_{A'\in\mathcal{A}_c(A)}E'(\cdot,A)\,\,\text{ and }\,\, E''_-(\cdot,A):=\sup_{A'\in\mathcal{A}_c(A)}E''(\cdot,A'),
    \end{eqnarray}
    where the $\Gamma\meno\liminf$ and the $\Gamma\meno\limsup$ are computed with respect to the topology of $L^0(\Rd;\Rk)$. Given $u\in L^0(\Rd;\Rk)$, it is immediate to check that  $E'(u,\cdot)$, $E''(u,\cdot)$, $E'_-(u,\cdot),$ and  $E_-''(u,\cdot)$ are all increasing set functions.
    
    \begin{theorem}\label{thm:compactness for E}
    Let $(E_n)_n\subset \E$ be a sequence of functionals. Then there exists a subsequence of $(E_n)_n$, not relabelled, and a functional $E\in\E_{\rm sc}$ such that for every $A\in\mathcal{A}_c(\Rd)$ the sequence $E_n(\cdot,A)$ $\Gamma$-converges to $E(\cdot,A)$ in the topology of $L^0(\Rd;\Rk)$.
\end{theorem}

\begin{proof}
    By a compactness theorem for $\Gamma$-convergence  of increasing functionals (see \cite[Theorem 16.9]{DalMaso}), there exists a subsequence of $(E_n)_n$, not relabelled, such that 
    \begin{equation}\label{eq:definizione limite E 1st}
        E(u,A):=E'_-(u,A)=E''_-(u,A)\quad \text{for every $u\in L^0(\Rd;\Rk)$ and } \text{for }A\in\mathcal{A}(\Rd) .\\ 
    \end{equation}
    Thus, we can define $E\colon L^0(\Rd;\Rk)\times\mathcal{B}(\Rd)\to[0+\infty]$ as 
    \begin{equation}\label{eq:definizione limite E 2nd}
        E(u,B)=\inf\{E(u,A):\,\,A\in\mathcal{A}(\Rd),\,\,B\subset A \}\quad \text{for every } B\in\mathcal{B}(\Rd).
    \end{equation}
   We claim that $E\in\E_{\rm sc}$ and that for every $A\in\mathcal{A}_c(\Rd)$ we have $E(\cdot,A)=E'(\cdot,A)=E''(\cdot,A)$.

  From some general $\Gamma$-convergence results (see \cite[Proposition 16.15]{DalMaso}), it follows that $E$ satisfies property (a) of Definition \ref{def:space of functionals E}.
  
    To show that $E$ enjoys also property (b), we make use of the De Giorgi-Letta Criterion for measures \cite{DeGiorgiLetta} (see also \cite[Theorem 14.23]{DalMaso} for the particular version of the theorem here employed). As already mentioned, for every $u\in L^0(\Rd;\Rk)$ the set function $E(u,\cdot)$ is increasing. The inner regularity follows by definition of $E(u,\cdot)$, while superadditivity is a consequence of \cite[Proposition 16.12]{DalMaso} and of the fact that $E_n$ satisfies (b) of Definition \ref{def:space of functionals E}. We are left with proving that $E(u,\cdot)$ is subadditive. To prove this, we make use of a truncation argument and of the following estimate for which we refer to \cite[Lemma 3.19]{DalToa23}, the proof in the vectorial case being the same as in the scalar case.
    
    \begin{lemma}[{\cite[Lemma 3.19]{DalToa23}}]\label{lemma:fundamental estimate}
        Let $(E_n)\subset \E$ be a sequence of functionals and let $A',A'',A,U\in\mathcal{A}_c(\Rd)$, with $A'\subset\subset A''\subset\subset A$. Let $u\in L^1_{\rm loc}(\Rd;\Rk)$ and let  $(w_n)_n,(v_n)_n\subset L^1_{\rm loc}(\Rd;\Rk)$ be two sequence converging to $u$ in $L^1_{\rm loc}(\Rd;\Rk)$ and such that 
        $(w_n|_A)\subset BV(A;\Rk)$, $(v_n|_U)\subset BV(U;\Rk)$. Then for every $\delta>0$ there exists a sequence $(\varphi_n)_n\subset C^\infty_c(\Rd,[0,1])$, with $\text{supp}(\varphi_n)\subset A''$ and $\varphi_n=1$ in a neighborhood of $\overline{A}'$, such that the functions $u_n$ defined by
        \begin{equation*}
            u_n:=\varphi_nw_n+(1-\varphi_n)v_n
        \end{equation*}
        converge to $u$ in $L^1_{\rm loc }(\Rd;\Rk)$, $u_n|_{A'\cup U}\in BV(A'\cup U;\Rk)$, and 
        \begin{eqnarray*}
            &\displaystyle \liminf_{n\to+\infty}E_n(u_n,A'\cup U)\leq(1+\delta)\liminf_{n\to+\infty}(E_n(w_n,A)+E_n(v_n,U))+\delta,\\
            &\displaystyle  \limsup_{n\to+\infty}E_n(u_n,A'\cup U)\leq(1+\delta)\limsup_{n\to+\infty}(E_n(w_n,A)+E_n(v_n,U))+\delta.
        \end{eqnarray*}
    \end{lemma}
 We now consider the truncated functions $\psi_R\circ u$ and prove that $E'(\psi_R\circ u,\cdot)$ and  $E''(\psi_R\circ u,\cdot)$ satisfy a weak subadditivity inequality.
\begin{lemma}\label{lemma:subadditivita troncata}
   Let $(E_n)_n\subset\E$ be a sequence of functionals, let $E'$ and  $E''$ be the functionals defined by \eqref{eq:definition Gammaliminf e Gammalimsup}, let $u\in L^0(\Rd;\Rk)$, and let  $A',A,U\in\mathcal{A}_c(\Rd)$, with $A'\subset\subset A$. Then for every $R>0$, $m\in\N$ we have 
   \begin{eqnarray}
        &\displaystyle \label{eq:subadd trunc'}E'(\psi_R\circ u,A'\cup U)\leq E''(\psi_R\circ u,A)+E'(\psi_R\circ u,U),\\
        &\displaystyle \label{eq:subadd trunc''}E''(\psi_R\circ u,A'\cup U)\leq E''(\psi_R\circ u,A)+E''(\psi_R\circ u,U).
   \end{eqnarray}
\end{lemma}   

\begin{proof}
    We prove \eqref{eq:subadd trunc'}.
        Without loss of generality, we assume that $E''(\psi_R\circ u,A)$ and $E'(\psi_R\circ u,U)$ are both finite. 
       Let $w_n,v_n\in L^0(\Rd;\Rk)$ be two sequence of functions converging in $L^0(\Rd;\Rk)$ to $\psi_R\circ u$ and such that
        \begin{equation}\label{eq:recovery on A,B}
            \limsup_{n\to\infty}E_n(w_n,A)=E''(\psi_R\circ u,A)\,\,\text{ and }\,\, \liminf_{n\to\infty}E_n(v_n,U)=E'(\psi_R\circ u,U).
        \end{equation}
        We fix a subsequence $(E_{n_h})_{h}$ of $(E_n)_n$ such that 
        \begin{equation}\nonumber 
            \lim_{h\to+\infty}E_{n_h}(v_{n_h},U)=\liminf_{n\to\infty}E_n(v_n,U)=E'(\psi_R\circ u,U).
        \end{equation}
      Without loss of generality we may assume  that there exists $M>0$ such that 
        \begin{equation*}
            E_{n_h}(w_{n_h},A)<M\quad\text{ and }\quad 
            E_{n_h}(v_{n_h},U)<M,
        \end{equation*}
        for every $h\in\N$.
         Remark \ref{re:Domain of function spaces} then implies $(w_{n_h}|_A)_h\subset GBV_\star(A;\Rk)$, $(v_{n_h}|_U)_h\subset GBV_\star(U;\Rk)$.
          By property (g) of Definition \ref{def:space of functionals E} for every $h,m\in\N$ we can find  $i_{h,m}, j_{h,m}\in\{1,...,m\}$ such that 
        \begin{eqnarray}\label{eq:auxiliary estimates 1st lemma subadditivo}
               &\displaystyle \hspace{-0.7cm} E_{n_h}(\psi^{i_{h,m}}_{\sigma R}\circ w_{n_h},A)\leq E_{n_h}(w_{n_h},A)+c_4\Ld(A^{\sigma R}_{w_{n_h},0})+C\frac{M+\Ld(A)}{m},\\
                \label{eq:auxiliary estimates 2nd lemma subadditivo}
               &\displaystyle \hspace{-0.7cm} E_{n_h}(\psi^{j_{h,m}}_{\sigma R}\circ v_{n_h},U)\leq E_{n_h}(v_{n_h},U)+c_4\Ld(U^{\sigma R}_{v_{n_h},0})+C\frac{M+\Ld(U)}{m}.
         \end{eqnarray}
        Since $w_{n_h}$ and $v_{n_h}$ converge in $L^0(\Rd;\Rk)$ to $\psi_R\circ u$ as $n\to+\infty$, $\|\psi_{\sigma R}^{i_{h,m}}\|_{L^\infty(\Rd;\Rk)}\leq \sigma^{m+1}R$ , $\|\psi_{\sigma R}^{j_{h,m}}\|_{L^\infty(\Rd;\Rk)}\leq \sigma^{m+1}R $, and $\|\psi_R\circ u\|_{L^\infty(\Rd;\Rk)}\leq \sigma R$,  for every $m\in\N$ we infer that 
        \begin{eqnarray}
           &\displaystyle  \nonumber \psi_{\sigma R}^{i_{h,m}}\circ w_{n_h}\to \psi_R\circ u\,\,\, \text{ in }L^1(\Rd;\Rk)\, \text{  as }h\to+\infty,\\
           &\displaystyle \nonumber  \psi_{\sigma R}^{j_{h,m}}\circ v_{n_h}\to \psi_R\circ u\,\,\, \text{ in }L^1(\Rd;\Rk)\, \text{  as }h\to+\infty,\\
           &\displaystyle\label{eq:misura a zero}\lim_{h\to\infty}\Ld(A^{\sigma R}_{w_{n_h},0})=
         \Ld(A_{\psi_R\circ u,0}^{\sigma R})=0,\\
          &\displaystyle\label{eq:misura a zero 2} \lim_{h\to\infty}\Ld(U^{\sigma R}_{v_{n_h},0})=
         \Ld(U_{\psi_R\circ u,0}^{\sigma R})=0.
        \end{eqnarray}
         For every $V\in\mathcal{A}(\Rd)$ we set
         \begin{equation}\label{eq:def su sottosucc}
\widehat{E}'(\cdot,V):=\Gamma\meno\liminf_{h\to\infty}E_{n_h}(\cdot,V)\,\,\text{ and }\,\,\widehat{E}''(\cdot,V):=\Gamma\meno\limsup_{h\to\infty}E_{n_h}(\cdot,V),\\
       \end{equation}
       and observe that 
       \begin{equation}\label{eq:monotnoia su sottosucc}
          E'(\cdot,V)\leq  \widehat{E}'(\cdot,V)\quad \text{and }\quad \widehat{E}''(\cdot,V)\leq E''(\cdot,V).
       \end{equation}
       
        \noindent By  \eqref{eq:recovery on A,B}-\eqref{eq:misura a zero 2}, we deduce that
        \begin{eqnarray}
            &  \label{eq:auxiliary inequality 1st lemma subadditivo}\displaystyle\limsup_{h\to\infty}E_{n_h}(\psi^{i_{h,m}}_{\sigma R}\circ w_{n_h},A)\leq E''(\psi_R\circ u,A)+C\frac{M+\Ld(A)}{m}\\ \label{eq:auxiliary inequality 2nd lemma subadditivo}
            &\displaystyle \liminf_{h\to\infty}E_{n_h}(\psi^{j_{h,m}}_{\sigma R}\circ v_{n_h},U)\leq E'(\psi_R\circ u,U)+C\frac{M+\Ld(U)}{m}.
        \end{eqnarray}
        By Proposition \ref{prop:characterisation GBVstar} we have that $\psi_{\sigma R}^{i_{h,m}}\circ w_{n_h}|_{A}\in BV(A;\Rk)$ and that $\psi_{\sigma R}^{j_{h,m}}\circ v_{n_h}|{U}\in BV(U,\Rk)$. Thus, for $\delta>0$ and for every $m\in\N$ we may apply Lemma \ref{lemma:fundamental estimate} to obtain a sequence of functions $(u^m_{h})_h\subset BV(A'\cup U;\Rk)$ such that $u^m_{h}\to \psi_R\circ u$ in $L^1_{\rm loc}(\Rd;\Rk)$ as $h\to+\infty$ and such that 
        \begin{eqnarray*}
            &\displaystyle\liminf_{h\to+\infty}E_{n_h}(u^m_{h},A'\cup U)\leq (1+\delta)\liminf_{h\to\infty}\Big(E_{n_h}(\psi^{i_{h,m}}_{\sigma R}\circ w_{n_h},A)+E_{n_h}(\psi_{\sigma_R}^{j_{h,m}}\circ v_{n_h},U)\Big)+\delta\\
            &\displaystyle\leq (1+\delta)\limsup_{h\to\infty}E_{n_h}(\psi^{i_{h,m}}_{\sigma R}\circ w_{n_h},A)+(1+\delta)\liminf_{h\to+\infty}E_{n_h}(\psi_{\sigma_R}^{j_{h,m}}\circ v_{n_h},U)+\delta
        \end{eqnarray*}
        
        This inequality, combined with  \eqref{eq:auxiliary inequality 1st lemma subadditivo}, and \eqref{eq:auxiliary inequality 2nd lemma subadditivo}, yields
        \begin{equation*}
             \liminf_{h\to\infty}E_{n_h}(u^m_{h}
            ,A'\cup U)\leq (1+\delta)\Big(E''(\psi_R\circ u,A)+E'(\psi_R\circ u,U)+\frac{S}{m}\Big)+\delta,
        \end{equation*}
    where $S$ is a positive constant independent of $m$. Recalling that $u^m_h\to \psi_R\circ u$ as $h\to+\infty$,  by \eqref{eq:def su sottosucc} and \eqref{eq:monotnoia su sottosucc} we obtain that 
    \begin{equation*}
        E'(\psi_R\circ u,A'\cup U)\leq \widehat{E}'(\psi_R\circ u,A'\cup U)\leq (1+\delta)\Big(E''(\psi_R\circ u,A)+E'(\psi_R\circ u,U)+\frac{S}{m}\Big)+\delta.
    \end{equation*}
 Letting $m\to +\infty$ and  $\delta\to0^+$, we obtain \eqref{eq:subadd trunc'}.

 A similar, but easier, argument shows that \eqref{eq:subadd trunc''} holds true.
    \end{proof}
    To prove a weak subadditivity inequality for $E'(u,\cdot)$ we will approximate $E'(u,A)$ by $E'(\psi_R\circ u,A)$. For technical reasons, this approximation result is obtained using property $(g)$ of Definition \ref{def:space of functionals E} for $E'$, which is proved in the following lemma.
    \begin{lemma}\label{lemma: condition g passes to the limit}
        Let $(E_n)_n$ be a sequence in $\E$, let $E'$ be the functional defined by \eqref{eq:definition Gammaliminf e Gammalimsup}, let $A\in\mathcal{A}_c(\Rd)$, $u\in GBV_\star(A;\Rk)$, and let $w\in W^{1,1}_{\rm loc}(\Rd;\Rk)$. Then for every $R>0$ and $m\in\N$ we have 
        \begin{multline}\label{eq:claim prop g limit}
            \frac{1}{m}\sum_{i=1}^mE'(w+\psi^i_R\circ (u-w),A)\leq E'(u,A) + c_3k^{1/2}\int_{A^R_{u,w}}|\nabla w|+c_4\Ld(A^R_{u,w}) \\
       +\frac{C}{m}\left(E'(u,A)+ \Ld(A)+\int_A|\nabla w|\, dx \right),
        \end{multline}
        where $C=\{9c_3k/c_1, 2c_3k^{1/2},c_4\}$.
    \end{lemma}
    \begin{proof}
        Let us fix $u\in L^0(\Rd;\Rk)$ and $w\in W^{1,1}_{\rm loc}(\Rd;\Rk).$ Without loss of generality, we may assume that $E'(u,A)<+\infty$. Let $(u_n)_n\subset L^0(\Rd;\Rk)$ be a sequence of functions converging to $u$ in $L^0(\Rd;\Rk)$ and such that 
        \begin{equation}\label{eq:limit prop g 1}
        \liminf_{n\to+\infty}E_n(u_n,A)=E'(u,A).
        \end{equation}
        Fix $R>0$ and $m\in\N$. For every $i\in\{1,...,m\}$, we set $v^i_n:=w+\psi_R^i\circ (u_n-w)$ and $v^i:=w+\psi_R\circ(u-w)$. Note that for every $i\in\{1,...,m\}$ the sequence $(v^i_n)_n$ converges to $v^i$ in $L^0(\Rd;\Rk)$, so that by Definition of $E'$ we have
    \begin{equation}\label{eq:limit prop g 2}
        \frac{1}{m}\sum_{i=1}^mE'(v^i,A)\leq \frac{1}{m}\sum_{i=1}^m\liminf_{n\to+\infty}E_n(v^i_n,A)\leq \liminf_{n\to+\infty}\frac{1}{m}\sum_{i=1}^mE_n(v^i_n,A).
    \end{equation}
    By property (g) and by \eqref{eq:limit prop g 1} we get
    \begin{eqnarray}\nonumber 
         \frac{1}{m}\sum_{i=1}^mE_n(v^i_n,A)&\leq&E_n(u_n,A)+
         \frac Cm\Big(\displaystyle E_n(u_n,A)+\Ld(A)+\int_A|\nabla w|\Big)\\&&\nonumber 
    +c_3k^{1/2}\int_{A_{u_n,w}^R}|\nabla w|\,dx+c_4\Ld(A_{u_n,w}^R).
         \end{eqnarray}
         By \eqref{eq:limit prop g 1}, this inequality gives 
         \begin{eqnarray}\nonumber
             \frac{1}{m}\sum_{i=1}^mE'(v^i,A)&\leq& E'(u,A)+
         \frac Cm\Big(\displaystyle E'(u,A)+\Ld(A)+\int_A|\nabla w|\Big)\\
         && \label{eq:prog g2} 
    +\limsup_{n\to+\infty}\Big(c_3k^{1/2}\int_{A_{u_n,w}^R}|\nabla w|\,dx+c_4\Ld(A_{u_n,w}^R)\Big).
         \end{eqnarray}
    Since $u_n\to u$ in $L^0(\Rd;\Rk)$, we have that 
    \begin{equation*}
        \limsup_{n\to+\infty}\chi_{A^R_{u_n,w}(x)}\leq \chi_{A^R_{u,w}}(x)\quad \text{for $\Ld\meno$a.e. } x\in A.
     \end{equation*}
 Thus, from Fatou's Lemma, we deduce that
    \begin{equation}\label{eq:limit prop g 3}
        \limsup_{n\to+\infty}\Big(\int_{A_{u_n,w}^R}|\nabla w|\, dx+\Ld(A^R_{u_n,w})\Big)\leq \int_{A^R_{u,w}}|\nabla w|\, dx+\Ld(A^R_w),
    \end{equation}
which combined with \eqref{eq:prog g2} yields \eqref{eq:claim prop g limit}.
    \end{proof}
 To obtain the weak subadditivity of $E'(u,\cdot)$ we show that $E'(u,A)$ can be approximated by $E'(\psi_R\circ u,A)$ for a suitable choice of $R$.
    \begin{lemma}\label{lemma: convergenza delle troncature nel gammalimsup}
        Let $(E_n)_n$ be a sequence in $\E$, let $E'$ be the functional defined by \eqref{eq:definition Gammaliminf e Gammalimsup}, let $A_1,A_2\in\mathcal{A}_c(\Rd)$,  $u_1\in L^0(A_1;\Rk)$ and $u_2\in L^0(A_2;\Rk)$.  Then there exists an increasing sequence $R_m>0$, with $R_m\to+\infty$ as $m\to+\infty$,  such that
        \begin{equation}\label{eq:monotonicity}
           \lim_{m\to+\infty}E'(\psi_{R_{m}}\circ u_j,A_j)=E'(u_j,A_j)
        \end{equation}
        for $j=1,2.$
    \end{lemma}
    \begin{proof}
     It is immediate to check that $E'$ satisfies property (c2), and by Lemma \ref{lemma: condition g passes to the limit}, we have that $E'$ satisfies property (g) for $B=A_1$ and $B=A_2$. Since  the functional $E'(\cdot,A_j)$ is lower semicontinuous with respect to the topology of $L^0(A_j;\Rk)$ for $j=1,2$ (see \cite[Proposition 6.8]{DalMaso}), an application of Lemma \ref{lemma:monotic truncations} then proves the claim.
    \end{proof}

 We now prove a weak subadditivity inequality for $E'(u,\cdot)$.
\begin{lemma}\label{lemma:subadditivita weak}
      Let $(E_n)_n$ be a sequence of functionals in $\E$ and  let $E', E''$ be defined by \eqref{eq:definition Gammaliminf e Gammalimsup}. Assume that there exists a functional $E\colon L^{0}(\Rd;\Rk)\times \mathcal{A}(\Rd)\to[0,+\infty]$ such that \eqref{eq:definizione limite E 1st} holds. Let $u\in L^0(\Rd;\Rk)$, and $A',A,U\in\mathcal{A}_c(\Rd)$ with $A'\subset\!\subset A$. Then we have
   \begin{equation}\label{eq:subadd su E'}
         E'(u,A'\cup U)\leq E'(u,A)+E'(u,U).
   \end{equation}
\end{lemma}
\begin{proof}
   By \cite[Proposition 15.15]{DalMaso} we may choose $A''\in \mathcal{A}_c(\Rd)$, with $A'\subset\subset A''\subset\subset A$,  such that 
    \begin{equation}\label{eq:def E lemma storto}
        E''(v,A'')=E'(v,A'')=E(v,A'') \quad \text{for every }v\in L^0(\Rd;\Rk).
    \end{equation}
    Lemma \ref{lemma: convergenza delle troncature nel gammalimsup} implies that there exists $R_m\to +\infty$ as $m\to+\infty$ such that
    \begin{eqnarray}
       \label{eq:lemma 1} &\displaystyle \lim_{m\to+\infty}E'(\psi_{R_m}\circ u,A'')=E'(u,A''),\\
        \label{eq:lemma 2}&\displaystyle \lim_{m\to+\infty}E'(\psi_{R_m}\circ u,U)=E'(u,U).
    \end{eqnarray}
    Thanks, to Lemma \ref{lemma:subadditivita troncata}, we 
    obtain that 
    \begin{equation*}
        E'(\psi_{R_m}\circ u,A'\cup U)\leq E''(\psi_{R_m}\circ u,A'')+E'(\psi_{R_m}\circ u,U)\leq E'(\psi_{R_m}\circ u,A)+E'(\psi_{R_m}\circ u,U),
    \end{equation*}
where in the equality we have used \eqref{eq:def E lemma storto}. Since $\psi_{R_m}\circ u\to u$ in $L^0(\Rd;\Rk)$ as $m\to+\infty$, by the lower semicontinuity of $E'$, \eqref{eq:lemma 1}, and \eqref{eq:lemma 2} we obtain
\begin{equation*}
    E'(u,A'\cup U)\leq E'(u,A'')+E'(u,U)\leq E'(u,A)+E'(u,U),
\end{equation*}
concluding the proof.
\end{proof}

We are now ready to prove the subadditivity of $E(u,\cdot)$ on $\mathcal{A}(\Rd)$.
\begin{lemma}\label{lemma:subadditivita}
Let $E_n$ be a sequence in $\E$ for which  \eqref{eq:definizione limite E 1st} holds for some functional $E$. Then for every $u\in L^0(\Rd;\Rk)$ and $A,U\in\mathcal{A}(\Rd)$ we have
\begin{equation*}
    E(u,A\cup U)\leq E(u,A)+E(u,U).
\end{equation*}
\begin{proof}
    The result follows from Lemma \ref{lemma:subadditivita weak} and by standard arguments (see for instance the proof of \cite[Lemma 18.4]{DalMaso}).
\end{proof}
\end{lemma}

Finally, we show that $(E_n(\cdot,A))_n$ $\Gamma$-converges for every $A\in\mathcal{A}_c(\Rd)$.
\begin{lemma}\label{lemma:equaility E',E''}
    Let $(E_n)_n$ be a sequence of functionals in $\E$, let $E', E''$ be the functionals defined by \eqref{eq:definition Gammaliminf e Gammalimsup}. Assume that there exists a functional $E$ such that  \eqref{eq:definizione limite E 1st} holds.  Then $E(u,A)=E'(u,A)=E''(u,A)$ for every $u\in L^0(\Rd;\Rk)$ and $A\in\mathcal{A}_c(\Rd)$.
\end{lemma}
\begin{proof}
 Fix $u\in L^0(\Rd;\Rk)$ and $A\in\mathcal{A}_c(\Rd)$. We first show that $E(u,A)=E'(u,A)$. Since $E(u,A)\leq E'(u,A)$, it is enough to show the converse inequality. Without loss of generality we may suppose that $E(u,A)<+\infty$. It is immediate to check that $E'$ and $E$ satisfies properties (c2) Definition  \ref{def:space of functionals E} on $\mathcal{A}(\Rd)$ and using Remark \ref{re:lowersemicontinuity of Veta} we check that $E'$ and $E$  also satisfy (c1) on $\mathcal{A}(\Rd)$. Hence, $u\in GBV_\star(A;\Rk)$. Fix $\e>0$ and consider a compact set $K\subset A$ such that 
  \begin{equation}\label{eq:epsilon bound alto subad}
      c_3V(u,A\setminus K)+c_4\Ld(A\setminus K)\leq \e.
  \end{equation}
  Consider now $A',A''\in\mathcal{A}_c(A)$ with $K\subset A'\subset \subset A''$ and let $U=A\setminus K$. Note that $A'\cup U=A$. Hence, by Lemma \ref{lemma:subadditivita weak} to obtain 
  \begin{equation}\label{eq:equality with E'}
      E'(u,A)\leq E'(u,A'')+E'(u,A\setminus K)\leq E(u,A)+\e,
  \end{equation}
  where in the last inequality we have used \eqref{eq:definizione limite E 1st}, property (c2) for $E'$ and \eqref{eq:epsilon bound alto subad}. These arguments shows that $E(u,A)=E'(u,A)$ for every $u\in L^0(\Rd;\Rk)$ and $A\in\mathcal{A}_c(\Rd)$.

  We now prove that $E(u,A)=E''(u,A)$. Since $E(u,A)\leq E''(u,A)$ it is enough to show that $E''(u,A)\leq E(u,A)$. It is not restrictive to assume that $E(u,A)<+\infty$, so that $u\in GBV_\star(A;\Rk)$. Given $R>0$, we can  exploit the same argument used to obtain \eqref{eq:equality with E'}, replacing \eqref{eq:subadd su E'} of Lemma \ref{lemma:subadditivita weak} with \eqref{eq:subadd trunc''}, to show that for every $\e>0$
  \begin{equation}
      E''(\psi_R\circ u,A)
     \label{eq:strumentale compattezza}\leq E(\psi_R\circ u,A)+\e=E'(\psi_R\circ u,A)+\e,
  \end{equation}
  where  we have used the equality $E=E'$. By Lemma \ref{lemma: convergenza delle troncature nel gammalimsup} there exists a positive sequence $R_m\to+\infty$ such that 
  \begin{equation*}
      \lim_{m\to+\infty} E'(\psi_{R_m}\circ u,A)=E'(u,A)=E(u,A).
  \end{equation*}
  Since $\psi_{R_m}\circ u\to u$ in $L^0(\Rd;\Rk)$ as $m\to +\infty$, exploiting the lower semicontinuity of $E''$ and this last equality, we deduce from \eqref{eq:strumentale compattezza} that
  \begin{equation*}
      E''(u,A)\leq E(u,A)+\e.
  \end{equation*}
Since $\e>0$ is arbitrary, we get  $E''(u,A)\leq E(u,A)$, concluding the proof.
\end{proof}
\smallskip

\noindent{\it Continuation of the Proof of Theorem \ref{thm:compactness for E}}. 
Thanks to  Lemma \ref{lemma:equaility E',E''},  for every $u\in L^0(\Rd;\Rk)$ and $A\in\mathcal{A}_c(\Rd)$ we have that 
\begin{equation}\label{eq:equality of E}
    E(u,A)=E'(u,A)=E''(u,A).
\end{equation}

We are left with proving that $E\in\E_{\it sc}$. We already noted that $E$ satisfies (a) of Definition \ref{def:space of functionals E}. From Lemma \ref{lemma:subadditivita} and De Giorgi-Letta theorem, we deduce that $E$ satisfies (b) as well.
Properties (c2)-(f) can be derived arguing as in \cite[Theorem 3.16]{DalToa23}. 

By Remark \ref{re:lowersemicontinuity of Veta} and \eqref{eq:equality of E}, we infer that for ever
$u\in L^0(\Rd;\Rk)$ and $A\in\mathcal{A}_c(\Rd)$ we have
\begin{equation*}
   c_1 V(u,A)-c_2\Ld(A)\leq E(u,A).
\end{equation*}
By inner regularity the inequality can be extended to  $ \mathcal{A}(\Rd)$ and, recalling \eqref{eq:definizione limite E 2nd}, to $\mathcal{B}(\Rd)$.
Property (g) for $E$ on $\mathcal{A}_c(\Rd)$ is proved in Lemma \ref{lemma: condition g passes to the limit}, while (h) is trivial. The extension of (g) and (h) to all Borel sets can easily be obtained using \eqref{eq:definizione limite E 2nd}.

Finally, by \cite[Proposition 6.8]{DalMaso} for every $A\in\mathcal{A}_c(\Rd)$ the functionals $E'(\cdot,A)$ and $E''(\cdot,A)$ are lower semicontinuous with respect to the convergence of $L^0(\Rd;\Rk)$. Since $E=E'=E''$, we conclude that $E\in\E_{\rm sc}$.
 \end{proof}
 
\section{Partial Integral representation}\label{sec:partial representation}

In this Section we present and prove a (partial) representation result for functionals in $\E_{\rm sc}$. We postpone the full representation to Section \ref{sec:repr}, where we will work with functionals in $\E_{\rm sc}$ satisfying an additional property.

We introduce a splitting of $E$ that mimics the structure of functionals in $E^{f,g}$.
\begin{definition}\label{def:splitting of the functionals}
Let $E\colon L^0(\Rd;\Rk)\times\mathcal{B}(\Rd)\to[0,+\infty]$ be a functional satisfying (b) and (c2) of Definition \ref{def:space of functionals E}, and let $A\in\mathcal{A}_c(\Rd)$. For every $u\in L^0(\Rd;\Rk)$, we introduce  $E^a(u,\cdot)$, $E^s(u,\cdot)$, $E^c(u,\cdot)$, and $E^j(u,A)$ as the Borel measures on $\mathcal{B}(A)$ defined as:
    \begin{eqnarray*}
        &\displaystyle E^a(u,\cdot)\,\, \text{is the absolutely continuous part of $E(u,\cdot)$ with respect to }\Ld, \\
        &\displaystyle E^s(u,\cdot)\,\, \text{is the singular part  of $E(u,\cdot)$ with respect to }\Ld,\\
        &\displaystyle E^j(u,B)=E^s(u,B\cap \jump{u})\,\,\text{ for every }B\in\mathcal{B}(A),\\ 
        &\displaystyle E^c(u,B)=E^s(u,B\setminus\jump{u}) \,\,\text{ for every }B\in\mathcal{B}(A).
    \end{eqnarray*}
\end{definition}
Note that by arguing as in \cite[Remark 4.2]{DalToa23}, we see that when $u\in GBV_\star(A;\Rk)$, the measures  $E^j(u,\cdot)$ and $E^c(u,\cdot)$ are the absolutely continuous parts with respect to $\hd\mres\jump{u}$ and $|D^cu|$, respectively. Additionally, for every $A\in\mathcal{A}_c(\R)$ it holds
\begin{equation*}
    E(u,\cdot)=E^{a}(u,\cdot)+E^{c}(u,\cdot)+E^j(u,\cdot)\,\,\,\text{ on }\mathcal{B}(A).
\end{equation*}

The aim of this section is to show that for every $A\in\mathcal{A}_c(\Rd)$ and for every $u\in GBV_\star(A;\Rk)$ we may represent the two measures $E^a(u,\cdot)$ and  $E^j(u,\cdot)$ as integrals. 

As in \cite{DalToa23}, the idea is to take advantage of the representation results of \cite{bouchitte1998global} for functionals on $BV(A;\Rk)$ with linear growth. These results cannot be directly applied to $E^a(u,\cdot)$ and $E^j(u,\cdot)$, since functionals in $\E$ are not bounded from below by the total variation measure $|Du|$. This difficulty is treated by restricting our attention first to functions $u\in BV(A;\Rk)\cap L^\infty(A,\Rk)$, and by considering the perturbed functionals $E^\delta(u,A):=E(u,A)+\delta|Du|(A)$. The absolutely continuous part and the jump part of $E^\delta$ can then be represented by means of an integral thanks to \cite{bouchitte1998global}. We then show that it is possible to pass to the limit as $\delta\to 0^+$ and to recover an integral representation of $E^a(u,\cdot)$ and of $E^j(u,\cdot)$ for every $u\in BV(A;\Rk)$.

The following result then allows us to obtain the result for a general $u\in GBV_\star(A;\Rk)$. 

\begin{lemma}\label{lemma:passaggio al limite dalle troncate al non nei vari pezzi}
Let $E\in\E_{\rm sc}$, let $A\in\mathcal{A}_c(\Rd)$, let  $u\in GBV_\star(A;\Rk)$, and  let $(R_m)_m$ be a sequence with $R_m>0$ and $R_m\to+\infty$. Then 
\begin{eqnarray}\label{eq:passage to the limit in full functional}
    &\displaystyle \lim_{m\to +\infty}\Big( \frac{1}{m} \sum_{i=1}^mE(\psi^{i}_{R_m}\circ u,B)\Big)=E(u,B),\\  \label{eq: passage to the limit nella parte AC} 
    &\displaystyle\lim_{m\to+\infty}\Big( \frac{1}{m} \sum_{i=1}^mE^a(\psi^{i}_{R_m}\circ u,B)\Big)=E^a(u,B),\\
    &\displaystyle \label{eq:passage to the limit nella parte singolare}\lim_{m\to+\infty}\Big( \frac{1}{m} \sum_{i=1}^mE^s(\psi^{i}_{R_m}\circ u,B)\Big)=E^s(u,B),\\  \label{eq:passage to the limit nella parte di salto}
    &\displaystyle\lim_{m\to+\infty} \Big( \frac{1}{m} \sum_{i=1}^mE^j(\psi^{i}_{R_m}\circ u,B)\Big)=E^j(u,B),\\\label{eq:passage to the limit nella parte di cantor}
    &\displaystyle \lim_{m\to+\infty}\Big( \frac{1}{m} \sum_{i=1}^mE^c(\psi^{i}_{R_m}\circ u,B)\Big)=E^c(u,B),
\end{eqnarray}
for every $B\in\mathcal{B}(A)$.
\end{lemma}
\begin{proof} For every $m\in\N$ we set 
\[
\mu_m(B):=\frac{1}{m}\sum_{i=1}^{m}E(\psi_{R_m}^i\circ u,B) \quad \text{ for every} B\in\mathcal{B}(A)
\]
and note that this defines a finite Radon measure, being $u\in GBV_\star(A;\Rk)$.

Let $U\in\mathcal{A}(A)$. Thanks to property (g) of Definition \ref{def:space of functionals E},  we have that
\begin{equation*}
    \mu_m(U)\leq E(u,U)+c_4\Ld(U_{u,0}^{R_m}) 
       +\frac{C}{m}\left(E(u,U) +\Ld(U)  \right).
\end{equation*}
 Taking the limsup for $m\to+\infty$ in the previous inequality, we get 
\begin{equation*}
    \limsup_{m\to+\infty}\mu_m(U)\leq E(u,U).
\end{equation*}

For every $m\in\N$, there exists $i(m)\in\{1,...,m\}$ such that 
\begin{equation*}
    E(\psi^{i(m)}_{R_m}\circ u,U)\leq \mu_m(U).
\end{equation*}
Since $\psi^{i(m)}_{R_m}\circ u\to u$ in $L^0(\Rd;\Rk)$ when $m\to+\infty$ and $E(\cdot,U)$ is lower semicontinuous with respect to the topology of $L^0(\Rd;\Rk)$, we have
\begin{equation*}
    E(u,U)\leq \liminf_{m\to+\infty}  E(\psi^{i(m)}_{R_m}\circ u,U)\leq    \limsup_{m\to+\infty}\mu_m(U)\leq E(u,U),
\end{equation*}
so that \eqref{eq:passage to the limit in full functional} holds for every $U\in\mathcal{A}(A)$. From \cite[Lemma 4.4]{DalToa23} we then deduce that \eqref{eq:passage to the limit in full functional} holds for every $B\in\mathcal{B}(A)$.

To show that \eqref{eq: passage to the limit nella parte AC}  holds true, consider a set $N\in\mathcal{B}(A)$, with $\Ld(N)=0$ , such that for every $B\in\mathcal{B}(A)$ we have $E^a(u,B)=E(u,B\setminus N)$ and  $\mu^a_m(B)=\mu_m(B\setminus N)$ for every $m\in\N$, where $\mu^a$ is the absolutely continuous part of $\mu$ with respect to $\Ld$.
By \eqref{eq:passage to the limit in full functional} we have that 
\begin{equation*}
\lim_{m\to+\infty}\mu_m^a(B)=\lim_{m\to+\infty}\mu_m(B\setminus N)=E(u,B\setminus N)=E^a(u,B),
\end{equation*}
which proves $\eqref{eq: passage to the limit nella parte AC}$. Taking the difference of \eqref{eq:passage to the limit in full functional} and of \eqref{eq: passage to the limit nella parte AC} we obtain \eqref{eq:passage to the limit nella parte singolare}.

Finally,
 equalities  \eqref{eq:passage to the limit nella parte di salto} and  \eqref{eq:passage to the limit nella parte di cantor} can be obtained arguing as in \cite[Proposition 4.3]{DalToa23}, replacing Theorem 2.2(d) with our Proposition \ref{prop: Properties of GBVsvector}(d).
\end{proof}

The following result shows that the singular part $E^s$ satisfies a simplified version of property (g) of Definition \ref{def:space of functionals E}. 
\begin{lemma}\label{lemma:estimates on the singular part}
Let $E\in\E$,  $A\in\mathcal{A}_c(\Rd)$, $u\in GBV_\star(A;\Rk)$, $R>0$,  $m\in\N$, $B\in\mathcal{B}(A)$, and  $w\in W^{1,1}_{\rm loc}(\Rd;\Rk)$. Then 
\begin{equation*}
    \frac{1}{m}\sum_{i=1}^mE^s(w+\psi^i_R\circ (u-w),B)\leq \Big(1+\frac{C}{m}\Big)E^s(u,B),
\end{equation*}
where $C=\max\{9c_3k/c_1,2c_3k^{1/2},c_4\}$.
\end{lemma}
\begin{proof}
   Let $N\in\mathcal{B}(\Rd)$ be a Borel set, with $\Ld(N)=0$, such that $E^s(u,B)=E(u,B \cap N)$ and $E^s(w+\psi^i_R\circ (u-w),B)=E(w+\psi^i_R\circ (u-w),B\cap N)$ for every $i\in\{1,...m\}$, and for every $B\in\mathcal{B}(A)$. Since \eqref{eq: g} holds for every Borel set $B\in\mathcal{B}(A)$, we get
   \begin{eqnarray*}
     &&\frac{1}{m}\sum_{i=1}^mE^s(w+\psi^i_R\circ (u-w),B)= \frac{1}{m}\sum_{i=1}^mE(w+\psi^i_R\circ (u-w),B\cap N)\\
      &&\quad \leq E(u,B\cap N) + \int_{B^R_{u,w}\cap N}(c_3k^{1/2}|\nabla w|+c_4)\, dx+\frac{C}{m}\left( E(u,B\cap N)+ \int_{B\cap N}(|\nabla w|+1)\, dx \right)\\
        &&\quad =E(u,B\cap N)+\frac{C}{m}E(u,B\cap N)=\Big(1+\frac{C}{m}\Big)E^s(u,B),
   \end{eqnarray*}
   which proves the claim.
   \end{proof}

   We now introduce the  perturbed functionals $E_\delta$ which will play a  fundamental role in the proof of our representation result for functionals in $\E$.
\begin{definition}\label{def:Eepsilon}
    Let  $E\in\E$ and let $A\in\mathcal{A}_c(\Rd)$ be given.  For every $\delta>0$ the functional $E_{\delta}\colon BV(A;\Rk)\times \mathcal{B}(A)\to[0,+\infty)$ is  defined for every $u\in BV(A;\Rk)$ and $B\in\mathcal{B}(A)$ as 
    \begin{equation}\label{eq;def Eepsilon}
        E_{\delta}(u,B):=E(u,B)+\delta|Du|(B).
    \end{equation}
    Given $u\in BV(A;\Rk)$, the measures $E_{\delta}^a(u,\cdot)$, $E^s_{\delta}(u,\cdot)$, $E^j_{\delta}(u,\cdot)$, and $E_{\delta}^c(u,\cdot)$ are defined as in Definition \ref{def:splitting of the functionals} with $E$ replaced by $E_{\delta}$.
\end{definition}
\begin{remark}
    Let $E\in \E$, $A\in\mathcal{A}_c(\Rd)$, and $\delta>0$. Thanks to (c1$'$) and (c2$'$) of Remark \ref{re:new c1 and c2} we have that for every $u\in BV(A;\Rk)$ it holds
    \begin{equation*}
        \delta|Du|(A)-c_2\Ld(A)\leq E_\delta(u,A)\leq (c_3+\delta)|Du|(A)+c_4\Ld(A).
    \end{equation*}
\end{remark}

\begin{definition}
    Let $A\in\mathcal{A}_c(\Rd)$ and let $E\colon BV(A;\Rk)\times\mathcal{B}(A)\to[0,+\infty]$. For every  $U\in\mathcal{A}(A)$ with Lipschitz boundary and $w\in BV(U;\Rk)$, we set
    \begin{equation}\label{eq:problema di minimo ausiliario}
        m^E(w,U):=\inf\{E(u,U):\,\,u\in BV(U;\Rk),\,\, \textup{tr}_Uu=\textup{tr}_Uw\},
    \end{equation}
    where $\text{tr}_U\colon BV(U;\Rk)\to L^1_{\hd}(\partial U;\Rk)$ is the trace operator.
    Given $t>0$, we set
    \begin{equation}
   \label{eq:problema di minimo ausiliario troncato}
      m^E_t(w,U):=\inf\{E(u,U):\,\,u\in BV(U;\Rk),\,\, \textup{tr}_Uu=\textup{tr}_Uw,\,\,\|u-w\|_{L^\infty(U;\Rk)}\leq t\}.
    \end{equation}
\end{definition}

We now introduce some functions which will play a crucial role in the integral representation of the bulk part $E^a$ and the surface part $E^j$ of a functional $E\in \E_{\rm sc}$. We recall that the cubes $Q(x,\rho)$ and $Q_\nu(x,\rho)$ are defined in (f) of Section 2, while  the functions $\ell_\xi$ and $u_{x,\zeta,\nu}$ are defined in (h) of Section 2.
\begin{definition}
Let $E\in\E$ and $\delta>0$. We define the functions $f,f_\delta \colon\Rd\times\Rdk\to [0,+\infty)$ and  $g,g_\delta\colon \Rd\times\Rk\times\mathbb{S}^{d-1}\to [0,+\infty)$ as 
\begin{eqnarray}
\hspace{-1.3 cm} &&\displaystyle \label{eq:definition of small f}f(x,\xi):=\limsup_{\rho\to0^+}\frac{m^{E}(\ell_\xi,Q(x,\rho))}{\rho^d} \quad \text{for every }x\in\Rd \text{ and }\xi\in\Rdk,\\
    \hspace{-1.3 cm}&&\displaystyle \label{eq:definition of small fepsilon} f_\delta(x,\xi):=\limsup_{\rho\to0^+}\frac{m^{E_\delta}(\ell_\xi,Q(x,\rho))}{\rho^d} \quad \text{for every }x\in\Rd\text{ and }\xi\in\Rdk,\\
   \hspace{-1.3 cm} &&\displaystyle \label{eq:definition of small g} g(x,\zeta,\nu):=\limsup_{\rho\to0^+}\frac{m^{E}(u_{x,\zeta,\nu},Q_\nu(x,\rho))}{\rho^{d-1}}\,\,\,\,\text{ for every }x\in\Rd, \,\zeta\in\Rk, \text{ and }\nu\in\mathbb{S}^{d-1},\\
  \hspace{-1.3 cm}   &&\displaystyle  \label{eq:definition of small gepsilon} g_\delta(x,\zeta,\nu):=\limsup_{\rho\to0^+}\frac{m^{E_\delta}(u_{x,\zeta,\nu},Q_\nu(x,\rho))}{\rho^{d-1}}\,\,\,\,\text{ for every }x\in\Rd, \,\zeta\in\Rk, \text{ and }\in\mathbb{S}^{d-1}.
\end{eqnarray}
\end{definition}

\begin{remark}\label{re:monotonicity of fepsilon, gepsilon}
    It is immediate to check that the functions $\delta\mapsto f_\delta(x,\xi)$ and $\delta\mapsto g_\delta(x,\zeta,\nu)$ are non-decreasing and that
     \begin{equation*}
         f(x,\xi)\leq f_\delta(x,\xi) \,\,\,\text{ and }\,\,\,g(x,\zeta,\nu)\leq g_\delta(x,\zeta,\nu)
     \end{equation*}
     for every $\delta>0$, $x\in\Rd$, $\xi\in\Rkd$, $\zeta\in\Rk$, and $\nu\in\mathbb{S}^{d-1}$.
\end{remark}

When $E_{\rm sc}$, we shall see that the functions $f_\delta$ and $g_\delta$ will be used in the integral representation of $E^a_\delta$ and $E_\delta^j$ thanks to the results of \cite{bouchitte1998global}. This will lead to an integral representation of $E^a$ and $E^j$ by means of the functions (see the proof of Theorem \ref{thm:rappresentazione})
\begin{eqnarray}
    &\displaystyle \label{eq:def fhat} \hat{f}(x,\xi):=\inf_{\delta>0}f_\delta(x,\xi)\quad \text{for every $x\in\Rd,$ $\xi\in\Rkd$,} \\
    &\displaystyle \nonumber \hat{g}(x,\zeta,\nu):=\inf_{\delta>0}g_\delta(x,\zeta,\nu)\quad \text{for every $x\in\Rd,$ $\zeta\in \Rk$, $\nu\in \mathbb{S}^{d-1}$.}
\end{eqnarray}
For the applications to homogenisation it is important to prove that $\hat{f}=f$ and $\hat{g}=g$, so that by \eqref{eq:definition of small f} and \eqref{eq:definition of small g} the integrands used in the bulk and surface part can be obtained by solving some auxiliary minum problems on small cubes. 

The proof of the equality $\hat{f}=f$ is not direct and requires a lot of technical tools, one of which being the following truncation lemma.  For future use we prove the result also for the rectangles $Q^\lambda_\nu(x,\rho)$, defined in (f) of Section 2. Given $\xi\in\Rkd$ and  $m\in\N$, we set
\begin{equation}
\label{eq:def costante cxi}  
 c_{\xi,m}:=(\sigma^m+1)d^{1/2}\Big(|\xi|+\frac{c_2+c_4+1}{c_1}\Big)
\end{equation}
\begin{lemma}\label{lemma:teo troncature}
    Let $E\in\E$. Assume that there exists a function $\hat{f}\in\mathcal{F}$ such that
    \begin{equation}\label{eq: Ea assoltamente continuo fittizio}
        E^a(u,A)=\int_A\hat{f}(x,\nabla u)\, dx,
    \end{equation}
    for every $A\in\mathcal{A}_c(\Rd)$ and $u\in BV(A;\Rk)$.
    Then there exists a set $N\in\mathcal{B}(\Rd)$, with $\Ld(N)=0$, satisfying the following property: for every $x\in\Rd\setminus N$, $m\in \N$,  $\xi\in\Rkd$,  $\nu\in\mathbb{S}^{d-1}$, and $\lambda\geq 1$ there exists $\rho^{\lambda,\nu}_{m,\xi}(x)>0$  such that for every $\rho\in (0,\rho^{\lambda,\nu}_{m,\xi}(x))$ and  $u\in BV(Q^{\lambda}_\nu(x,\rho);\Rk)$, with $\text{tr}_{Q^\lambda_\nu(x,\rho)}u=\text{tr}_{Q^\lambda_\nu(x,\rho)}\ell_\xi$,  
    there exists $v\in BV(Q^\lambda_\nu(x,\rho);\Rk)$, with $\text{tr}_{Q^\lambda_\nu(x,\rho)}v=\text{tr}_{Q^\lambda_\nu(x,\rho)}\ell_\xi$ and $\|v-\ell_\xi\|_{L^\infty(Q_\nu^\lambda(x,\rho);\Rk)}\leq c_{\xi,m}\lambda\rho$, satisfying the inequality
    \begin{equation}
    \label{eq:necessaria per troncatura ld}
       \displaystyle  E(v,Q^\lambda_\nu(x,\rho))\leq \Big(1+\frac{C}{m}\Big) E(u,Q^\lambda_\nu(x,\rho))+\frac{C}{m}\lambda^{d-1}\rho^d,
    \end{equation}
    where $C$ is a constant depending only on the structural constants $c_1,c_2,c_3,c_4, c_5$, and $k$. Moreover, if $f$ is continuous on $\Rd\times\Rkd$, then $N=\varnothing$. Finally, if $\hat{f}$ is independent of $x$, then $\rho^{\lambda,\nu}_{m,\xi}(x)=+\infty$.
\end{lemma}

\begin{proof}
    Let us fix $m\in\N$. Arguing as in \cite[Lemma 4.16]{DalToa23}, one can construct a Borel function $\omega_m:\Rd\to\Rdk$ such that 
\begin{equation}\label{eq: def omega m}
    f(x,\omega_m(x))\leq f(x,\xi)+\frac{1}{m}
\end{equation}
for every $x\in\Rd$ and every  $\xi\in\Rkd$.
Note that (f2$'$), (f3), and \eqref{eq: def omega m} imply that for every  $x\in\Rd$ one has 
\begin{equation}
    \label{eq:bounds on sigma eta}
    |\omega_m(x)|\leq \frac{c_2+c_4+1}{c_1}.
\end{equation}
Since $\omega_m\in L^\infty(\Rd;\Rdk)$, by the Lebesgue Differentiation Theorem there exists $N_m\in\mathcal{B}(\Rd)$, with $\Ld(N_m)=0$, such that for every $x\in\Rd\setminus N_m$, $\nu\in\mathbb{S}^{d-1}$, and $\lambda\geq 1$ there exists $\rho^{\lambda,\nu}_{m,\xi}(x)>0$ such that for every $\rho\in(0,\rho^{\lambda,\nu}_{m,\xi}(x))$ we have
\begin{equation}\nonumber 
   \int_{Q^\lambda_\nu(x,\rho)}|\omega_m(x)-\omega_m(y)|\,dy\leq \frac{1}{m}\rho^d\lambda^{d-1}.
\end{equation}
We set $N:=\bigcup_{m\in\N}N_m$ and note that $\Ld(N)=0$.
Using (f5) we get that for every $x\in\Rd\setminus N$, $m\in \N$,  $\xi\in\Rkd$,  $\nu\in\mathbb{S}^{d-1}$ and $\rho\in(0,\rho^{\lambda,\nu}_{m,\xi}(x))$ we  have
\begin{equation}\label{eq:lebesgue differentiation}
   \int_{Q^\lambda_\nu(x,\rho)}|\hat{f}(y,\omega_m(x))-\hat{f}(y,\omega_m(y))|\,dy\leq \frac{c_5}{m}\rho^d\lambda^{d-1}.
\end{equation}

  For every $y\in Q^\lambda_\nu(x,\rho)$ and $i\in\{1,...,m\}$  we set 
  \begin{eqnarray}
      &\displaystyle \nonumber w(y):=\omega_m(x)(y-x)+\ell_\xi(x), \\
      &\displaystyle \label{eq:def oggetti lemma tronc}R:=d^{1/2}\Big(|\xi|+\frac{c_2+c_4+1}{c_1}\Big)\lambda\rho,\\
      &\displaystyle \nonumber v^i(y):=w(y)+\psi_R^i (u(y)-w(y)).
  \end{eqnarray}
  Note that with this choice of $R$ we have $|\text{tr}_{Q^\lambda_\nu(x,\rho)}(u-w)|\leq R$  $\hd$-a.e. on $ \partial Q^{\lambda}_\nu(x,\rho)$. 
  Recalling \eqref{eq:properties psiRi}, we obtain  $\text{tr}_{Q_\nu^\lambda(x,\rho)}v^i=\text{tr}_{Q_\nu^\lambda(x,\rho)}\ell_\xi$ and  
  \begin{equation*}
    \|v^i-\ell_\xi\|_{L^\infty(Q^\lambda_\nu(x,\rho);\Rk)}\leq (\sigma^m+1) R=c_{\xi,m}\lambda\rho.
  \end{equation*}

  We claim that there exists $i\in\{1,...,m\}$ such that \eqref{eq:necessaria per troncatura ld} holds with $v=v^i.$ To this aim, for every $i\in\{1,...,m\}$ we consider the following partition of $Q^\lambda_\nu(x,\rho)$:
\begin{align*}
   & Q^i_{\rm in}:=\{y\in Q^\lambda_\nu(x,\rho)\colon |u(y)-w(y)|\leq \sigma^{i-1}R\},\\
   & Q^i:=\{y\in Q^\lambda_\nu(x,\rho)\colon \sigma^{i-1}R< |u(y)-w(y)|<\sigma^iR\},\\
   & Q^i_{\rm out}:=\{y\in Q^\lambda_\nu(x,\rho)\colon |u(y)-w(y)|\geq \sigma^iR \}.
\end{align*}
Using (b) of Definition \ref{def:space of functionals E} we may  write
\begin{equation*}
E(v^i,Q^\lambda_\nu(x,\rho))=E(v^i,Q^i_{\rm in})+E(v^i,Q^i)+E(v^i,Q^i_{\rm out}).
\end{equation*}
Recalling \eqref{eq:properties psiRi} and \eqref{eq:def oggetti lemma tronc}, from  \eqref{eq: Ea assoltamente continuo fittizio} we obtain
\begin{equation*}
E^a(v^i,Q^\lambda_\nu(x,\rho))=\int_{Q^i_{\rm in}}\hat{f}(y,\nabla u(y))\,dy
 +\int_{Q^i}\hat{f}(y,\nabla v^i(y))\,dy+\int_{Q^i_{\rm out}}\hat{f}(y,\omega_m(x))\,dy.
\end{equation*}
Taking advantage of \eqref{eq:bounds on sigma eta}-\eqref{eq:def oggetti lemma tronc}, by (f2$'$) and (f3) we get 
\begin{eqnarray*}
    \int_{Q^i}\hat{f}(y,\nabla v^i(y))\, dy&\leq&c_3k^{1/2}\int_{Q^i}|\omega_m(x)+\nabla \psi^i_R(u(y)-w(y))(\nabla u(y)-\omega_m(x))|\, dy+c_4\Ld(Q^i)\\
   &\leq& c_3k^{1/2}\int_{Q^i}|\nabla u(y)|\, dy+(2c_3k^{1/2}|\omega_m(x)|+c_4)\Ld(Q^i)
   \\&\leq &
   \frac{c_3k^{1/2}}{c_1}\int_{Q^i}\hat{f}(y,\nabla u(y))\,dy+\Big(2c_3k^{1/2}\frac{c_2+c_4+1}{c_1}+c_2+c_4\Big)\Ld(Q^i).
\end{eqnarray*}
By \eqref{eq: def omega m} and \eqref{eq:lebesgue differentiation} we have
\begin{equation*}
    \int_{Q^i_{\rm out}}\hat{f}(y,\omega_m(x))\,dy\leq \int_{Q^i_{\rm out}}\hat{f}(y,\omega(y))\,dy +\frac{c_5}{m}\lambda^{d-1}\rho^d\leq \int_{Q^i_{\rm out}}\hat{f}(y,\nabla u(y))\,dy+\frac{c_5+1}{m}\lambda^{d-1}\rho^d.
\end{equation*}
 From the previous inequalities, we get 
 \begin{eqnarray*}
E^a(v^i,Q^\lambda_\nu(x,\rho))\leq \int_{Q^i_{\rm in}}\hat{f}(y,\nabla u(y))\,dy+ \frac{c_3k^{1/2}}{c_1}\int_{Q^i}\hat{f}(y,\nabla u(y))\,dy\\+C_1\Ld(Q^i) +\int_{Q^i_{\rm out}}\hat{f}(y,\nabla u(y))\,dy+\frac{c_5+1}{m}\lambda^{d-1}\rho^d,
 \end{eqnarray*}
for some constant $C_1$ depending only on  $c_1,c_2,c_3,c_4,c_5$, and $k$.
\noindent Summing over $i\in\{1,...,m\}$ and dividing by $m$, we obtain 
\begin{equation}\label{eq:lemmazz}
    \frac{1}{m}\sum_{i=1}^mE^a(v^i,Q^\lambda_\nu(x,\rho))\leq\Big(1+\frac{C}{m}\Big) E^a(u,Q^\lambda_\nu(x,\rho))+\frac{C}{m}
    \lambda^{d-1}\rho^d,
\end{equation}
for some constant $C$ depending only on  $c_1,c_2,c_3,c_4,c_5$, and  $k$, which we may assume to be larger than the constant of Lemma \ref{lemma:estimates on the singular part}. Finally, using that lemma, we get 
\begin{equation}\label{eq:lemmazzz}
    \frac{1}{m}\sum_{i=1}^mE^s(v^i,Q^\lambda_\nu(x,\rho))\leq \Big(1+\frac{C}{m}\Big)E^s(u,Q^\lambda_\nu(x,\rho)),
\end{equation}
so that combining \eqref{eq:lemmazz} with \eqref{eq:lemmazzz} we obtain that there exists $i\in\{1,...,m\}$ such that \eqref{eq:necessaria per troncatura ld} holds with $v=v^i$. 

Suppose now that $\hat{f}$ is continuous in $\Rd\times\Rdk$. Thanks to (f2$'$) and to (f4), we have that for every $x\in\Rd$ there exists  $\omega(x)$ which minimizes the continuous function $\xi\mapsto\hat{f}(x,\xi)$. 

Since $\hat{f}\in\mathcal{F}$, we also have that 
\begin{equation*}
   |\omega(x)|\leq \frac{c_2+c_4}{c_1}.
\end{equation*}
This inequality, together with the uniform continuity of $\hat{f}$ on compact sets of $\Rd\times\Rdk$, implies that there exists a $\rho^{\lambda,\nu}_{m}(x)>0$ such that for every $\rho\in(0,\rho^{\lambda,\nu}_{m}(x))$ one has

\begin{equation}\label{eq:cambio in x}
    |\hat{f}(y,\xi)-\hat{f}(x,\xi)|\leq \frac{1}{m}
\end{equation}

\noindent for every $y\in Q^\lambda_\nu(x,\rho)$ and $|\xi|\leq  (c_2+c_4)/c_1$. In particular, we have
\begin{equation}\label{eq:scambio omega}
    |\hat{f}(y,\omega(x))-\hat{f}(x,\omega(x))|\leq \frac{1}{m}.
\end{equation}Exploiting these inequalities and the minimality of $\omega(x)$, we obtain that
\begin{equation}\label{eq:replacement of lebesgue}
    \hat{f}(y,\omega(x))\leq \hat{f}(y,\xi)+\frac{2}{m}
\end{equation}
for every $y\in Q^\lambda_\nu(x,\rho)$ and $\xi\in\Rdk$.
Indeed, if $|\xi|\leq (c_2+c_4)/c_1$, we apply  \eqref{eq:cambio in x} twice to get
$$ \hat{f}(y,\omega(x))\leq \hat{f}(x,\omega(x))+\frac{1}{m}\leq \hat{f}(x,\xi)+\frac{1}{m}\leq \hat{f}(y,\xi)+\frac{2}{m}.$$ If  $|\xi|> (c_4+c_2)/c_1$, recalling the minimality of $\omega(x)$, by (f3) and (f2$'$)  we have $\hat{f}(x,\omega(x))\leq\hat{f}(x,0)\leq c_4\leq \hat{f}(y,\xi) $, which, togheter with \eqref{eq:scambio omega}, implies \eqref{eq:replacement of lebesgue}. As in the previous part of the proof, for every $y\in\Rd$ we denote $w(y):=\omega(x)(y-x)+\ell_\xi(x)$.  We can replace \eqref{eq:lebesgue differentiation} by \eqref{eq:replacement of lebesgue} in the  argument that we used in the case where $\hat{f}$ was not assumed to be continuous on $\Rd\times\Rdk$ and this leads to the existence of $i\in\{1,...,m\}$ such that 
\begin{equation*}
    E(w+\psi^i_R\circ(u-w),Q^\lambda_\nu(x,\rho))\leq \Big(1+\frac{C}{m}\Big) E(u,Q^\lambda_\nu(x,\rho))+\frac{C}{m}
    \lambda^{d-1}\rho^d.
\end{equation*}

In the case where  $\hat{f}$ does not depend on $x$ the same is true for $\omega$ and  \eqref{eq:replacement of lebesgue} holds for every $\rho>0$, since it is a direct consequence of the minimality.
\end{proof}

The next result follows immediately from Lemma \ref{lemma:teo troncature}. Given $\xi\in\Rkd$, we set 
\begin{equation}\label{def: costante resti}
    C_\xi:=2c_3k^{1/2}C|\xi|+2C(c_4+1),
\end{equation}
where $C$ is the constant of Lemma \ref{lemma:teo troncature}.
 We also recall that the constant $c_{\xi,m}$ is given by \eqref{eq:def costante cxi}.

\begin{corollary}\label{cor:troncature}   
Let $E\in\E$ and assume that there exists a function $\hat{f}\in\mathcal{F}$ satisfying \eqref{eq: Ea assoltamente continuo fittizio}. Then there exists a  set $N\in\mathcal{B}(\Rd)$, with $\Ld(N)=0$, such that for every $x\in\Rd\setminus N$, $m\in\N$, $\xi\in\Rdk$, $\nu\in\mathbb{S}^{d-1}$, and $\lambda\geq 1$ there exists $\rho^{\nu,\lambda}_{m,\xi
 }(x)>0$ such that for every $\rho\in(0,\rho^{\nu,\lambda}_{m,\xi}(x))$ there exists $u\in BV(Q^\lambda_\nu(x,\rho);\Rk)\cap L^\infty(Q^\lambda_\nu(x,\rho);\Rk)$, with $\text{tr}_{Q_\nu^\lambda(x,\rho)}u={\text{tr}_{Q_\nu^\lambda(x,\rho)}}\ell_\xi$ and  $\|u-\ell_\xi\|_{L^\infty(Q^\lambda_\nu(x,\rho);\Rk)}\leq  c_{\xi,m}\lambda\rho$, such that
\begin{equation}\label{eq:stimare con minimi in norma linfinito limitata}
    E(u,Q^\lambda_\nu(x,\rho))\leq m^E(\ell_\xi,Q^\lambda_\nu(x,\rho))+\frac{C_\xi\lambda^{d-1}\rho^d}{m}.
\end{equation}
In particular for $t=c_{\xi,m}\lambda$, we have
\begin{equation}\label{eq:corollary claim}
m_{t\rho}^E(\ell_{\xi},Q^{\lambda}_\nu(x,\rho))\leq m^E(\ell_\xi,Q^\lambda_\nu(x,\rho))+\frac{C_\xi \lambda^{d-1}\rho^d}{m}.
\end{equation}

\noindent
Moreover, if $\hat{f}$ is continuous in $\Rd\times\Rdk$ then $N=\varnothing$. Finally, if $\hat{f}$ is independent on $x$ then $\rho^{\eta,\lambda}_{\xi,\nu}(x)=+\infty$.
\end{corollary}
\begin{proof}
   Let $N\in\mathcal{B}(\Rd)$ and $\rho^{\nu,\lambda}_{m,\xi}(x)>0$ be as in  Lemma \ref{lemma:teo troncature}. Consider a $v\in BV(Q^\lambda_\nu (x,\rho);\Rk)$, with $\text{tr}_{Q^\lambda_\nu(x,\rho)}v=\text{tr}_{Q^\lambda_\nu(x,\rho)}\ell_\xi$,  such that 
    \begin{equation*}
        E(v,Q^\lambda_\nu(x,\rho))\leq m^E(\ell_\xi,Q^\lambda_\nu(x,\rho))+\frac{\lambda^{d-1}\rho^d}{m}.
    \end{equation*}
    Since the function $\ell_\xi$ is a competitor for the minimisation problem in the right-hand side of the previous inequality, we also get
    \begin{equation*}
 E(v,Q^\lambda_\nu(x,\rho))\leq (c_3k^{1/2}|\xi|+c_4+1)\lambda^{d-1}\rho^d.
    \end{equation*}
    Thus, we may apply Lemma \ref{lemma:teo troncature}  to obtain a function $u\in BV(Q^\lambda_\nu(x,\rho);\Rk)\cap L^\infty(Q^\lambda_\nu(x,\rho);\Rk)$, with $\text{tr}_{Q_\nu^\lambda(x,\rho)}u={\text{tr}_{Q_\nu^\lambda(x,\rho)}}\ell_\xi$ and  $\|u-\ell_\xi\|_{L^\infty(Q^\lambda_\nu(x,\rho);\Rk)}\leq c_{\xi,m}\lambda\rho$, such that \eqref{eq:stimare con minimi in norma linfinito limitata} holds, concluding the proof.
\end{proof}

 The following proposition shows that under the hypotheses of Lemma \ref{lemma:teo troncature}  the functions  $f_\delta $ and $f^\infty_\delta$ converge $\Ld$-a.e. as $\delta\to 0^+$ to $f$ and $f^\infty$, respectively.

\begin{proposition}\label{prop:fepsilon passes to the limit}
    Let $E\in\E$.  Assume that there exists $\hat{f}\in\mathcal{F}$ such that \eqref{eq: Ea assoltamente continuo fittizio} holds. Then
    \begin{eqnarray}
    & \displaystyle \label{eq: pointwise limit for f }f(x,\xi)=\lim_{\delta\to 0^+} f_\delta(x,\xi)\quad {\rm \, for\,\,     } \Ld\meno {\rm a.e}.\,  x\in\Rd {\rm\,\, and \,\,for\,\, every     }\,\xi\in\Rkd,\\
    &\displaystyle \label{eq: pointwise limit for finfty} f^\infty(x,\xi)=\lim_{\delta\to 0^+} f^\infty_\delta(x,\xi)\quad {\rm for\,\,     } \Ld\meno {\rm a.e.}\,  x\in\Rd {\rm\,\,and  \,\,for\,\, every     }\,\xi\in\Rkd.
    \end{eqnarray}
\end{proposition}
    \begin{proof}
        From Remark \ref{re:monotonicity of fepsilon, gepsilon}, we deduce that the limits in the right-hand side of \eqref{eq: pointwise limit for f }  and of \eqref{eq: pointwise limit for finfty} exist and that 
        \begin{eqnarray*}
            &\displaystyle f(x,\xi)\leq \inf_{\delta>0} f_\delta(x,\xi),\\
            &\displaystyle f^\infty(x,\xi)\leq \inf_{\delta>0}  f^\infty_\delta(x,\xi).
        \end{eqnarray*}

       We are left with proving that the converse inequality also holds. Let $N\in\mathcal{B}(\Rd)$ the $\Ld$-negligible set of Corollary \ref{cor:troncature}, let $ x\in\Rd\setminus N$, $m\in\N$, $\xi\in\Rdk$, and  $t\geq 1$.
        By \eqref{eq:definition of small f}, for $\rho>0$ small enough we have
        \begin{equation}\label{eq:bound on f}
        \frac{m^E(t\ell_{\xi},Q(x,\rho))}{t\rho^d}\leq \frac{f(x,t\xi)}{t}+\frac{1}{tm}.
        \end{equation}
        We can now apply Corollary \ref{cor:troncature}, with $\xi$ replaced by $t\xi$, to obtain a function $u\in BV(Q(x,\rho);\Rk)$, with $\text{tr}_{Q(x,\rho)}u=\text{tr}_{Q(x,\rho)}t\ell_\xi$ and  $\|u-t\ell_{\xi}\|_{L^\infty(Q(x,\rho);\Rk)}\leq c_{t\xi,m}\rho$, such that \eqref{eq:stimare con minimi in norma linfinito limitata} holds with $\xi$ replaced by $t\xi$. By taking $\rho>0$ small enough, we may also suppose that $\|u-t\ell_{\xi}\|_{L^\infty(Q(x,\rho);\Rk)}\leq 1/2$.  Inequality \eqref{eq:bound on f} then yields
        \begin{equation*}
\frac{E(u,Q(x,\rho))}{t\rho^d}\leq \frac{f(x,t\xi)}{t}+\frac{C_{t\xi}+1}{tm}\leq \frac{f(x,t\xi)}{t}+\frac{C_\xi+1}{m},
        \end{equation*}
       where $C_{\xi}>0$ is the constant given by \eqref{def: costante resti}. Note that in the last inequality we have used the estimate $C_{t\xi}\leq tC_\xi$ for every $t\geq 1$.
       
        We now compare $E_\delta(u,Q(x,\rho))$ and $E(u,Q(x,\rho))$. Since $\|u-t\ell_{\xi}\|_{L^\infty(Q(x,\rho);\Rk)}\leq 1/2$, we have $|[u]|\leq 1$ $\hd$-a.e. on $J_u$. Hence, by (c1$'$) we get 
\begin{eqnarray*}
E_\delta(u,Q(x,\rho))&=& E(u,Q(x,\rho))+\delta|Du|(Q(x,\rho))\\ &\leq& E(u,Q(x,\rho))+\delta \int_{Q(x,\rho)}|\nabla u| \, dx+\delta|D^cu|(Q(x,\rho))
+\delta \int_{\jump{u}}|[u]|\land 1\, d\hd\\
&\leq& \Big(1+\frac{\delta}{c_1}\Big)E(u,Q(x,\rho))+\frac{c_2\delta }{c_1}\rho^d.
\end{eqnarray*}
Letting $\delta$ be so small that $\delta /c_1\leq 1/m$ and $\delta c_2/c_1\leq 1/m,$ we then obtain 
\[
\frac{E_\delta(u,Q(x,\rho))}{t\rho^d}\leq\Big(1+\frac1m\Big)\frac{E(u,Q(x,\rho))}{t\rho^d}+\frac1m\leq \Big(1+\frac1m\Big)\Big(\frac{f(x,t\xi)}{t}+\frac{C_\xi+1}{m}\Big)
\]and, recalling that $\text{tr}_{Q(x,\rho)}u=\text{tr}_{Q(x,\rho)}t\ell_\xi$, we get
\[
\frac{m^{E_\delta}(\ell_{t\xi},Q(x,\rho))}{t\rho^d}\leq \Big(1+\frac1m\Big)\Big(\frac{f(x,t\xi)}{t}+\frac{C_\xi+1}{m}\Big).
\]
Evaluating this last inequality at $t=1$ and taking the $\limsup$ as $\rho \to 0^+$, we deduce
\begin{equation}\nonumber
    f_\delta(x,\xi)\leq \Big(1+\frac1m\Big)\Big(f(x,\xi)+\frac{C_\xi+1}{m}\Big),
\end{equation}
while taking  the $\limsup$ as $\rho\to 0^+$ first and letting then $t\to+\infty$, we obtain
\begin{equation}\nonumber
f^\infty_\delta(x,\xi)\leq \Big(1+\frac1m\Big)\Big(f^\infty(x,\xi)+\frac{C_\xi+1}{m}\Big).
\end{equation}
Finally, letting $\delta\to 0^+$ first and then taking the limit for $m\to+\infty$, we conclude the proof.
    \end{proof}

We now pass to the study of the minimisation problems used to define the integrands for the surface terms.
\begin{lemma}\label{cor:Corollario Troncature Hd}
    Let $E\in\E$,   
     $x\in\Rd$,  $\nu\in\mathbb{S}^{d-1}$,   $m\in\N$, $\zeta\in\Rk$, and $\rho>0$. Then there exists $u\in BV(Q_\nu(x,\rho);\Rk)\cap L^\infty(Q_\nu(x,\rho);\Rk)$, with $\text{tr}_{Q_\nu(x,\rho)}u=\text{tr}_{Q_\nu(x,\rho)}u_{x,\zeta,\nu}$ and $\|u\|_{L^\infty(Q_\nu(x,\rho);\Rk)}\leq \sigma^m$, such that
    \begin{equation}\label{eq:truncation hd}
         E(u,Q_\nu(x,\rho))\leq m^E(u_{x,\zeta,\nu},Q_\nu(x,\rho))+\frac{K(|\zeta|\land 1)}{m}\rho^{d-1}+K\rho^{d},
    \end{equation}
    where $K:=\max\{c_3kC,(C+c_4+1)\}$ and $C$ is the constant of (g) of Definition \ref{def:space of functionals E}.
\end{lemma}
\begin{proof}
     By \eqref{eq:problema di minimo ausiliario} there exists  $v\in BV(Q_\nu(x,\rho);\Rk)$, with $\text{tr}_{Q_\nu(x,\rho)}u=\text{tr}_{Q_\nu(x,\rho)}u_{x,\zeta,\nu}$, such that
    \begin{equation}\label{eq:passaggio intermedio troncatura hd}
E(v,Q_\nu(x,\rho))\leq m^E(u_{x,\zeta,\nu},Q_\nu(x,\rho))+
\rho^{d}\leq c_3k(|\zeta|\land 1)\rho^{d-1}+\rho^d,
    \end{equation}
    where the second inequality follows from (c2$'$) and the fact that $u_{x,\zeta,\nu}$ is a competitor for the minimisation problem.

    Let us fix $R\geq|\zeta|$. Note that with this choice of $R$ by \eqref{eq:properties psiRi} we have $\text{tr}_{Q_\nu(x,\rho)}u=\text{tr}_{Q_\nu(x,\rho)}u_{x,\zeta,\nu}$ and $\|u\|_{L^\infty(Q_\nu(x,\rho);\Rk)}\leq \sigma^m$. By property (g)  of Definition \ref{def:space of functionals E}  and \eqref{eq:passaggio intermedio troncatura hd} there exists $i\in\{1,...,m\}$ such that $u:=\psi_R^i\circ v$ satisfies
    \begin{eqnarray*}
E(u,Q_\nu(x,\rho))&\leq& E(v,Q_\nu(x,\rho)) \, dx+c_4\rho^d  +\frac{C}{m}\left(E(v,Q_\nu(x,\rho))+\rho^d\right)\\
        &\leq&  m^E(u_{x,\zeta,\nu},Q_\nu(x,\rho))+(C+c_4+1)\rho^d+\frac{c_3kC(|\zeta|\land 1) }{m}\rho^{d-1},
    \end{eqnarray*}
    concluding the proof.
\end{proof}

We are now ready to prove that $g_\delta$ converges to $g$.

\begin{proposition}\label{prop:limite in gepsilon}
    Let $E\in\E$ , $ x\in\Rd,$ $\zeta\in\Rk,$ and  $\nu\in\mathbb{S}^{d-1}$. Then 
    \begin{equation}\label{eq:limite per g}
        g(x,\zeta,\nu)=\lim_{\delta\to 0^+}g_\delta(x,\zeta,\nu).
    \end{equation}
\end{proposition}
\begin{proof}
    Thanks to Remark \ref{re:monotonicity of fepsilon, gepsilon}, the limit in the right-hand side of \eqref{eq:limite per g} exists and 
    \begin{equation}\nonumber 
        g(x,\zeta,\nu)\leq \lim_{\delta\to 0^+}g_\delta(x,\zeta,\nu).
    \end{equation}

    We now prove that the converse inequality holds true as well. Let us fix $m\in\N$.   By \eqref{eq:definition of small g} for  $\rho>0$ small enough we have that
    \begin{equation}\label{eq:eila}
        \frac{m^E(u_{x,\zeta,\nu},Q_\nu(x,\rho))}{\rho^{d-1}}\leq g(x,\zeta,\nu)+\frac{1}{m}.
    \end{equation}
    We can now apply Lemma \ref{cor:Corollario Troncature Hd} to obtain a function  $u\in BV(Q_\nu(x,\rho);\Rk)\cap L^\infty(Q_\nu(x,\rho);\Rk)$, with $\text{tr}_{Q_\nu(x,\rho)}u=\text{tr}_{Q_\nu(x,\rho)}u_{x,\zeta,\nu}$ and  
    $\|u\|_{L^\infty(Q_\nu(x,\rho):\Rk)}\leq \sigma^m$, such that
    \begin{equation*}
        E(u,Q_\nu(x,\rho))\leq m^E(u_{x,\zeta,\nu}.Q_\nu(x,\rho))+\frac{K}{m}\rho^{d-1}+K\rho^d.\end{equation*}
        From \eqref{eq:eila} we then  deduce that
    \begin{equation*}
\frac{E(u,Q_\nu(x,\rho))}{\rho^{d-1}}\leq g(x,\zeta,\nu)+\frac{K+1}{m}+K\rho.
    \end{equation*}
    
To conclude, we compare $E_\delta(u,Q_\nu(x,\rho))$ with $E(u,Q_\nu(x,\rho))$. Since $\|u\|_{L^\infty(Q_\nu(x,\rho);\Rk)}\leq \sigma^m$, we have $|[u]|\leq 2\sigma^m|[u]|\land 1.$ Hence, we get 
    \begin{eqnarray*}
       E_\delta(u,Q_\nu(x,\rho))&=& E(u,Q_\nu(x,\rho))+\delta|Du|(Q_\nu(x,\rho))
       \\ &\leq&  E(u,Q_\nu(x,\rho))+\delta\int_{Q_\nu(x,\rho)}\!\!\!\!\!|\nabla u|\, dx+\delta|D^cu|(Q_\nu(x,\rho))+2\delta \sigma^m\int_{\jump{u}}\!\!\!\!\!|[u]|\land 1\, \hd
       \\ &\leq& \Big(1+\frac{2\delta \sigma^m}{c_1}\Big)E(u,Q_\nu(x,\rho))+\frac{2c_2\delta\sigma^m}{c_1}\rho^d.  
    \end{eqnarray*}
    
    \noindent Letting $\delta$  be so small that $2\delta 
    \sigma^m/c_1\leq 1/m$ and $2c_2\delta 
    \sigma^m/c_1\leq 1$, we get
    \begin{equation*}
        \frac{E_\delta(u,Q_\nu(x,\rho))}{\rho^{d-1}}\leq \Big(1+\frac1m\Big) g(x,\zeta,\nu)+\frac{K+1}{m}+\Big(1+K\Big)\rho.
    \end{equation*}
    Recalling the definition of $m^{E_\delta}$, from this last estimate  we deduce that
    \begin{equation*}
         \frac{m^{E_\delta}(u_{x,\zeta,\nu},Q_\nu(x,\rho))}{\rho^{d-1}}\leq \Big(1+\frac1m\Big) g(x,\zeta,\nu)+\frac{K+1}{m}+\Big(1+K\Big)\rho.
    \end{equation*}
    Taking the limsup as  $\rho\to 0^+$ , we obtain
    \begin{equation*}
    g_\delta(x,\zeta, \nu)\leq \Big(1+\frac1m\Big) g(x,\zeta,\nu)+\frac{K+1}{m}.
    \end{equation*}
    Taking the $\limsup$ for $\delta\to0^+$ first, and the limit for $m\to+\infty$ then, 
    we conclude the proof.
\end{proof}

As the following proposition shows, the functions $f$ and $g$ defined by \eqref{eq:definition of small f} and \eqref{eq:definition of small g}, belong to $\mathcal{F}$ and to $\mathcal{G}$, respectively. 
\begin{proposition}\label{prop:properties of f and g}
    Let $E\in\E_{\rm sc}$,  $f$, $\hat{f}$, and $g$ be the functions defined by \eqref{eq:definition of small f}, \eqref{eq:def fhat}, and \eqref{eq:definition of small g}, respectively. Then $\hat{f}, f\in\mathcal{F}$, $g\in\mathcal{G}$.
\end{proposition}
\begin{proof}
    The proof of the inclusions  $\hat{f}$, $f\in\mathcal{F}$ and of properties  (g1)-(g4) for $g$, can be obtained by adapting the arguments of \cite[Section 5]{DalToa23}, with minor changes.
    
We are left with proving that $g$ satisfies (g5). To see this, let $\zeta_1,\zeta_2\in\Rk\setminus\{0\}$ with $c_6k|\zeta_1|  \leq |\zeta_2|$. We set $\lambda:=|\zeta_1|/|\zeta_2|$ and note that $\lambda\leq 1/(c_6k)$. Let $R\in SO(k)$ be a rotation that maps $\lambda \zeta_2$ to $\zeta_1$. Let $x\in\Rd$, $\nu\in\mathbb{S}^{d-1}$, and  $m\in\N$. By \eqref{eq:definition of small g} there exists $u\in BV(Q_\nu(x,\rho));\Rk)$, with $\text{tr}_{Q_\nu(x,\rho)}u=\text{tr}_{Q_\nu(x,\rho)}u_{x,\zeta_2,\nu}$, such that 
\begin{equation}\label{eq:g5 proof}
E(u,Q_\nu(x,\rho))\leq m^E(u_{x,\zeta_2,\nu},Q_\nu(x,\rho))+\rho^d.    
\end{equation}
We set $v:=\lambda Ru$ and note that $\text{tr}_{Q_\nu(x,\rho)}v=\text{tr}_{Q_\nu(x,\rho)}u_{x,\zeta_1,\nu}$. Then, by means of (h) of Definition \ref{def:space of functionals E}, we have
\begin{equation*}\label{eq:g5 proof 2}
    E(v,Q_\nu(x,\rho))\leq E(u,Q_\nu(x,\rho))+(c_4+c_2)\rho^d,
\end{equation*}
so that, by \eqref{eq:g5 proof}, we infer
\begin{equation}\label{eq:g5 proof 3}
   E(v,Q_\nu(x,\rho))\leq m^E(u_{x,\zeta_2,\nu},Q_\nu(x,\rho))+(c_4+c_2+1)\rho^d.
\end{equation}
Dividing this last inequality by $\rho^{d-1}$ and taking the $\limsup$ for ${\rho\to 0^+}$, we obtain
\begin{equation*}
    \limsup_{\rho\to 0^+} \frac{E(v,Q_\nu(x,\rho))}{\rho^{d-1}}\leq g(x,\zeta_2,\nu),
\end{equation*}
which, in light of the fact that $\text{tr}_{Q_\nu(x,\rho)}v=\text{tr}_{Q_\nu(x,\rho)}u_{x,\zeta_1,\nu}$, implies 
\begin{equation*}
    g(x,\zeta_1,\nu)\leq g(x,\zeta_2,\nu),
\end{equation*}
which concludes the proof.
\end{proof}

Using the results of \cite{bouchitte1998global}, we can now establish an integral representation both for the absolutely continuous part and the jump part of the perturbed functionals $E_\delta$. This representation is achieved by means of the functions $f_\delta$ and $g_\delta$ defined above.

\begin{proposition}\label{prop:partial representation}
   Let $E\in\E_{\rm sc}$,  $A\in\mathcal{A}_c(\Rd)$, and  $\delta>0$. Let $E_\delta^a$ and $E^j_\delta$ be the functionals introduced in Definition \ref{def:Eepsilon} and let $f_\delta$ and $g_\delta$ be the functions defined by \eqref{eq:definition of small fepsilon} and \eqref{eq:definition of small gepsilon}, respectively.  Then
    \begin{eqnarray}\label{eq:representation of ac epsilon}
        &\displaystyle E^a_\delta(u,B)=\int_{B}f_\delta(x,\nabla u)\, dx,\\\label{eq:representation of jump epsilon}
        &\displaystyle E_\delta^j(u,B)=\int_{B\cap \jump{u}}g_\delta(x,[u],\nu)\, \hd,
    \end{eqnarray}
    for every $u\in BV(A;\Rk)$ and for every $B\in\mathcal{B}(A)$.
\end{proposition}
\begin{proof}
    The result follows from \cite{bouchitte1998global}, using the same arguments of \cite[Theorem 6.1]{DalToa23}. 
\end{proof}

With this proposition at hand, we are ready to prove the integral representation of $E^a$ and of $E^j$.

\begin{theorem}\label{thm:rappresentazione}
    Let $E\in\E_{\rm sc}$,  $A\in\mathcal{A}_c(\Rd), $ and let $f$ and $g$ be the functions defined by \eqref{eq:definition of small f} and \eqref{eq:definition of small g}, respectively. Then
    \begin{eqnarray}\label{eq:rappresentazione ac}
 &\displaystyle E^a(u,B)=\int_{B}f(x,\nabla u)\, dx,\\
        &\displaystyle \label{eq:rappresentazione salto}E^j(u,B)=\int_{\jump{u}\cap B}g(x,[u],\nu)\, \hd
    \end{eqnarray}
    for every $u\in GBV_\star(A;\Rk)$ and for every $B\in\mathcal{B}(A)$.
\end{theorem}
\begin{proof}
    Let $\hat{f}$ be the function defined by \eqref{eq:def fhat}. We first show that 
    \begin{equation}\label{eq: ac intermedia}
        E^a(u,B)=\int_B\hat{f}(x,\nabla u)\, dx\,\,\, \text{ for every $u\in BV(A;\Rk)$ and $B\in\mathcal{B}(A)$.}
    \end{equation}
    
    To this aim,  we begin noting that $E^a_\delta(u,B)=E^a(u,B)+\delta\int_B|\nabla u|$ for every $u\in BV(A;\Rk)$ and for every $B\in\mathcal{B}(A)$, so that 
    \begin{equation*}
        E^a(u,B)=\inf_{\delta>0}E_\delta^a(u,B)=\lim_{\delta>0}E_\delta^a(u,B).
    \end{equation*}
   By definition of $f_\delta$, we have that
   \begin{equation*}
       f_\delta(x,\xi)\leq (c_3k^{1/2}+\delta)|\xi|+c_4,
   \end{equation*}
   for every $x\in\Rd$ and for every $\xi\in\Rkd$, so that by invoking the Dominate Convergence Theorem, 
   we obtain \eqref{eq: ac intermedia}.

  We now prove that 
  \begin{equation}\label{eq:rappresentazione salto sulle bv}
      E^j(u,B)=\int_{B}g(x,[u],\nu_u)\, d\hd\,\,\, \text{ for every $u\in BV(A;\Rk)$ and $B\in\mathcal{B}(A)$.}
  \end{equation}
  To show this, note that $E_\delta^j(u,B)=E^j(u,B)+\delta\int_{J_u\cap B} |[u]|d\hd$ for every $u\in BV(A;\Rk)$ and $B\in\mathcal{B}(A)$; therefore,
\begin{equation*}
      E^j(u,B)=\inf_{\delta>0}E_\delta^j(u,B)=\lim_{\delta>0}E_\delta^j(u,B).
\end{equation*}
   It is immediate to see that for every $x\in\Rd$, $\zeta\in\Rk$ and $\nu\in\mathbb{S}^{d-1}$  we have 
   \begin{equation*}
    g_\delta(x,\zeta,\nu)\leq (c_3k+\delta)|\zeta|\land 1.
   \end{equation*}
    Hence, recalling Proposition \ref{prop:limite in gepsilon} and using \eqref{eq:representation of jump epsilon},  by the Dominated Convergence Theorem we obtain \eqref{eq:rappresentazione salto}. 
 
 Consider now  $u\in GBV_\star(A;\Rk)$. Let $R_m>0$ with $R_m\to +\infty$. . By Lemma \ref{lemma:passaggio al limite dalle troncate al non nei vari pezzi},  both \eqref{eq: passage to the limit nella parte AC} and \eqref{eq:passage to the limit nella parte di salto} hold  for every $B\in\mathcal{B}(A)$.  To conclude, it is enough to show that
 \begin{eqnarray}
    &\displaystyle\label{eq:fine rapp 1} \lim_{m\to+\infty}\frac1m\sum_{i=1}^mE^a(\psi_{R_m}^{i}\circ u,A)=\int_Af(x,\nabla u)\,dx,\\
     &\displaystyle\label{eq: fine rapp 2} \lim_{m\to+\infty}\frac1m\sum_{i=1}^mE^j(\psi_{R_m}^{i}\circ u,A)=\int_{\jump{u}\cap A}g(x,[u],\nu_u)\,d\hd.
 \end{eqnarray}
 Since for every $i\in\{1,...m\}$ we have that $ \psi_{R_m}^{i}\circ u\in BV(A;\Rk)\cap L^\infty(A;\Rk)$, from \eqref{eq:rappresentazione ac} and \eqref{eq:rappresentazione salto sulle bv} we deduce that
 \begin{eqnarray*}
     &\displaystyle E^a(\psi_{R_m}^{i}\circ u,A)=\int_Af(x,\nabla (\psi_{R_m}^{i}\circ u))\, dx,\\
     &\displaystyle E^j(\psi_{R_m}^{i}\circ u,A)=\int_{\jump{\psi_{R_m}^{i}(u)}\cap A}\!\!\!g(x,[\psi_{R_m}^{i}\circ u],\nu_u)\, \hd.
 \end{eqnarray*}
 By \eqref{eq:properties psiRi} we have $|\nabla(\psi_{R_m}^{i}\circ u)|\leq |\nabla
 u|$ for every $m\in\N$ and $i\in\{1,...,m\}$. Morover, $\nabla u\in L^1(\Rd;\Rkd)$ by Proposition \ref{prop: Properties of GBVsvector}(a). For every $m\in\N$ there exists $i(m),j(m)\in\{1,...,m\}$ such that 
\begin{equation*}
    \int_Af(x,\nabla (\psi_{R_m}^{i(m)}\circ u))\, dx\leq\frac1m\sum_{i=1}^mE^a(\psi_{R_m}^{i}\circ u,A)\leq\int_Af(x,\nabla (\psi_{R_m}^{j(m)}\circ u))\, dx.
\end{equation*}
 We observe that both $\nabla(\psi_{R_m}^{i(m)}\circ u)$ and $\nabla(\psi_{R_m}^{j(m)}\circ u)$  converge to $\nabla u$ pointwise $\Ld$-a.e. in $A$ as $m\to+\infty$. Hence, recalling that  by Proposition \ref{prop:properties of f and g} the function $f$ belongs to $\mathcal{F}$, the Dominated Convergence Theorem implies \eqref{eq:fine rapp 1}.
 
 As for \eqref{eq: fine rapp 2}, by Proposition \ref{prop: Properties of GBVsvector}(d), for every $i\in\{1,...,m \}$ we have that $|[\psi_{R_m}^{i}\circ u]|\land1\leq |[u]|\land 1$ and that $[\psi_{R_m}^{i}\circ u](x)\to [u](x)$ for $\hd$-a.e. $x\in \jump{u}$. For every $m\in\N$ we choose $i(m),j(m)\in\{1,...,m\}$ such that
\begin{equation*}
    \int_{\jump{u}\cap A}g(x,[\psi_{R_m}^{i(m)}\circ u],\nu_u)\, d\hd\leq\frac1m\sum_{i=1}^mE^j(\psi_{R_m}^{i}\circ u,A)\leq\int_{\jump{u}\cap A}g(x,[\psi_{R_m}^{j(m)}\circ u],\nu_u)\, d\hd.
\end{equation*}
  Since by Proposition \ref{prop:properties of f and g} the function $g$ belongs to $\mathcal{G}$, an application of the Dominated Convergence Theorem  yields \eqref{eq: fine rapp 2}, concluding the proof.
\end{proof}

\section{A smaller collection of integrands}\label{sec:smallerclass}

As the scalar case studied in \cite{DalToa23} and \cite{DalToa23b} suggests, to recover a full integral representation for functionals in $\E_{\rm sc}$, it is convenient to consider a smaller collection of integrands, whose definition is closely related to those studied in \cite{CagnettiGlobal}. In particular, we will show that the Cantor part $E^c$ can be represented as an integral functional whenever $E\in\E_{\rm sc}$ is the $\Gamma$-limit of a sequence of functionals associated to integrands in this in this smaller class and the volume integrand corresponding to $E$ does not depend on $x$ (see Theorem \ref{thm:Cantor}). 

In the rest of the paper we fix two new constants $c_7>0$ and  $\alpha\in (0,1)$. Moreover, we fix a continuous non-decreasing function  $\vartheta\colon [0,+\infty)\to [0,+\infty)$ such that
\begin{equation}\label{eq:requirements on vartheta}
\vartheta(0)=0\,\,\,\textup{ and }\,\,\,\vartheta(t)\geq \frac{c_1}{c_3}t-1\quad\text{for every $t\geq 0$}.
\end{equation}

The smaller collection of volume integrands is introduced in the following definition.

\begin{definition}
   Let $\mathcal{F}^\alpha$ be the collection of functions $f\in\mathcal{F}$ such that 
    \begin{equation}\label{eq:def falfa}
        \Big|\frac{f(x,s\xi)}{s}-\frac{f(x,t\xi)}{t}\Big|\leq \frac{c_7}{s}f(x,s\xi)^{1-\alpha}+\frac{c_7}{s}+\frac{c_7}{t}f(x,t\xi)^{1-\alpha}+\frac{c_7}{t}
    \end{equation}
    for $\Ld$-a.e. $x\in\Rd$ and for every $s,t>0$ and $\xi\in\Rdk$.
\end{definition}
\begin{remark}\label{re:Definition of finfty gzero}
Arguing as in \cite[Remark 4.3]{DalToa23b}, one can show that  $f\in\mathcal{F}^\alpha$ if and only if for $\Ld$-a.e. $x\in\Rd$ and for every $\xi\in\Rkd$ we have
\begin{equation*}
     f^\infty(x,\xi)=\lim_{s\to+\infty}\frac{1}{s}f(x,s\xi),
\end{equation*}
and  
\begin{equation*}
\Big|\frac{1}{s}f(x,s\xi)-f^\infty(x,\xi)\Big|\leq \frac{c_7}{s}+\frac{c_7}{s}f(x,s\xi)^{1-\alpha}
\end{equation*}
for $\Ld$-a.e. $x\in\Rd$, for every $\xi\in\Rkd$ and for every $s>0$. This is closely related to condition (H4) of \cite{bouchitte1998global}.
\end{remark}

We now introduce the smaller collection of surface integrands.
\begin{definition}
     Let $\mathcal{G}^\vartheta$ be the collection of functions $g\in\mathcal{G}$ such that 
    \begin{equation}\nonumber 
        \Big|\frac{g(x,s\zeta,\nu)}{s}-\frac{g(x,t\zeta,\nu)}{t}\Big|\leq \vartheta (s|\zeta|)\frac{g(x,s\zeta,\nu)}{s}+\vartheta(t|\zeta|)\frac{g(x,t,\nu)}{t},
    \end{equation}
    for every $s,t>0$, $x\in\Rd$, $\zeta\in\Rk$, and $\nu\in\mathbb{S}^{d-1}$.
\end{definition}

\begin{remark}
The arguments of \cite[Remark 4.5]{DalToa23b} show that $g\in\mathcal{G}^\vartheta$ if and only if for $x\in\Rd$, $\zeta\in\Rk$ and $\nu\in\mathbb{S}^{d-1}$ we have that the limit
\begin{equation}\label{eq:def g0}
   \displaystyle g^0(x,\zeta,\nu):=\lim_{s\to 0^+}\frac{1}{s}g(x,s\zeta,\nu)
\end{equation}
exists and 
\begin{equation*}
\Big|\frac1s g(x, s\zeta, \nu) - g^0(x, \zeta, \nu)\Big|\leq\vartheta(s|\zeta|)\frac1sg(x, s\zeta,\nu) \leq c_3k\vartheta(s|\zeta|)|\zeta|
\end{equation*}
for every $s>0$, $x\in\Rd$, $\zeta\in\Rk$, and  $\nu\in\mathbb{S}^{d-1}$. By \cite[Remark 3.5]{CagnettiGlobal} this is closely related to condition (g5) of that paper.
\end{remark}

\begin{remark}
    The class $\mathcal{G}^\vartheta$ is non-empty, since the function $(x,\zeta,\nu)\mapsto \sum_{i=1}^k((c_1|\zeta_i|)\land c_3)$ belongs to $\mathcal{G}^\vartheta$. Arguing as in \cite[Remark  4.7]{DalToa23b}, it is possible to show that $\mathcal{G}^\vartheta\neq \emptyset$ if and only if $\vartheta$ satisfies \eqref{eq:requirements on vartheta}
    .
\end{remark}

As in Section \ref{sec:classE}, we  introduce a space of functionals $\E^{\alpha,\vartheta}$ containing $E^{f,g}$ for every $f\in\mathcal{F}^\alpha$ and $g\in\mathcal{G}^\vartheta$.

\begin{definition}
     Let $\E^{\alpha,\vartheta}$ be the space of functionals $E\in \E$  satisfying the condition
    \begin{multline}
        \Big|\frac{E(su,A)}{s}-\frac{E(tu,A)}{t}\Big|\leq \frac{c_7}{s}\Ld(A)^\alpha E(su,A)^{1-\alpha}+\vartheta(su^A)\frac{E(su,A)}{s}+\frac{c_7}{s}\Ld(A)\\
        \label{eq:def ealphatheta}\hspace{2 cm}+\frac{c_7}{t}\Ld(A)^{\alpha}E(tu,A)^{1-\alpha}+\vartheta(tu^A)\frac{E(tu,A)}{t}+\frac{c_7}{t}\Ld(A),
    \end{multline}
    for every $t,s>0$, $A\in\mathcal{A}_c(\Rd)$, and $u\in BV(A;\Rk)\cap L^\infty(A;\Rk)$, where $u^A:=\textup{osc}_Au:=\textup{ess}\sup_{x,y\in A}|u(x)-u(y)|$. We also set $\E^{\alpha,\vartheta}_{\rm sc}:=\E^{\alpha,\vartheta}\cap \E_{\rm sc}.$
\end{definition}

\begin{proposition}
    Let $f\in\mathcal{F}^\alpha$ and $g\in\mathcal{G}^\vartheta$. Then the functional $E^{f,g}$ of  Definition \ref{def:Functionals Efg} belongs to $\E^{\alpha,\vartheta}$.
\end{proposition} 
\begin{proof}
    The result is proved as in \cite[Proposition 4.9]{DalToa23b}, replacing  Remark 2.9 by our Proposition \ref{prop: integral functionals are in E}.
\end{proof}

To study the $\Gamma$-limit of sequences of functionals $(E_n)_n\subset \E^{\alpha,\vartheta}$, given $A\in\mathcal{A}_c(\Rd)$ and a function $u\in BV(A;\Rk)\cap L^\infty(A;\Rk)$, it is important to be able to find approximate recovery sequences for $u$ which are bounded in $L^\infty(A;\Rk)$. This is taken care of in the next lemma.

\begin{lemma}\label{lemma:recovery sequence troncate}
Let $m\in\N$, $A\in\mathcal{A}_c(\Rd)$, $u\in BV(A;\Rk)\cap L^\infty(A;\Rk)$, $E\in\E$, 
and  $(E_n)_n\subset\E$, with $E_n(\cdot,A)$ $\Gamma$-converging to $E$ in the topology of $L^0(\Rd;\Rk)$.
Then there exist sequences $(u_n)_n\subset BV(A;\Rk)$ and $(v_n)_n\subset BV(A;\Rk)$ converging to $u$ in $L^1(A;\Rk)$ such that 
\begin{eqnarray}
    &\displaystyle \label{eqref:stima per oscillazione}u^A_n\leq 4\sigma^{m}u^A \quad  \text{ and }\quad   \|v_n\|_{L^\infty(A;\Rk)}\leq 2\sigma^m\|u\|_{L^\infty(A;\Rk)} \quad \text{for every } n\in\N, \\
\label{eq:recovery troncata}
   &\displaystyle \limsup_{n\to+\infty }E_n(u_n,A)\lor \limsup_{n\to+\infty } E_n(v_n,A)\leq E(u,A)+C\frac{E(u,A)+\Ld(A)}{m},
\end{eqnarray}
 where $C$ is the constant in property (g) of Definition \ref{def:space of functionals E}.
\end{lemma}

\begin{proof}
By property (d) of Definition \ref{def:space of functionals E}, it is not restrictive to assume that $u^A=\|u\|_{L^{\infty}(A;\Rk)}$. By  $\Gamma$-convergence there exists a sequence $(w_n)\subset L^0(\Rd;\Rk)$ converging to $u$ in $L^0(\Rd;\Rk)$ such that 
\begin{equation}
\label{eq:recuperiamml}\lim_{n\to+\infty}E_n(w_n,A)=E(u,A).
\end{equation}

Let us fix $R:=2\|u\|_{L^\infty(A;\Rk)}$. By property (g) of Definition \ref{def:space of functionals E}  for every $n\in\N$ there exists an index $i(n)\in\{1,...,m\}$ such that 
\begin{equation*}
     E_n(\psi_R^{i(n)}\circ w_n,A)\leq E_n(w_n,A)+C\frac{E(u,A)+\Ld(A)}{m}+c_4\Ld(A^R_{v_n,0}).
\end{equation*}
Since $w_n\to u$ in $L^0(\Rd;\Rk)$, by our choice of $R$ we have $\Ld(A_{w_n,0}^{R})\to 0$. Setting $u_n:=\psi_R^{i(n)}\circ w_n$, we deduce  from \eqref{eq:recuperiamml} and the previous inequality that
\begin{equation*}
    \limsup_{n\to+\infty}E_n(u_n,A)\leq E(u ,A)+C\frac{E(u,A)+\Ld(A)}{m}.
\end{equation*}

\noindent 
We conclude noting that by \eqref{eq:properties psiRi} we have $u_n^A\leq 2\|u_n\|_{L^\infty(A; \Rk)}\leq 4\sigma^{m}\|u\|_{L^\infty(A;\Rk)}=4\sigma^mu^A$ and  $u_n\to u$ in $L^1(A;\Rk)$.

The construction of $v_n$ is similar.
\end{proof}

We now want study  $\Gamma$-limits of sequences of functionals in $\E^{\alpha,\vartheta}$. To this aim, it is convenient to introduce a family of subspaces of $\E$, in which \eqref{eq:def ealphatheta} holds in a weaker form. 
Given $m\in\N$, we set
\begin{equation}\label{def:vartheta m}
    \vartheta_m(t):=\vartheta(4\sigma^mt)
\end{equation}
for every $t\geq 0$.
\begin{definition}
    Given $m\in\N$, we denote by $\E^{\alpha,\vartheta_m}_m$ the space of functionals $E\in\E$ satisfying the condition 
       \begin{multline}
       \Big|\frac{E(su,A)}{s}-\frac{E(tu,A)}{t}\Big|\leq \frac{c_7}{s}\Ld(A)^\alpha E(su,A)^{1-\alpha}+\vartheta_m(su^A)\frac{E(su,A)}{s}+\frac{c_7}{s}\Ld(A)\\
     \hspace{-1.5 cm}\label{eq:def ealphatheta m}\hspace{2 cm}+\frac{c_7}{t}\Ld(A)^{\alpha}E(tu,A)^{1-\alpha}+\vartheta_m(tu^A)\frac{E(tu,A)}{t}+\frac{c_7}{t}\Ld(A)+R_{m,s}(u,A)+R_{m,t}(u,A)
    \end{multline}
    for every $t,s>0$, $A\in\mathcal{A}_c(\Rd)$, $u\in BV(A;\Rk)\cap L^\infty(A,\Rk)$, where $\vartheta_m$ is the function defined by \eqref{def:vartheta m} and  for every $t>0$
    \begin{equation}\label{eq:def R}
\hspace{-0.2 cm}R_{m,t}(u,A):=C\frac{1+\vartheta_m(tu^A)}{t}\frac{E(tu,A)+\Ld(A)}{m}
      +C^{\alpha }\frac{c_7}{t}\Big(\frac{E(tu,A)+\Ld(A)}{m}\Big)^{1-\alpha}\Ld(A)^\alpha,
    \end{equation}
    $C$ being the constant of property (g) of Definition \ref{def:space of functionals E}.
    We also set $\ClosureE:=\Big(\bigcap_{m=1}^\infty\E^{\alpha,\vartheta_m}_m\Big)\cap\E_{\rm sc}$.
\end{definition}

We now show that $\Gamma$-limits of sequences in $\E^{\alpha,\vartheta}$ belong to the larger space $\E^{\alpha,\vartheta}_w$.
\begin{proposition}\label{prop:limits are in Eweak}
    Let  $(E_n)_n$ be a sequence of functionals in $\E^{\alpha,\vartheta}$ and let  $E\in\E$. Assume that for every $A\in\mathcal{A}_c(\Rd)$ the sequence $E_n(\cdot,A)$ $\Gamma$-converges to $E(\cdot,A)$ with respect to the topology of $L^0(\Rd;\Rk)$. Then $E\in\E^{\alpha,\vartheta_m}_m\cap \E_{\rm sc}$ for every $m\in\N$. In particular, $E\in\ClosureE$.
\end{proposition}
\begin{proof}
    By the semicontinuity of $\Gamma$-limits we have $E\in\E_{\rm sc}$. Thus, given $m\in\N$, we only need to prove that $E\in\E^{\alpha,\vartheta_m}_m$.
    
    By the continuity of $\vartheta_m$ and exchanging the roles of $s$ and $t$, to conclude it is enough to show that 
    \begin{multline}\label{eq:claim Ealpha then Eaplha m}
    (1-\vartheta_m(su^A))\frac{E(su,A)}{s}-\frac{c_7}{s}\Ld(A)-\frac{c_7}{s}\Ld(A)^\alpha E(su,A)^{1-\alpha}\\
    \leq (1+\vartheta_m(tu^A))\frac{E(tu,A)}{t}
    +\frac{c_7}{t}\Ld(A)+\frac{c_7}{t}\Ld(A)^\alpha E(tu,A)^{1-\alpha} +R_{m,t}(u,A),
    \end{multline}
    assuming that the left-hand side is stricty positive.
    
   Let $A\in\mathcal{A}_c(\Rd)$, $u\in L^\infty(A;\Rk)$, and  $t,s>0$. By Lemma \ref{lemma:recovery sequence troncate}, there exists a sequence $(u_n)_n\subset BV(A;\Rk)\cap L^\infty(A;\Rk)$ such that $u_n\to u$ in $L^1(A;\Rk)$, $u_n^A\leq 4\sigma^mu^A$, and
    \begin{equation}\label{eq:energy bound e then e m}
        \limsup_{n\to+\infty}E_n(tu_n,A)\leq E(tu,A)+C\frac{E(tu,A)+\Ld(A)}{m}.
    \end{equation}
        Since $u_n\in BV(A;\Rk)\cap L^\infty(A;\Rk)$, by  \eqref{eq:def ealphatheta} we get 
    \begin{multline*} 
        (1-\vartheta(su_n^A))\frac{E_n(su_n,A)}{s}-\frac{c_7}{s}\Ld(A)-\frac{c_7}{s}\Ld(A)^\alpha E_n(su_n,A)^{1-\alpha}\\
        \leq (1+\vartheta(tu_n^A))\frac{E_n(tu_n,A)}{t}
    +\frac{c_7}{t}\Ld(A)+\frac{c_7}{t}\Ld(A)^\alpha E_n(tu_n,A)^{1-\alpha}.
    \end{multline*}
    Taking the limsup as $n\to+\infty$ and using the monotonicity of $\vartheta$ and  \eqref{eq:energy bound e then e m}, we deduce  that
    \begin{multline}\label{eq:limsup leq cosa}
      \limsup_{n\to+\infty }\Big((1-\vartheta_m(su^A))\frac{E_n(su_n,A)}{s}-\frac{c_7}{s}\Ld(A)-\frac{c_7}{s}\Ld(A)^\alpha E_n(su_n,A)^{1-\alpha} \Big)\\
      \leq\limsup_{n\to+\infty } \Big((1-\vartheta(su_n^A))\frac{E_n(su_n,A)}{s}-\frac{c_7}{s}\Ld(A)-\frac{c_7}{s}\Ld(A)^\alpha E_n(su_n,A)^{1-\alpha}\Big)\\
     \leq\limsup_{n\to+\infty}\Big(1+\vartheta(tu_n^A))\frac{E_n(tu_n,A)}{t}
    +\frac{c_7}{t}\Ld(A)+\frac{c_7}{t}\Ld(A)^\alpha E_n(tu_n,A)^{1-\alpha}\Big)\\
    \leq (1+\vartheta_m(tu^A))\frac{E(tu,A)}{
    t} +\frac{c_7}{t}\Ld(A)+\frac{c_7}{t}\Ld(A)^\alpha E(tu,A)^{1-\alpha}+R_{m,t}(u,A),
    \end{multline}
where $R_{m,t}(u,A)$ is defined by \eqref{eq:def R}.
To deal with the first term in the previous chain of inequalities we introduce the function $\Phi$ defined for every $z\in[0,+\infty)$ as 
\begin{equation*}
    \Phi(z):=(1-\vartheta_m(su^A))\frac{z}{s}-\frac{c_7}{s}\Ld(A)^{1-\alpha}z^{1-\alpha}-\frac{c_7}{s}\Ld(A).
\end{equation*}
Since the left-hand side of \eqref{eq:claim Ealpha then Eaplha m} is strictly positive, we have that $\Phi(E(su,A))>0$, which implies $(1-\vartheta_m(su^A))>0$. We set  $z_0:=c_7^{1/\alpha}(1-\vartheta_m(su^A))^{-1/\alpha}\Ld(A)$ and observe that $\Phi$ is increasing on $(z_0,+\infty)$ and that if $\Phi(z)>0$ then $z>z_0$; in particular, $E(su,A)>z_0$.
Finally, from the $\Gamma$-convergence of $E_n(\cdot,A)$ to $E(\cdot,A)$ and from the convergence of  $u_n$ to $u$ in $L^0(\Rd;\Rk)$, we deduce that
\[
E(su,A)\leq\liminf_{n\to+\infty}E_n(su_n,A),
\]
which, by the monotonicity of $\Phi$, implies 
\begin{equation*}
\Phi(E(su,A))\leq\liminf_{n\to+\infty}\Phi(E_n(su_n,A)).
\end{equation*}
Recalling the definition of $\Phi$, from this inequality and
\eqref{eq:limsup leq cosa} we obtain 
\eqref{eq:claim Ealpha then Eaplha m}, concluding the proof.
\end{proof}

The next technical results will used in the proof of the representation theorem presented in Section \ref{sec:repr}.  Given $\xi\in\Rkd$, and $m\in\N$, we set 
\begin{equation}\label{eq:def kappamxi}
\kappa_{\xi,m}:=2c_{\xi,m}+2d^{1/2}|\xi|,\\
\end{equation}
where $c_{\xi,m}$ the constant defined by \eqref{eq:def costante cxi}.
\begin{lemma}\label{prop:closure ealpha}
    For every $\xi\in\Rkd$ there exists a constant with the following property: for every $m\in\N$ and $E\in \E^{\alpha,\vartheta_m}\cap \E_{\rm sc}$ there exists a set Let $m\in\N$ and  $E\in\E^{\alpha,\vartheta_m}_m\cap\E_{\rm sc}$.  Then  
 there exists $N_m\in\mathcal{B}(\Rd),$ with $\Ld(N_m)=0$, satisfying the following property:  for every $x\in\Rd\setminus N_m$,  $\lambda\geq 1$, $\nu\in\mathbb{S}^{d-1}$,  $s, t>0$, and $\xi\in\Rdk$ there exist $\rho^{\nu,\lambda}_{m,\xi,t,s}(x)>0$, and $M_\xi$, the latter depending only on $|\xi|$ and on the structurual constants $c_1,...,c_7$,$k$ and $\alpha$, such that for every $\rho\in(0,\rho^{\nu,\lambda}_{m,\xi,t,s}(x))$  we have
    \begin{eqnarray}
    \nonumber 
       && \Big|\frac{m^{E}(s\ell_\xi,Q^\lambda_\nu(x,\rho))}{s}-\frac{m^{E}(t\ell_\xi,Q^\lambda_\nu(x,\rho))}{t}\Big|
        \\&&\nonumber \,\,\leq\frac{c_7}{s}\lambda^{\alpha(d-1)}\rho^{d\alpha}m^E(s\ell_\xi,Q^\lambda_\nu(x,\rho))^{1-\alpha}+\vartheta_{m}\Big(\frac{s\kappa_{\xi,m}\lambda\rho}{(s\land 1)(t\land 1)}\Big)\frac{m^E(s\ell_\xi,Q^\lambda_\nu(x,\rho))}{s}+\frac{c_7}{s}\lambda^{d-1}\rho^d\\
       &&\nonumber \quad +\frac{c_7}{t}\lambda^{\alpha(d-1)}\rho^{d\alpha}m^E(t\ell_\xi,Q^\lambda_\nu(x,\rho))^{1-\alpha}+\vartheta_{m}\Big( \frac{t\kappa_{\xi,m}\lambda\rho}{(s\land 1)(t\land 1)}\Big)\frac{m^E(t\ell_\xi,Q^\lambda_\nu(x,\rho))}{t}+ \frac{c_7}{t}\lambda^{d-1}\rho^d\\&&\nonumber \quad
      +\Big(1+\vartheta_m\Big(\frac{s\kappa_{\xi,m}\lambda\rho}{(s\land 1)(t\land 1)}\Big)\Big)\frac{M_\xi}{(s\land 1)m}\lambda^{d-1}\rho^d
     \nonumber +\Big(\frac{1}{s^\alpha}+\frac{1}{s }\Big)\frac{M_\xi}{m^{1-\alpha}}\lambda^{d-1}\rho^{ d}\\&&\quad \label{eq:stima lunga ac}+\Big(1+\vartheta_m\Big(\frac{t\kappa_{\xi,m}\lambda\rho}{(s\land 1)(t\land 1)}\Big)\Big)\frac{M_\xi}{(t\land 1)m}\lambda^{d-1}\rho^d
      +\Big(\frac{1}{t^\alpha}+\frac{1}{t }\Big)\frac{M_\xi}{m^{1-\alpha}}\lambda^{d-1}\rho^{ d},
     \end{eqnarray}
      where $\kappa_{\xi,m}$ is the constant defined  by \eqref{eq:def kappamxi}. If, in addition the function $f$ defined by \eqref{eq:definition of small f} does not depend on $x$, the set $N_m=\emptyset$ and $\rho^{\nu,\lambda}_{m,\xi,t,s}(x)=+\infty$.
\end{lemma} 
\begin{proof}
     By hypothesis $E\in\E_{\rm sc}$, so that by Theorem \ref{thm:rappresentazione}, the function $f$ defined by \eqref{eq:definition of small f} satisfies \eqref{eq: Ea assoltamente continuo fittizio}  for every $A\in\mathcal{A}_c(\Rd)$ and  $u\in BV(A;\Rk)$. Hence, we may apply Corollary \ref{cor:troncature} to obtain a set $N_m\in\mathcal{B}(\Rd)$, with $\Ld(N_m)=0$,  satisfying the following property: for every $x\in \Rd\setminus N_m$,  $\xi\in\Rkd$, $t>0$, $\nu\in\Sn^{d-1}$, and $\lambda\geq 1$ there exists $\rho^{\nu,\lambda}_{m,t\xi}(x)$ such that for every  $\rho\in(0,\rho^{\nu,\lambda}_{m,t \xi}(x))$, there exists $u\in BV(Q^\lambda_\nu(x,\rho);\Rk)\cap L^\infty(Q^\lambda_\nu(x,\rho);\Rk)$, with  tr$_{Q^\lambda_\nu(x,\rho)}u=$tr$_{ Q^\lambda_\nu(x,\rho)}\ell_{\xi}$ and $\|u-\ell_{\xi}\|_{L^\infty(Q^\lambda_\nu(x,\rho);\Rk)}\leq \frac1tc_{t\xi,m}\lambda\rho\leq \frac{1}{t\land 1}c_{\xi,m}\lambda \rho$, such that
\begin{equation}\label{eq:def tu ealphatheta}
    E(tu,Q^\lambda_\nu(x,\rho))\leq m^E(t\ell_{\xi},Q^\lambda_\nu(x,\rho))+\frac{C_{t\xi}}{m}\lambda^{d-1}\rho^d\leq m^E(t\ell_{\xi},Q^\lambda_\nu(x,\rho))+\frac{(t\lor 1)C_{\xi}}{m}\lambda^{d-1}\rho^d ,
\end{equation}
where $C_\xi$ is  defined by \eqref{def: costante resti}. Note that the oscillation of $u$ satisfies 
\begin{equation}\label{eq:oscillazione properieta 1}
u^{Q^\lambda_\nu(x,\rho)}\leq (\tfrac{2}{t\land 1}c_{\xi,m}+\tfrac{2}{t\land 1}d^{1/2}|\xi|)\lambda\rho=\frac{\kappa_{\xi,m}}{t\land 1}\lambda\rho.
\end{equation}
We can estimate $m^E(t\ell_{\xi},Q^\lambda_\nu(x,\rho))$ by evaluating $E(\cdot,Q^\lambda_\nu(x,\rho))$ at $t\ell_{\xi}$. Recalling \eqref{eq:def tu ealphatheta} and (c2$'$) of Definition \ref{def:space of functionals E}, this leads to
    \begin{equation}\label{eq:stima da quasi minimalita}
        E(tu,Q^\lambda_\nu(x,\rho))\leq \big(c_3k^{1/2}t|\xi|+c_4+(t\lor 1)C_\xi\big)\lambda^{d-1}\rho^d.
    \end{equation}

     Since $E\in\E^{\alpha,\vartheta_m}_m$, by \eqref{eq:def ealphatheta m} for every $s>0$ we have that
    \begin{eqnarray}
        && \nonumber (1-\vartheta_m(su^{Q^\lambda_\nu(x,\rho)}))\frac{E(su,Q^\lambda_\nu(x,\rho))}{s}-\frac{c_7}{s}(\lambda^{d-1}\rho^d)^{\alpha} E(su,Q^\lambda_\nu(x,\rho))^{1-\alpha}\\
        &&\hspace{2 cm}\nonumber \qquad-\frac{c_7}{s}\lambda^{d-1}\rho^d -R_{m,s}(u,Q^\lambda_\nu(x,\rho))\\
   && \nonumber \leq (1+\vartheta_m(tu^{Q^\lambda_\nu(x,\rho)}))\frac{E(tu,Q^\lambda_\nu(x,\rho))}{t} +\frac{c_7}{t}(\lambda^{d-1}\rho^d)^{\alpha} E(tu,Q^\lambda_\nu(x,\rho))^{1-\alpha}\\
   &&\hspace{2 cm} \qquad \label{eq:stima lunga property 1} +\frac{c_7}{t}\lambda^{d-1}\rho^d
  +R_{m,t}(u,Q^\lambda_\nu(x,\rho)),
    \end{eqnarray}
where $R$ is defined by \eqref{eq:def R}. Using the monotonicity of $\vartheta_m$, \eqref{eq:def tu ealphatheta}, \eqref{eq:oscillazione properieta 1},  \eqref{eq:stima da quasi minimalita}, and the subadditivity of the function $z\mapsto z^{1-\alpha}$ on $[0,+\infty)$ we see that there exists a positive constant $M_\xi$, independent of $m$, such that the left-hand side can be bounded from below by
    \begin{eqnarray}
       && \nonumber \Big(1-\vartheta_{m}\Big(\frac{s\kappa_{\xi,m}\lambda\rho}{t\land 1}\Big)\Big)\frac{E(su,Q^\lambda_\nu(x,\rho))}{s}-\frac{c_7}{s}\lambda^{\alpha(d-1)}\rho^{\alpha d} E(su,Q^\lambda_\nu(x,\rho))^{1-\alpha} \\&&\nonumber-\frac{c_7}{s}\lambda^{d-1}\rho^d 
   -C\Big(1+\vartheta_m\Big(\frac{s\kappa_{\xi,m}\lambda\rho}{t\land 1}\Big)\Big)\frac{E(su,Q^\lambda_\nu(x,\rho))+\lambda^{d-1}\rho^d}{sm}
        \\&& \label{eq:esselunga}-C^\alpha\frac{c_7}{s}\frac{E(su,Q^\lambda_\nu(x,\rho))^{1-\alpha}\lambda^{\alpha(d-1)}\rho^{\alpha d}+\lambda^{d-1}\rho^d}{m^{1-\alpha}},
    \end{eqnarray}
while the right-hand side of \eqref{eq:stima lunga property 1} can be bounded from above by
    \begin{eqnarray}
   && \nonumber \Big(1+\vartheta_m\Big(\frac{t\kappa_{\xi,m}\lambda\rho}{t\land 1}\Big)\Big)\frac{m^E(t\ell_{\xi},Q^\lambda_\nu(x,\rho))}{t} +\frac{c_7}{t} m^E(t\ell_{\xi},Q^\lambda_\nu(x,\rho))^{1-\alpha}\lambda^{\alpha(d-1)}\rho^{\alpha d}\\  
   &&\nonumber +\frac{c_7}{t}\lambda^{d-1}\rho^d+\Big(1+\vartheta_m\Big(\frac{t\kappa_{\xi,m}\lambda\rho}{t\land 1}\Big)\Big)\frac{C(c_3k^{1/2}|\xi|+c_4+C_\xi+1)+C_\xi}{(t\land 1)m}\lambda^{d-1}\rho^d
        \\&& \label{eq:upperbound lemma tecnico}+\Big(C^\alpha\frac{c_7}{t^\alpha}+\frac{c_7}{t }\Big)\frac{\big(c_3k^{1/2}|\xi|+c_4+C_\xi\big)^{1-\alpha}+1+C_\xi^{1-\alpha}}{m^{1-\alpha}}\lambda^{d-1}\rho^{ d}.
    \end{eqnarray}
    Therefore, by \eqref{eq:stima lunga property 1}-\eqref{eq:upperbound lemma tecnico} there exists a constant $M_\xi$, independent of $m$, such that
    \begin{eqnarray}
       && \hspace{-1  cm}\nonumber \Big(1-\vartheta_{m}\Big(\frac{s\kappa_{\xi,m}\lambda\rho}{t\land 1}\Big)\Big)\frac{E(su,Q^\lambda_\nu(x,\rho))}{s}-\frac{c_7}{s}\lambda^{\alpha(d-1)}\rho^{\alpha d} E(su,Q^\lambda_\nu(x,\rho))^{1-\alpha} -\frac{c_7}{s}\lambda^{d-1}\rho^d \\\hspace{-1 cm}&&\hspace{-1 cm}\nonumber 
   -C\Big(1{+}\vartheta_m\Big(\frac{s\kappa_{\xi,m}\lambda\rho}{t\land 1}\Big)\Big)\frac{E(su,Q^\lambda_\nu(x,\rho)){+}\lambda^{d-1}\rho^d}{sm}
       -C^\alpha\frac{c_7}{s}\frac{E(su,Q^\lambda_\nu(x,\rho))^{1-\alpha}\lambda^{\alpha(d-1)}\rho^{\alpha d}{+}\lambda^{d-1}\rho^d}{m^{1-\alpha}}\\
        &&\hspace{-1  cm}\leq  \nonumber \Big(1+\vartheta_m\Big(\frac{t\kappa_{\xi,m}\lambda\rho}{t\land 1}\Big)\Big)\frac{m^E(t\ell_{\xi},Q^\lambda_\nu(x,\rho))}{t} +\frac{c_7}{t} m^E(t\ell_{\xi},Q^\lambda_\nu(x,\rho))^{1-\alpha}\lambda^{\alpha(d-1)}\rho^{\alpha d}\\  
   &&\hspace{-1 cm}\nonumber+\frac{c_7}{t}\lambda^{d-1}\rho^d+\Big(1+\vartheta_m\Big(\frac{t\kappa_{\xi,m}\lambda\rho}{t\land 1}\Big)\Big)\frac{M_\xi}{(t\land 1)m}\lambda^{d-1}\rho^d
      +\Big(\frac{1}{t^\alpha}+\frac{1}{t }\Big)\frac{M_\xi}{m^{1-\alpha}}\lambda^{d-1}\rho^{ d}.
        \end{eqnarray}

We claim that the previous inequality still holds replacing $E(su,Q^\lambda_\nu(x,\rho))$ by $m^E(s\ell_\xi,Q^\lambda_\nu(x,\rho))$.
To prove the claim, for given $s,t>0$ we introduce the function $\Phi$ defined for every $z\in[0,+\infty)$ as 
\begin{multline*}
    \Phi(z):=\Big(1-\vartheta_m\Big(\frac{s\kappa_{\xi,m}\lambda\rho}{t\land 1}\Big)\Big)\frac{z}{s}-\frac{c_7}{s}\lambda^{\alpha(d-1)}\rho^{\alpha d}z^{1-\alpha
    }-\frac{c_7}{s}\lambda^{d-1}\rho^d\\
    -C\Big(1+\vartheta_m\Big(\frac{s\kappa_{\xi,m}\lambda}{t\land 1}\rho\Big)\Big)\frac{z+\lambda^{d-1}\rho^d}{sm}
        -C^\alpha\frac{c_7}{s}\frac{z^{1-\alpha}\lambda^{\alpha(d-1)}\rho^{\alpha d}+\lambda^{d-1}\rho^d}{m^{1-\alpha}},
\end{multline*}
so that the left-hand side of \eqref{eq:esselunga} is equal to  $\Phi(E(su,Q^\lambda_\nu(x,\rho))$.
We now show that 
\begin{equation}\label{eq:Phi mleq phi E}
\Phi(m^{E}(s\ell_{\xi},Q^{\lambda}_\nu(x,\rho)))\leq \Phi(E(su,Q^{\lambda}_\nu(x,\rho))).
\end{equation}
Since the righ-hand side is clearly larger than zero, it is enough to prove this inequality when 
$\Phi(m^{E}(s\ell_{\xi},Q^{\lambda}_\nu(x,\rho)))>0$. Note that this positivity condition implies that $1-\vartheta_m(\frac{s\kappa_{\xi,m}\lambda\rho}{t\land 1})>0$. Some straightforward computations, show that if $\Phi(z)>0$ then $z>z_0$, where
\begin{equation*}
    z_0:=c_7^{1/\alpha}\Big((1-\vartheta_m\Big(\frac{s\kappa_{\xi,m}\lambda\rho}{t\land 1}\Big)\Big)-\Big(\Big(1+\vartheta_m(\frac{s\kappa_{\xi,m}\lambda\rho}{t\land 1}\Big)\Big)\frac{C}{m}\Big)^{-1/\alpha}\Big(1+\frac{C^\alpha}{m^{1-\alpha}}\Big)^{1/\alpha}\lambda^{d-1}\rho^d.
\end{equation*}
One can also see that  $\Phi$ is increasing on $(z_0,+\infty)$.
Since $\text{tr}_{Q^{\lambda}_\nu(x,\rho)}su=\text{tr}_{Q^{\lambda}_\nu(x,\rho)}s\ell_\xi$, we obtain that 
\begin{equation}\label{eq:stima finale tecnico 1}
   m^{E}(s\ell_{\xi},Q^{\lambda}_\nu(x,\rho)))\leq E(su,Q^{\lambda}_\nu(x,\rho)).
\end{equation}
Since we assumed that  $\Phi(m^{E}(s\ell_{\xi},Q^{\lambda}_\nu(x,\rho)))>0$, we have that $m^{E}(s\ell_{\xi},Q^{\lambda}_\nu(x,\rho))>z_0$. Hence, recalling that $\Phi$ is increasing on $(z_0,+\infty)$, from \eqref{eq:stima finale tecnico 1} we obtain that  \eqref{eq:Phi mleq phi E}. Thus, in \eqref{eq:esselunga} we can substitute $E(su,Q^\lambda_\nu(x,\rho))$ by $m^E(s\ell_\xi,Q(x,\rho))$. This new inequality, together with
\begin{equation*}
    \frac{m^{E}(s\ell_{\xi},Q^{\lambda}_\nu(x,\rho))}{s}\leq \frac{c_3k^{1/2}|\xi|+c_4}{s\land 1}\lambda^{d-1}\rho^d,
\end{equation*}
implies that
\begin{eqnarray}
       && \nonumber \Big(1-\vartheta_{m}\Big(\frac{s\kappa_{\xi,m}\lambda\rho}{t\land 1}\Big)\Big)\frac{m^E(s\ell_\xi,Q^\lambda_\nu(x,\rho))}{s}-\frac{c_7}{s}\lambda^{\alpha(d-1)}\rho^{\alpha d} m^E(s\ell_\xi,Q^\lambda_\nu(x,\rho))^{1-\alpha} \\&&-\frac{c_7}{s}\lambda^{d-1}\rho^d
       \nonumber 
   -\Big(1+\vartheta_m\Big(\frac{s\kappa_{\xi,m}\lambda\rho}{t\land 1}\Big)\Big)\frac{M_\xi}{(s\land 1)m}\lambda^{d-1}\rho^d
-\frac{1}{s}\frac{M_\xi}{m^{1-\alpha}}\lambda^{d-1}\rho^d\\
        &&\leq  \nonumber \Big(1+\vartheta_m\Big(\frac{t\kappa_{\xi,m}\lambda\rho}{t\land 1}\Big)\Big)\frac{m^E(t\ell_{\xi},Q^\lambda_\nu(x,\rho))}{t} +\frac{c_7}{t} m^E(t\ell_{\xi},Q^\lambda_\nu(x,\rho))^{1-\alpha}\lambda^{\alpha(d-1)}\rho^{\alpha d}\\  
   &&+\frac{c_7}{t}\lambda^{d-1}\rho^d+\Big(1+\vartheta_m\Big(\frac{t\kappa_{\xi,m}\lambda\rho}{t\land 1}\Big)\Big)\frac{M_\xi}{(t\land 1)m}\lambda^{d-1}\rho^d
     \nonumber +\Big(\frac{1}{t^\alpha}+\frac{1}{t }\Big)\frac{M_\xi}{m^{1-\alpha}}\lambda^{d-1}\rho^{ d}.
        \end{eqnarray}

\noindent Exchanging the roles of $s$ and $t$, this gives \eqref{eq:stima lunga ac}.

If, in addition, the function $f$ defined by \eqref{eq:definition of small f} does not depend on $x$, then in Corollary \ref{cor:troncature}  we have  $N_m=\emptyset$ and $\rho^{\nu,\lambda}_{m,t\xi}(x)=\rho^{\nu,\lambda}_{m,s\xi}(x)=+\infty $, concluding the proof in this case.
\end{proof}
The following lemma deals with the case with boundary conditions related to the the jump functions $u_{x,\zeta,\nu}$.
\begin{lemma}\label{prop:closure ealpha jump} There exists a positive constant $M>0$, depending only on the structural constants $c_1,...,c_7$, $k$, and on $\alpha$, such that for every $m\in\N$, $E\in\E^{\alpha,\vartheta_m}_m\cap\E_{\rm sc}$, $x\in\Rd$, $\zeta\in\Rk$, $\nu\in\mathbb{S}^{d-1}$, $s,t>0$, and $\rho>0$ we have 
    \begin{eqnarray}
       && \nonumber \hspace{-0.5 cm}\Big|\frac{m^{E}(su_{x,\zeta,\nu},Q_\nu(x,\rho))}{s}-\frac{m^{E}(tu_{x,\zeta,\nu},Q_\nu(x,\rho))}{t}\Big|\\
        &&\nonumber \leq\frac{c_7}{s}\rho^{d\alpha}m^E(su_{x,\zeta,\nu},Q_\nu(x,\rho))^{1-\alpha}
        +\vartheta_{2m}(s|\zeta|)\frac{m^E(su_{x,\zeta,\nu},Q_\nu(x,\rho))}{s}+\frac{c_7}{s}\rho^d\\
        &&\nonumber \quad\frac{c_7}{t}\rho^{d\alpha}m^E(tu_{x,\zeta,\nu},Q_\nu(x,\rho))^{1-\alpha}
        +\vartheta_{2m}(t|\zeta|)\frac{m^E(tu_{x,\zeta,\nu},Q_\nu(x,\rho))}{t}+\frac{c_7}{t}\rho^d\\
        &&\nonumber \quad +
           (1+\vartheta_{2m}(s|\zeta|))\Big(\frac{M|\zeta|}{m}\rho^{d-1}+\frac{M}{sm}\rho^d\Big)+\frac{1}{s^{\alpha}}\frac{M|\zeta|^{1-\alpha}}{m^{1-\alpha}}\rho^{d-1+\alpha}+\frac{1}{s}\frac{M}{m^{1-\alpha}}\rho^d\\
          && \label{eq:disuguaglianza lunga in closure}\quad 
          +(1+\vartheta_{2m}(t|\zeta|))\Big(\frac{M|\zeta|}{m}\rho^{d-1}+\frac{M}{tm}\rho^d\Big)+\frac{1}{t^{\alpha}}\frac{M|\zeta|^{1-\alpha}}{m^{1-\alpha}}\rho^{d-1+\alpha}+\frac{1}{t}\frac{M}{m^{1-\alpha}}\rho^d.
     \end{eqnarray}
\end{lemma}
\begin{proof}
  By Lemma \ref{cor:Corollario Troncature Hd} for every $x\in\Rd$, $\zeta\in\Rk$, $t>0$, $\nu\in\Sn^{d-1}$, and $\rho>0$   there exists a function $u\in BV(Q_\nu(x,\rho);\Rk)\cap L^\infty(Q_\nu(x,\rho);\Rk)$ such that $\text{tr}_{Q_\nu(x,\rho)}u=\text{tr}_{Q_\nu(x,\rho)}u_{x,\zeta,\nu}$, $\|u\|_{L^\infty(Q_\nu(x,\rho);\Rk)}\leq\sigma^m|\zeta|$, and
    \begin{equation}\label{eq:recovery for jump closure}
        E(tu, Q_\nu(x,\rho))\leq m^{E}(tu_{x,\zeta,\nu},Q_\nu(x,\rho))+\frac{ t K|\zeta|}{m}\rho^{d-1}+K\rho^d,
    \end{equation}
    where $K>0$ is the constant in Lemma \ref{cor:Corollario Troncature Hd}.
    Note that the oscillation of $u$ satisfies $u^{Q_\nu(x,\rho)}\leq 2\sigma^m|\zeta|$. We now estimate $m^{E}(tu_{x,\zeta,\nu},Q_\nu(x,\rho))$ by  evaluating  $E(\cdot,Q_\nu(x,\rho))$ at $tu_{x,\zeta,\nu}$ and by  (c2) we get
    \begin{equation*}
        m^{E}(tu_{x,\zeta,\nu},Q_\nu(x,\rho))\leq c_3kt|\zeta|\rho^{d-1}+c_4\rho^d,
    \end{equation*}
    so that from \eqref{eq:recovery for jump closure} we deduce 
    \begin{equation}\label{eq:bound on etu closure}
           E(tu, Q_\nu(x,\rho))\leq t|\zeta|\Big(c_3k+\frac{K}{m}\Big)\rho^{d-1}+(c_4+K)\rho^d.
    \end{equation}

 Since  $E\in\E^{\alpha,\vartheta_m}_m$, by  \eqref{eq:def ealphatheta m}  for every $s>0$ we have that 
   \begin{multline*}
   (1-\vartheta_m(su^{Q_\nu(x,\rho)}))\frac{E(su,Q_\nu(x,\rho))}{s}-\frac{c_7}{s}\rho^{\alpha d} E(su,Q_\nu(x,\rho))^{1-\alpha}-\frac{c_7}{s}\rho^d
    -R_s(u,Q_\nu(x,\rho))\\\hspace{-0.25cm}
    \leq (1+\vartheta_m(tu^{Q_\nu(x,\rho)}))\frac{E(tu,Q_\nu(x,\rho))}{t}
   +\frac{c_7}{t}\rho^{\alpha d}E(tu,Q_\nu(x,\rho))^{1-\alpha}+\frac{c_7}{t}\rho^d+R_t(u,Q_\nu(x,\rho)).
    \end{multline*}
 Using the inequality $u^{Q_\nu(x,\rho)}\leq 2\sigma^m|\zeta|$,  the monotonicity of $\vartheta_m$, and the subadditivity of $z\mapsto z^{1-\alpha}$, we deduce that the left-hand side of the last inequality can be estimated from below by
\begin{eqnarray}
       && \nonumber (1-\vartheta_{2m}(s|\zeta|))\frac{E(su,Q^\lambda_\nu(x,\rho))}{s}-\frac{c_7}{s}\rho^{\alpha d} E(su,Q_\nu(x,\rho))^{1-\alpha} -\frac{c_7}{s}\rho^d  \\&& \nonumber 
   -C(1+\vartheta_{2m}(s|\zeta|))\frac{E(su,Q_\nu(x,\rho))+\rho^d}{sm}
  -C^\alpha\frac{c_7}{s}\frac{E(su,Q_\nu(x,\rho))^{1-\alpha}\rho^{\alpha d}+\rho^d}{m^{1-\alpha}}.
    \end{eqnarray}
while the right-hand side can be bounded from above by
\begin{multline*}
   (1+\vartheta_{2m} (t|\zeta|))\frac{m^E(tu_{x,\zeta,\nu},Q_\nu(x,\rho))}{t}
    +\frac{c_7}{t}m^E(tu_{x,\zeta,\nu},Q_\nu(x,\rho))^{1-\alpha}\rho^{\alpha d}+\frac{c_7}{t}\rho^{ d}\\ 
     +(1+\vartheta_{2m}(t|\zeta|))\frac{(C+1)|\zeta|(kc_3+2K)}{m}\rho^{d-1}+\frac{(1+\vartheta_{2m}(t|\zeta|))}{t}\frac{(C+1)(c_4+2K+1)}{m}\rho^d\\+
  \frac{c_7}{t^\alpha}\frac{{ |\zeta|^{1-\alpha}(C^{\alpha }+1)\Big(k^{1-\alpha}c_3^{1-\alpha}+2K^{1-\alpha}\Big)}}{m^{1-\alpha}}\rho^{d-1+\alpha}+\frac{c_7}{t}\frac{(C^\alpha+1)(c_4^{1-\alpha}+2K^{1-\alpha}+1)}{m^{1-\alpha}}\rho^d,
 \end{multline*}
 Therefore, there exists a constant $M>0$ such that 
 \begin{eqnarray*}
     && \nonumber (1-\vartheta_{2m}(s|\zeta|))\frac{E(su,Q_\nu(x,\rho))}{s}-\frac{c_7}{s}\rho^{\alpha d} E(su,Q_\nu(x,\rho))^{1-\alpha} -\frac{c_7}{s}\rho^d  \\&& \nonumber 
   -C(1+\vartheta_{2m}(s|\zeta|))\frac{E(su,Q_\nu(x,\rho))+\rho^d}{sm}
  -C^\alpha\frac{c_7}{s}\frac{E(su,Q_\nu(x,\rho))^{1-\alpha}\rho^{\alpha d}+\rho^d}{m^{1-\alpha}}\\
  &&\leq  (1+\vartheta_{2m} (t|\zeta|))\frac{m^E(tu_{x,\zeta,\nu},Q_\nu(x,\rho))}{t}
    +\frac{c_7}{t}m^E(tu_{x,\zeta,\nu},Q_\nu(x,\rho))^{1-\alpha}\rho^{\alpha d}+\frac{c_7}{t}\rho^{ d}\\ &&
     +(1+\vartheta_{2m}(t|\zeta|))\frac{M|\zeta|}{m}\rho^{d-1}+\frac{(1+\vartheta_{2m}(t|\zeta|))}{t}\frac{M}{m}\rho^d+
  \frac{1}{t^\alpha}\frac{{ M|\zeta|^{1-\alpha}}}{m^{1-\alpha}}\rho^{d-1+\alpha}+\frac{1}{t}\frac{M}{m^{1-\alpha}}\rho^d,
 \end{eqnarray*}

\noindent Arguing as in the proof of Lemma \ref{prop:closure ealpha} we may substitute $E(su,Q_\nu(x,\rho))$ by $m^{E}(su_{x,\zeta,\nu},Q_\nu(x,\rho))$, so that, taking into account the estimate
\begin{equation*}
     m^{E}(su_{x,\zeta,\nu},Q_\nu(x,\rho))\leq c_3s|\zeta|\rho^{d-1}+c_4\rho^d,
\end{equation*}
we obtain 
\begin{eqnarray*}
    && (1+\vartheta_{2m} (s|\zeta|))\frac{m^E(su_{x,\zeta,\nu},Q_\nu(x,\rho))}{s}
    +\frac{c_7}{s}m^E(su_{x,\zeta,\nu},Q_\nu(x,\rho))^{1-\alpha}\rho^{\alpha d}+\frac{c_7}{s}\rho^{ d}\\ &&
     +(1+\vartheta_{2m}(s|\zeta|))\frac{M|\zeta|}{m}\rho^{d-1}+\frac{(1+\vartheta_{2m}(s|\zeta|))}{s}\frac{M}{m}\rho^d+
  \frac{1}{s^\alpha}\frac{{ M|\zeta|^{1-\alpha}}}{m^{1-\alpha}}\rho^{d-1+\alpha}+\frac{1}{s}\frac{M}{m^{1-\alpha}}\rho^d\\
  &&\leq  (1+\vartheta_{2m} (t|\zeta|))\frac{m^E(tu_{x,\zeta,\nu},Q_\nu(x,\rho))}{t}
    +\frac{c_7}{t}m^E(tu_{x,\zeta,\nu},Q_\nu(x,\rho))^{1-\alpha}\rho^{\alpha d}+\frac{c_7}{t}\rho^{ d}\\ &&
     +(1+\vartheta_{2m}(t|\zeta|))\frac{M|\zeta|}{m}\rho^{d-1}+\frac{(1+\vartheta_{2m}(t|\zeta|))}{t}\frac{M}{m}\rho^d+
  \frac{1}{t^\alpha}\frac{{ M|\zeta|^{1-\alpha}}}{m^{1-\alpha}}\rho^{d-1+\alpha}+\frac{1}{t}\frac{M}{m^{1-\alpha}}\rho^d.
 \end{eqnarray*}
Exchanging the roles of $s,t$, we obtain \eqref{eq:disuguaglianza lunga in closure}.
\end{proof}
We now investigate the properties of the integrands $f$ and $g$ associated with functionals in $\ClosureE$.

\begin{proposition}
\label{prop:f is in Falpha}    Let $E\in\ClosureE$. Then the function $f$ defined by \eqref{eq:definition of small f}  belongs to $\mathcal{F}^\alpha$.
\end{proposition}
\begin{proof}
  For every $m\in\N$  let $N_m\in\mathcal{B}(\Rd)$ be the $\Ld$-negligible set of Lemma \ref{prop:closure ealpha} and  let  $N:=\bigcup_{m\in \N}N_m$. By the same lemma for every $x\in\Rd\setminus N$,  $\xi\in\Rdk$, $m\in\N$, $s,t>0$ and every $\rho$ small enough  we have
        \begin{eqnarray}
    \nonumber 
       && \Big|\frac{m^{E}(s\ell_\xi,Q(x,\rho))}{s}-\frac{m^{E}(t\ell_\xi,Q(x,\rho))}{t}\Big|
        \\&&\nonumber \,\,\leq\frac{c_7}{s}\rho^{d\alpha}m^E(s\ell_\xi,Q(x,\rho))^{1-\alpha}+\vartheta_{m}\Big(\frac{s\kappa_{\xi,m}\rho}{(s\land 1)(t\land 1)}\Big)\frac{m^E(s\ell_\xi,Q(x,\rho))}{s}+\frac{c_7}{s}\rho^d\\
       &&\nonumber \quad +\frac{c_7}{t}\rho^{d\alpha}m^E(t\ell_\xi,Q(x,\rho))^{1-\alpha}+\vartheta_{m}\Big( \frac{t\kappa_{\xi,m}\rho}{(s\land 1)(t\land 1)}\Big)\frac{m^E(t\ell_\xi,Q(x,\rho))}{t}+ \frac{c_7}{t}\rho^d\\&&\nonumber \quad
      +\Big(1+\vartheta_m\Big(\frac{s\kappa_{\xi,m}\rho}{(s\land 1)(t\land 1)}\Big)\Big)\frac{M_\xi}{(s\land 1)m}\rho^d
     \nonumber +\Big(\frac{1}{s^\alpha}+\frac{1}{s }\Big)\frac{M_\xi}{m^{1-\alpha}}\rho^{ d}\\&&\quad \label{eq:stima in f sta in F} +\Big(1+\vartheta_m\Big(\frac{t\kappa_{\xi,m}\rho}{(s\land 1)(t\land 1)}\Big)\Big)\frac{M_\xi}{(t\land 1)m}\rho^d
      +\Big(\frac{1}{t^\alpha}+\frac{1}{t }\Big)\frac{M_\xi}{m^{1-\alpha}}\rho^{ d},
     \end{eqnarray}
    where $\kappa_{\xi,m}$ is the constant defined by \eqref{eq:def kappamxi}.
  Since $\vartheta_m$ is continuous and $\vartheta_m(0)=0$, we have that
     \begin{eqnarray*}
         &\displaystyle \limsup_{\rho\to 0^+} \Big(\vartheta_m\Big(\frac{s\kappa\rho}{(s\land 1)(t\land 1)}\Big)\frac{m^E(s\ell_\xi,Q(x,\rho))}{s\rho^d}\Big)=0, \\
         &\displaystyle
       \limsup_{\rho\to 0^+} \Big(\vartheta_m\Big(\frac{t\kappa\rho}{(s\land 1)(t\land 1)}\Big)\frac{m^E(t\ell_\xi,Q(x,\rho))}{t\rho^d}\Big)=0.
     \end{eqnarray*}
    Therefore, dividing \eqref{eq:stima in f sta in F} by $\rho^d$ and taking the limsup as $\rho\to 0^+$, we obtain
     \begin{eqnarray}
         \nonumber 
         \frac{f(x,s\xi)}{s}&\leq& \frac{f(x,t\xi)}{t} +\frac{c_7}{s}f(x,s\xi)^{1-\alpha}+ \frac{c_7}{s}+f(x,t\xi)^{1-\alpha}+\frac{c_7}{t}\\&& +
    \frac{M_\xi}{(s\land 1)m}
     \nonumber +\Big(\frac{1}{s^\alpha}+\frac{1}{s }\Big)\frac{M_\xi}{m^{1-\alpha}}+
\frac{M_\xi}{(t\land 1)m}
      +\Big(\frac{1}{t^\alpha}+\frac{1}{t }\Big)\frac{M_\xi}{m^{1-\alpha}}.
     \end{eqnarray}
   Letting $m\to+\infty$ and exchanging the roles of $s$ and $t$, we recover \eqref{eq:def falfa}, which concludes the proof.
\end{proof}

\begin{proposition}\label{prop:g is not in gtheta}
    Let $E\in\ClosureE$, $m\in\N$, $x\in\Rd$,  $\zeta\in\Rk$, $\nu\in\mathbb{S}^{d-1}$, and $s,t>0$. Then the function $g$ defined by \eqref{eq:definition of small g} satisfies 
\begin{eqnarray*}\label{eq:claim prop g}    \Big|\frac{g(x,s\zeta,\nu)}{s}-\frac{g(x,t\zeta,\nu)}{t}\Big|&\leq& \vartheta_{2m}(s|\zeta|)\frac{g(x,s\zeta,\nu)}{s}+\vartheta_{2m} (t|\zeta|)\frac{g(x,t\zeta,\nu)}{t}\\&&+(2+\vartheta_{2m}(s|\zeta|)+\vartheta_{2m}(t|\zeta|))\frac{M|\zeta|}{m},
    \end{eqnarray*}
    where $M>0$ is the constant of Lemma \ref{prop:closure ealpha jump}.
\end{proposition}
\begin{proof} Thanks to Lemma  \ref{prop:closure ealpha jump}, inequality \eqref{eq:disuguaglianza lunga in closure} holds. Dividing this inequality by $\rho^{d-1}$ and taking the limsup as $\rho\to 0^+$, we immediately get
\begin{align*}
    \nonumber  \frac{g(x,s\zeta,\nu)}{s}
         &\leq\frac{g(x,t\zeta,\nu)}{t}
        +\vartheta_{2m}(s|\zeta|)\frac{g(x,\zeta,\nu)}{s}+\vartheta_{2m}(t|\zeta|)\frac{g(x,\zeta,\nu)}{t}\\
        &\nonumber \quad +
           (1+\vartheta_{2m}(s|\zeta|))\frac{M|\zeta|}{m}
          +(1+\vartheta_{2m}(t|\zeta|))\frac{M|\zeta|}{m}.
          \end{align*}
Exchanging the roles of $s$ and $t$ we conclude the proof.
\end{proof}

\begin{remark}
    In the scalar case, due to the different nature of the required vertical truncations, it is possible to prove a stronger version of Propositions \ref{prop:limits are in Eweak} and  \ref{prop:g is not in gtheta}. First, the space $\E^{\alpha,\vartheta}$ turns out to be closed under $\Gamma$-convergence with respect to the topology of $L^0(\Rd)$. Then, exploiting this stronger version of Proposition \ref{prop:limits are in Eweak}, it is shown that the function $g$ associated to $E$ by \eqref{eq:definition of small g} belongs to $\mathcal{G}^{\vartheta}$. 

    However, due to the presence of error terms in the truncation procedure, which depend on $\vartheta_{m}$, in the case $k>1$ one cannot recover the closedness of $\E^{\alpha,\vartheta}$ and subsequently conclude that $g\in\mathcal{G}^\vartheta$.
\end{remark}

\section{Full Integral representation}\label{sec:repr}
 We are finally ready to state and prove the full integral representation result for functionals in the class $\ClosureE$. More precisely, we will show that if $E\in\ClosureE$ and if the bulk integrand $f$ defined by \eqref{eq:definition of small f}  does not depend on $x$, then the Cantor part $E^c$ of the functional $E$ can be represented by means of $f^\infty$. Since this result will be employed to  obtain an integral representation for functionals arising from homogenisation of functionals in $\E^{\alpha,\vartheta}$, the  hypothesis that  $f$ does not depend on $x$ is not restrictive for our purposes.

 We now state the main result of this section.
\begin{theorem}\label{thm:Cantor}
Let $E\in\ClosureE$ and $f$ and $g$ be defined by \eqref{eq:definition of small f} and by \eqref{eq:definition of small g}, respectively. Assume that there exists a function $\hat{f}\colon\Rdk\to [0,+\infty)$ such that
\begin{equation}\label{eq:dependence only on one variable}
f(x,\xi)=\hat{f}(\xi) \quad\text{for every }x\in\Rd \text{ and for every }\xi\in\Rdk.
\end{equation}
Then $E=E^{f,g}$, where $E^{f,g}$ is the functional introduced in Definition \ref{def:Functionals Efg}.
\end{theorem}

To prove this result, we try to characterise the Radon-Nikod\'ym derivative of the measures $E^c(u,\cdot)$ with respect to $|D^cu|$ for any $u\in GBV_\star(A;\Rk)$ with $A\in\mathcal{A}_c(\Rd)$. This is done via a careful truncation procedure and taking advantage of the results of  \cite{bouchitte1998global}. We recall that given $A\in\mathcal{A}_c(\Rd)$ and $u\in BV(A;\Rk)$, Alberti's Rank-One Theorem (see \cite[Corollary 4.6]{alberti_1993} and \cite{RankOne4thapproach,DePhiRind,MassVitto}) ensures that for $|D^cu|$-a.e. $x\in A$ the matrix $dD^cu/d|D^cu|(x)$ has rank one, i.e., there exists two Borel functions $a_u\colon A\to\Sn^{k-1}$ and $\nu_u\colon A\to\Sn^{d-1}$  such that for $|D^cu|$-a.e. $x\in A$ we have
\begin{equation}\label{eq:Rank one Cantor part}
    \frac{dD^cu}{d|D^cu|}(x)=a_u(x)\otimes\nu_u(x).
\end{equation}
\begin{lemma}\label{lemma:Radon-Nykodim Cantor}
    Let $E\in\E_{\rm sc}$, $A\in\mathcal{A}_c(\Rd)$, and $u\in BV(A;\Rk)$. Assume that there exists a function $\hat{f}\colon\Rdk\to[0,+\infty)$ satisfying \eqref{eq:dependence only on one variable}. For $|D^cu|$-a.e. $x\in\Rd$, for every $\lambda\geq 1$, and  $\rho>0$ we set
    \begin{equation}\label{eq:def quantities in Cantor lemma}
     s_\rho^\lambda(x):=\frac{|D^cu|(Q^\lambda_{\nu_u(x)}(x,\rho))}{\lambda^{d-1}\rho^d}\quad\text{ and } \quad \xi_\rho^\lambda(x):=s^\lambda_\rho(x)(a_u(x)\otimes\nu_u(x)),
    \end{equation}
    where $(a_u(x),\nu_u(x))\in \Sn^{k-1}\times \mathbb{S}^{d-1}$ is given by \eqref{eq:Rank one Cantor part}.
    Then 
    \begin{eqnarray}\label{eq: behaviour of s at infinity cantor lemma}
       &\displaystyle \lim_{\rho\to +\infty}s_\rho^\lambda(x)=+\infty\quad\text{ and } \quad\lim_{\rho\to+\infty} \rho s_\rho^\lambda(x)=0\,\,\, \ \text{for every $\lambda\geq 1$},\\
    \label{eq:Radon-Nykodim di Dcu}
       & \displaystyle\frac{dE^c(u,\cdot)}{d|D^cu|}(x)=\lim_{\lambda\to+\infty}\limsup_ {\rho\to 0^+}\frac{m^{E}( \ell_{\xi^\lambda_\rho(x)},Q^\lambda_{\nu_u(x)}(x,\rho))}{\lambda^{d-1}\rho^ds_{\rho}^\lambda(x)},
    \end{eqnarray}
     for $|D^cu|$-a.e. $x\in A$.

\end{lemma}

\begin{proof} 
   Let us fix $\delta>0$ and consider the function $E_\delta$ defined by \eqref{eq;def Eepsilon}. It is shown in \cite[Lemma 3.7]{bouchitte1998global}   that \eqref{eq: behaviour of s at infinity cantor lemma} holds true for $|D^cu|$-a.e. $x\in A$ and that for $\delta>0$ one has
    \begin{equation*}
        \frac{dE^c(u,\cdot)}{d|D^cu|}(x)+\delta=\frac{dE^c_\delta(u,\cdot)}{d|D^cu|}(x)=\lim_{\lambda\to+\infty}\limsup_{\rho\to 0^+}\frac{m^{E_\delta}(\ell_{\xi^\lambda_\rho(x)},Q^\lambda_{\nu_u(x)}(x,\rho))}{\lambda^{d-1}\rho^ds_\rho^\lambda(x)}.
    \end{equation*}
    
    \noindent Since $E\leq E_\delta$, we deduce that
    \begin{equation*}
        \frac{dE^c(u,\cdot)}{d|D^cu|}(x)+\delta\geq \limsup_{\lambda\to+\infty}\limsup_{\rho\to 0^+}\frac{m^{E}(\ell_{\xi^\lambda_\rho(x)},Q^\lambda_{\nu_u(x)}(x,\rho))}{\lambda^{d-1}\rho^ds_\rho^\lambda(x)}.
    \end{equation*}  
    
   Thus, it is enough to show that there exists a constant $K>0$, depending only on $c_1,...,c_4$ and on $k$, such that for every $\lambda \geq 1$, for every $\delta>0$, and for $|D^cu|$-a.e. $x\in A$ we have
    \begin{equation}\label{eq:claim lemma cantor}
        \limsup_{\rho\to 0^+}\frac{m^{E_\delta}(\ell_{\xi^\lambda_\rho(x)},Q^\lambda_{\nu_u(x)}(x,\rho))}{\lambda^{d-1}\rho^ds_\rho^\lambda(x)}\leq \Big(1+K\delta\Big)\limsup_{\rho\to 0^+}\frac{m^E(\ell_{\xi^\lambda_\rho(x)},Q^\lambda_{\nu_u(x)}(x,\rho))}{\lambda^{d-1}\rho^ds_\rho^\lambda(x)}.
    \end{equation}

    \noindent We fix $m\in\N$,  $\lambda\geq 1$, and consider a point $x\in A$  such that \eqref{eq:Rank one Cantor part} and  \eqref{eq: behaviour of s at infinity cantor lemma} hold. By Theorem \ref{thm:rappresentazione} hypothesis \eqref{eq: Ea assoltamente continuo fittizio} of Corollary \ref{cor:troncature} is satisfied. Thus, there exists a  function $u_\rho^{x}\in BV(Q^\lambda_{\nu_u(x)}(x,\rho);\Rk)$, with $\text{tr}_{Q_{\nu_u(x)}^\lambda(x,\rho)}u^x_\rho=\text{tr}_{Q_{\nu_u(x)}^\lambda(x,\rho)}\ell_{\xi^\lambda_\rho(x)}$ and $\|u^x_\rho-\ell_{\xi^\lambda_\rho(x)}\|_{L^\infty(Q^\lambda_{\nu_u(x)}(x,\rho);\Rk)}\leq c_{\xi^\lambda_\rho(x),m}\lambda\rho$, such that 
    \begin{equation}\label{eq: u is minimizing lemma representation}
    E(u^{x}_\rho,Q^\lambda_{\nu_{u}(x)}(x,\rho))\leq m^{E}(\ell_{\xi^\lambda_\rho(x)},Q^\lambda_{\nu_u(x)}(x,\rho))+\frac{C_{\xi^\lambda_\rho(x)}}{m}\lambda^{d-1}\rho^d,
   \end{equation}
   where  $c_{\xi^\lambda_\rho(x),m}$ and $C_{\xi^\lambda_\rho(x)}$ are the positive constants defined by \eqref{eq:def costante cxi} and \eqref{def: costante resti}, respectively. 
   Recalling that $\|u^x_\rho-\ell_{\xi^\lambda_\rho(x)}\|\leq c_{\xi^\lambda_\rho(x),m}\lambda\rho$ and observing that $|[u^x_\rho]|=|[u^x_\rho-\ell_{\xi^\lambda_\rho(x)}]|$, we deduce that $|[u^x_\rho]|\leq 2c_{\xi^\lambda_\rho(x),m}\lambda\rho$. Therefore, we may estimate

\begin{eqnarray*}
    \int_{\jump{u_{\rho}^x}}|[u^x_\rho]|\,d\hd&\leq& \int_{(\jump{u^x_{\rho}}\setminus\jump{u_{\rho}^x}^1)}|[u^x_\rho]|\,d\hd
    +2c_{\xi^\lambda_\rho(x),m}\lambda\rho\hd(\jump{u_\rho^x}^1)\\
    &\leq &(1+2c_{\xi^\lambda_\rho(x),m}\lambda\rho) \int_{\jump{u_{\rho}^x}}|[u^x_\rho]|\land 1\,d\hd.
\end{eqnarray*}

    \noindent From  this inequality, by (c1) and (c2) of Definition \ref{def:space of functionals E} we see that
    \begin{eqnarray*}
      &\displaystyle \hspace{-3 cm}E_\delta(u^x_\rho,Q_{\nu_u(x)}^\lambda(x,\rho))=E(u^x_\rho,Q_{\nu_u(x)}^\lambda(x,\rho))+\delta |D^cu^x_\rho|(Q^\lambda_{\nu_u(x)}(x,\rho))\\
      &\displaystyle \hspace{-0.2 cm}\leq E(u^x_\rho,Q_{\nu_u(x)}^\lambda(x,\rho))+\delta \int_{Q^\lambda_{\nu_u(x)}(x,\rho)}|\nabla u^x_\rho |\, dx+ \delta |D^cu|(Q^\lambda_{\nu_u(x)}(x,\rho))
      \\& \hspace{-2.9 cm}\displaystyle+\delta(1+2c_{\xi^\lambda_\rho(x),m}\lambda\rho)\int_{\jump{u_{\rho}^x}}|[u^x_\rho]|\land 1\,d\hd
        \\
       & \displaystyle \hspace{2 cm}\leq\Big(1+\delta\frac{1+2c_{\xi^\lambda_\rho(x),m}\lambda\rho}{c_1}\Big) E(u^x_\rho,Q^\lambda_{\nu_u(x)}(x,\rho))
        +\delta\frac{c_2}{c_1}(1+2c_{\xi^\lambda_\rho(x),m}\lambda\rho)\lambda^{d-1}\rho^d.
    \end{eqnarray*}
    Since $\text{tr}_{Q_{\nu_u(x)}^\lambda(x,\rho)}u^x_\rho=\text{tr}_{Q_{\nu_u(x)}^\lambda(x,\rho)}\ell_{\xi^\lambda_\rho(x)}$, combining the previous inequality with \eqref{eq: u is minimizing lemma representation}, 
    we get 
    \begin{eqnarray}
       && \hspace{-0.5 cm}\displaystyle \nonumber\frac{m^{E_\delta}(\ell_{\xi^\lambda_\rho(x)},Q^\lambda_{\nu_u(x)}(x,\rho))}{\lambda^{d-1}\rho^ds_\rho^\lambda(x)}
         \leq  \Big(1+\delta\frac{1+2c_{\xi^\lambda_\rho(x),m}\lambda\rho}{c_1}\Big) \Big(\frac{m^E(\ell_{\xi^\lambda_\rho(x)},Q^\lambda_{\nu_u(x)}(x,\rho))}{\lambda^{d-1}\rho^ds^\lambda_\rho(x)}+
           \frac{C_{\xi^\lambda_\rho(x)}}{s^\lambda_\rho(x)m}\Big)\\
          &&\hphantom{m^{E_\delta}(\ell_{\xi^\lambda_\rho(x)},Q^\lambda_{\nu_u(x)}(x,\rho))}\displaystyle\label{eq:lemma 6.1 0}+\delta\frac{c_2}{c_1}(1+2c_{\xi^\lambda_\rho(x),m}\lambda\rho)\frac{1}{s^\lambda_\rho(x)}.
 \end{eqnarray}
    
It is immediate to check that when $s^{\lambda}_\rho(x)\geq 1$ we have
    \begin{equation}
\nonumber\frac{m^E(\ell_{\xi^\lambda_\rho(x)},Q^\lambda_{\nu_u(x)}(x,\rho))}{\lambda^{d-1}\rho^ds^\lambda_\rho(x)}\leq c_3k^{1/2}+c_4.
    \end{equation}
   Hence,  from this inequality,  \eqref{eq:def costante cxi}, \eqref{def: costante resti}, and \eqref{eq: behaviour of s at infinity cantor lemma},  we see that there exist two cosntants $C_1$ and $C_2$ depending only on the structural constants $c_1,...,c_4$, $k$, and $d$ such that 
    \begin{eqnarray}
        &&\displaystyle \nonumber\limsup_{\rho\to0^+}\frac{\delta (1 +2 c_{\xi^\lambda_\rho(x),m}\lambda\rho)}
        {c_1}\Big(\frac{m^E(\ell_{\xi^\lambda_\rho(x)},Q^\lambda_{\nu_u(x)}(x,\rho))}{\lambda^{d-1}\rho^ds^\lambda_\rho(x)}+\frac{C_{\xi^\lambda_\rho(x)}}{s^\lambda_\rho(x)m}\Big)\\
        &&\displaystyle\nonumber \leq \limsup_{\rho\to 0^+}\frac{\delta}{c_1}\Big(1+ 2(\sigma^m+1)\lambda\rho\big(s^\lambda_\rho(x)d^{1/2}+C_1\big)\Big)\Big((c_3k^{1/2}+c_4)+\frac{C_{\xi^{\lambda_\rho(x)}}}{s^\lambda_\rho(x)m}\Big)\\\label{eq:lemma 6.1 1}
        &&\leq \frac{\delta}{c_1}\Big(c_3k^{1/2}+c_4+\frac{C_2}{m}\Big),
    \end{eqnarray}
   As for the remaining terms of \eqref{eq:lemma 6.1 0}, by \eqref{def: costante resti} we see that 
    \begin{equation}\label{eq:lemma 6.1 2}
        \limsup_{\rho\to 0^+}\frac{C_{\xi^\lambda_\rho(x)}}{s^\lambda_\rho(x)m}=\limsup_{\rho\to 0^+}\frac{2c_3Cs^\lambda_\rho(x)+2C(c_4+1)}{s^\lambda_\rho(x)m}=\frac{2c_3C}{m},
    \end{equation}
   \noindent while by \eqref{eq:def costante cxi} and \eqref{eq: behaviour of s at infinity cantor lemma} we get that
    \begin{equation}\label{eq:lemma 6.1 3}
        \limsup_{\rho\to 0^+}\frac{c_2}{c_1}\delta\big(1+2c_{\xi^\lambda_\rho(x),m}\lambda\rho\big)\frac{1}{s^\lambda_\rho(x)}=0.
    \end{equation}
      \noindent Finally, using \eqref{eq:lemma 6.1 1}-\eqref{eq:lemma 6.1 3} we can take the limsup in \eqref{eq:lemma 6.1 0} as $\rho\to 0^+$, obtaining an estimate that depends on $m$ and $\delta$. Then, taking the limit as $m\to+\infty$  we get \eqref{eq:claim lemma cantor}, concluding the proof.
   \end{proof}

   The next result ensures that estimates on the volume integrand $f$ translate into estimates on the right-hand side of \eqref{eq:Radon-Nykodim di Dcu}. For the proof of this result we refer the reader to \cite[Lemma 5.3]{DalToa23b}.

   \begin{lemma}\label{lemma:rettangoli}
   Let $E\in\E_{\rm sc}$,  $\xi\in \Rdk$,  $\lambda\geq 1$,   $\nu\in\mathbb{S}^{d-1}$,  $t>0$, and $\mu\in [0,+\infty)$. Assume that for every $x\in\Rd$ and $\rho>0$ we have
   \begin{equation*}
       m^E_{t\rho}(\ell_\xi,Q(x,\rho))\leq \mu\rho^d.
   \end{equation*}
   Then 
   \begin{equation*}
       m^E_{t\lambda\rho}(\ell_\xi,Q^\lambda_\nu(x,\rho))\leq \mu\lambda^{d-1}\rho^d
   \end{equation*}
   for every $x\in\Rd$ and $\rho>0$. If, in addition, there exists some $x_0\in\Rd$ such that
   \begin{equation*}
       \limsup_{\rho\to 0^+}\frac{m^E_{t\rho}(\ell_\xi,Q(x_0,\rho))}{\rho^d}=\mu,
   \end{equation*}
   then
    \begin{equation*}
       \limsup_{\rho\to 0^+}\frac{m^E_{t\lambda\rho}(\ell_\xi,Q^\lambda_\nu(x_0,\rho))}{\lambda^{d-1}\rho^d}=\mu.
   \end{equation*}
   \end{lemma}

With these two lemmas at hand, we are now ready to prove Theorem \ref{thm:Cantor}.
   \begin{proof}[Proof of Theorem \ref{thm:Cantor}]
      Thanks to Theorem \ref{thm:rappresentazione} and to \eqref{eq:dependence only on one variable}, for every $A\in\mathcal{A}_c(\Rd)$ we have 
      \begin{equation}
          \label{eq: representation ac Cantor }\displaystyle E^a(u,B)=\int_B\hat{f}(\nabla u)\, dx\quad \text{ and }\quad  E^s(u,B)=\int_{\jump{u}\cap B}g(x,[u],\nu_u)\,d\hd
      \end{equation}
      for every $u\in GBV_\star(A;\Rk)$ and $B\in\mathcal{B}(A)$.

      We are left with proving that for every $A\in\mathcal{A}_c(\Rd)$ we can represent the Cantor part $E^c$ as 
      \begin{equation}\label{eq:claim Cantor Theorem}
          E^c(u,B )=\int_{B}\hat{f}^\infty\Big(\frac{dD^cu}{d|D^cu|}\Big)\,d|D^cu|
      \end{equation}
      for every $u\in GBV_\star(A;\Rk)$ and $B\in\mathcal{B}(A)$. 
      
     Since $\ell_\xi$ is an admissible function for the minimisation problem $m^E(\ell_\xi,Q(x,\rho))$, by \eqref{eq: representation ac Cantor } we get that
    \begin{equation*}
        m^E(\ell_{\xi},Q(x,\rho))\leq \hat{f}(\xi)\rho^d 
    \end{equation*}
   for every $x\in\Rd,$ $\xi\in\Rkd$, and  $\rho>0$.
    Recalling \eqref{eq:definition of small f} and \eqref{eq:dependence only on one variable}, applying Lemma \ref{lemma:rettangoli} we obtain 
        \begin{equation*}
            \limsup_{\rho\to 0^+}\frac{m^E(\ell_\xi,Q^\lambda_\nu(x,\rho)))}{\lambda^{d-1}\rho^d}=\hat{f}(\xi)
        \end{equation*}
    for every $\lambda\geq 1$.
    In particular, 
    for every $t>0$ we have
    \begin{equation}\label{eq:identity for finfty in cantor}
        \limsup_{\rho\to 0^+}\frac{m^E(t\ell_{\xi},Q^\lambda_\nu(x,\rho))}{\lambda^{d-1}\rho^dt}=\frac{\hat{f}(t\xi)}{t}
    \end{equation}
   so that, taking the $\limsup$ for $t\to+\infty$, we get
    \begin{equation}\label{eq:finftycantor}
          \limsup_{t\to +\infty}\limsup_{\rho\to 0^+}\frac{m^E(t\ell_{\xi},Q^\lambda_\nu(x,\rho))}{\lambda^{d-1}\rho^dt}=\hat{f}^\infty(\xi).
    \end{equation}

     We claim that for every $\xi\in\Rkd$, $\lambda\geq 1$, and $x\in\Rd$ we have 
     \begin{equation}\label{eq:claim Cantor Thm}
         \limsup_{\rho\to 0^+}\frac{m^E(s_\rho \ell_{\xi},Q^\lambda_\nu(x,\rho))}{\lambda^{d-1}\rho^ds_\rho}=\hat{f}^\infty(\xi)
     \end{equation}
     whenever $s_\rho\to +\infty$ and $\rho s_\rho\to 0^+$ as $\rho\to 0^+$. To see this, first note that by \eqref{eq:definition of small f} and \eqref{eq:dependence only on one variable} for every $\rho>0$ we have
    $m^E(s_\rho\ell_{\xi},Q^\lambda_\nu(x,\rho))\leq \hat{f}(s_\rho \xi)\lambda^{d-1}\rho^d$,   whence
     \begin{equation*}
    \limsup_{\rho\to 0^+}\frac{m^E(s_\rho\ell_{\xi},Q^\lambda_\nu(x,\rho))}{\lambda^{d-1}\rho^ds_\rho}\leq \hat{f}^\infty(\xi).
     \end{equation*}
    
    Thus, to prove \eqref{eq:claim Cantor Thm} it is enough to show that 
     \begin{equation}\label{eq:right inequality cantor}
           \hat{f}^\infty(\xi) \leq\limsup_{\rho\to 0^+}\frac{m^E(s_\rho\ell_{\xi},Q^\lambda_\nu(x,\rho))}{\lambda^{d-1}\rho^ds_\rho}.
     \end{equation}
      Let $\eta>0$. By definition of $\hat{f}^\infty$, there exists $t_\eta>1/\eta$ such that
     \begin{equation}\label{eq:def good t cantor 1}
         \displaystyle \hat{f}^\infty(\xi)-\eta<\frac{\hat{f}(t_\eta\xi)}{t_\eta}.
     \end{equation}
     Hence, recalling \eqref{eq:identity for finfty in cantor} we have
     \begin{equation}\label{eq:fra eta e eta}
         \hat{f}^\infty(\xi)-\eta\leq  \limsup_{\rho\to 0^+}\frac{m^E(t_\eta\ell_{\xi},Q^\lambda_\nu(x,\rho))}{\lambda^{d-1}\rho^dt_\eta}.
     \end{equation}
     Let us fix $m\in\N$. By Lemma \ref{prop:closure ealpha} we obtain
         \begin{eqnarray}
    \nonumber 
       &&\limsup_{\rho\to 0^+}\frac{m^{E}(t_\eta\ell_\xi,Q^\lambda_\nu(x,\rho))}{\lambda^{d-1}\rho t_\eta}
        \leq\limsup_{\rho\to+\infty}\frac{m^{E}(s_\rho\ell_\xi,Q^\lambda_\nu(x,\rho))}{\lambda^{d-1}\rho^ds_\rho}\\&&\nonumber +
        \limsup_{\rho\to 0^+}\Big(\vartheta_m(s_\rho\kappa_{\xi,m}\lambda\rho)\frac{m^{E}(s_\rho\ell_\xi,Q^\lambda_\nu(x,\rho))}{\lambda^{d-1}\rho^ds_\rho}+\frac{c_7}{s_\rho^\alpha}\Big(\frac{m^E(s_\rho\ell_{\xi},Q^\lambda_\nu(x,\rho))}{\lambda^{d-1}\rho^ds_\rho}\Big)^{1-\alpha}+\frac{c_7}{s_\rho}\\&&\nonumber \quad
      +(1+\vartheta_m(
      s_\rho\kappa_{\xi,m}\lambda\rho))\frac{M_\xi}{m}
     \nonumber +\Big(\frac{1}{s_\rho^\alpha}+\frac{1}{s_\rho }\Big)\frac{M_\xi}{m^{1-\alpha}}\Big)
 \\&&\quad +\limsup_{\rho\to0^+}\Big(\vartheta_m(t_\eta\kappa_{\xi,m}\lambda\rho)\frac{m^E(t_\eta\ell_{\xi},Q^\lambda_\nu(x,\rho))}{\lambda^{d-1}\rho^dt_\eta}\nonumber 
        +\frac{c_7}{t_\eta^\alpha}\Big(\frac{m^E(t_\eta\ell_\xi,Q^\lambda_\nu(x,\rho))}{\lambda^{d-1}\rho^dt_\eta }\Big)^{1-\alpha}+\frac{c_7}{t_\eta}\\
   &&\quad \label{eq:conticione}+(1+\vartheta_m(t_\eta \kappa_{\xi,m}\lambda\rho))\frac{M_\xi}{m}
      +\Big(\frac{1}{t_\eta^\alpha}+\frac{1}{t_\eta }\Big)\frac{M_\xi}{m^{1-\alpha}}\Big),
     \end{eqnarray}
   where $M_\xi>0$ is the constant of Lemma \ref{prop:closure ealpha}, which we recall is independent of $\rho,s_\rho$, $t$, and $m$.

    We separately study the summands of this last expression. Since $\vartheta_m$ is continuous and $\vartheta_m(0)=0$, while $s_\rho\rho\to0$ and $s_\rho\to+\infty$ as $\rho\to 0^+$, recalling that $\frac{m^E(s_\rho\ell_\xi,Q^\lambda_\nu(x,\rho))}{\lambda^{d-1}\rho^ds_\rho}$ is bounded uniformly with respect to $\rho$  by (c2) of Definition \ref{def:space of functionals E}, we obtain
    \begin{eqnarray}
        &&\nonumber \limsup_{\rho\to 0^+}\Big(\vartheta_m(s_\rho\kappa_{\xi,m}\lambda\rho))\frac{m^{E}(s_\rho\ell_{\xi},Q^\lambda_\nu(x,\rho))}{\lambda^{d-1}\rho^ds_\rho}+\frac{c_7}{s_\rho^\alpha}\Big(\frac{m^E(s_\rho\ell_{\xi},Q^\lambda_\nu(x,\rho))}{\lambda^{d-1}\rho^ds_\rho}\Big)^{1-\alpha}\\  &&\label{eq:contoni 1 cantor}\hspace{2 cm}+\frac{c_7}{s_\rho}+(1+\vartheta_m(s_\rho\kappa_{\xi,m}\lambda\rho))\frac{M_\xi}{m}
      +\Big(\frac{1}{s_\rho^\alpha}+\frac{1}{s_\rho }\Big)\frac{M_\xi}{m^{1-\alpha}}\Big)=\frac{M_\xi}{m}.
    \end{eqnarray}
    \noindent 
    By (c2) of Definition \ref{def:space of functionals E} there exists $N_\xi>0$, independent of $\rho$ and $t_\eta$, such that 
    \begin{equation*}
        \frac{m^E(t_\eta\ell_{\xi},Q^\lambda_\nu(x,\rho))}{\lambda^{d-1}\rho^dt_\eta}\leq N_\xi.
    \end{equation*}
  Therefore, arguing as in the proof of \eqref{eq:contoni 1 cantor},  by  \eqref{eq:def good t cantor 1} and \eqref{eq:fra eta e eta},  we obtain
    \begin{eqnarray}\nonumber 
    &&\limsup_{\rho\to0^+}\Big( \vartheta_m(t_\eta\kappa_{\xi,m}\lambda\rho)\frac{m^E(t_\eta\ell_{\xi},Q^\lambda_\nu(x,\rho))}{\lambda^{d-1}\rho^dt_\eta}\nonumber 
        +\frac{c_7}{t_\eta^\alpha}\Big(\frac{m^E(t_\eta\ell_\xi,Q^\lambda_\nu(x,\rho))}{\lambda^{d-1}\rho^dt_\eta}\Big)^{1-\alpha}+\frac{c_7}{t_\eta}\\
        &&\hspace{1.5 cm}\nonumber +(1+\vartheta_m(t_\eta\kappa_{\xi,m}\lambda\rho))\frac{M_\xi}{m}
      +\Big(\frac{1}{t_\eta^\alpha}+\frac{1}{t_\eta}\Big)\frac{M_\xi}{m^{1-\alpha}}\Big)
      \\ && \hspace{0.5 cm}\label{eq:contoni 2 Cantor}\leq \frac{N_\xi^{1-\alpha}}{t_\eta^\alpha}+\frac{c_7}{t_\eta}+\Big(\frac{1}{t_\eta^\alpha}+\frac{1}{t_\eta}\Big)\frac{M_\xi}{m^{1-\alpha}}< N_\xi^{1-\alpha}\eta^{\alpha}+c_7\eta +(\eta+\eta^{\alpha})\frac{M_\xi}{m^{1-\alpha}},
    \end{eqnarray}
    where in the last inequality we have used that $t_\eta>1/\eta$.
    \noindent Finally, combining \eqref{eq:fra eta e eta} and \eqref{eq:contoni 2 Cantor}, we get
    \begin{equation*}
      \hat{f}^\infty(\xi)-\eta-\frac{M_\xi}{m}- N_\xi^{1-\alpha}\eta^{\alpha}+c_7\eta +(\eta+\eta^{\alpha})\frac{M_\xi}{m^{1-\alpha}}\leq \limsup_{\rho\to 0^+}\frac{m^{E}(s_\rho\ell_{\xi},Q^\lambda_\nu(x,\rho))}{\lambda^{d-1}\rho^ds_\rho}.
    \end{equation*}
    Letting $m\to+\infty$ and $\eta\to 0^+$, we obtain \eqref{eq:right inequality cantor}.

   Consider now a function $u\in BV(A;\Rk)$. Let $a_u(x)$, $\nu_u(x)$, and $s^\lambda_\rho(x)$ be as in \eqref{eq:Rank one Cantor part} and  \eqref{eq:def quantities in Cantor lemma}, and set $\xi(x):=a_u(x)\otimes\nu_u(x)$.  By Lemma \ref{lemma:Radon-Nykodim Cantor} for $|D^cu|$-a.e. $x\in A$ we have 
   \begin{equation}\label{eq:quasi finito cantor}
       \frac{dE^c(u,\cdot)}{d|D^cu|}(x)= \lim_{\lambda\to+\infty}\limsup_{\rho\to 0^+}\frac{m^{E}(s^\lambda_\rho(x)\ell_{\xi(x)},Q^\lambda_{\nu_u(x)}(x,\rho))}{\lambda^{d-1}\rho^ds^\lambda_\rho(x)}.
   \end{equation}
    Recalling \eqref{eq: behaviour of s at infinity cantor lemma}, in light of \eqref{eq:claim Cantor Thm} and of \eqref{eq:quasi finito cantor}, we then  infer that
    \begin{equation*}
   \frac{dE^c(u,\cdot)}{d|D^cu|}(x)=\hat{f}^\infty(a_{u}(x)\otimes\nu_u(x)),
    \end{equation*}
    which by \eqref{eq:Rank one Cantor part} gives  \eqref{eq:claim Cantor Theorem} for every $u\in BV(A;\Rk)$ and $B\in\mathcal{B}(A)$.

     Let us assume now that $u\in GBV_\star(A;\Rk)$ and for every $R>0$ we consider the set $A^R_u$ introduced in Proposition \ref{prop:Properties of derivatives of GBVstar}. We claim that it is enough to prove \eqref{eq:claim Cantor Theorem} for every $B\in\mathcal{B}(A)$ for which there exists $R>0$ such that $B\subset A^R_{u}$. Indeed, by Proposition \ref{prop: Properties of GBVsvector}  we have $A_{u,0}^{R}\nearrow A_{\rm reg}:=\{x\in A\colon \widetilde{u}(x) \,\,\text{exists}\}$ and that $|D^cu|(A\setminus A_{\rm reg})=0$. Hence, every $B\in\mathcal{B}(A)$ can be written, up to a $|D^cu|$-negligble set, as the increasing union of Borel sets each contained in $A^R_{u,0}$, for some $R$. This proves the claim.

     Let us fix $R>0$ and a Borel set $B\subset A^R_u$. Let  $R_m>R$ be a sequence with $R_m\to+\infty$. Thanks to Lemma \ref{lemma:passaggio al limite dalle troncate al non nei vari pezzi}, we have  
     \begin{equation}\label{eq:fine 7.1}
         \lim_{m\to+\infty}\frac{1}{m}\sum_{i=1}^mE^c(\psi^i_{R_m}\circ u,B)=E^c(u,B).
     \end{equation}
      Since each function $\psi^i_{R_m}\circ u$ belongs to $BV(A;\Rk)$
     and the integral representation holds in $BV(A;\Rk)$, we have
     \begin{equation}\label{eq:fine 7.1 2}
         E^c(\psi^i_{R_m}\circ u,B)=\int_B\hat{f}^\infty\Big (\frac{dD^c(\psi_{R_m}^{i}\circ u)}{d|D^c(\psi_{R_m}^{i}\circ u)|}\Big)\, d|D^c(\psi_{R_m}^{i}\circ u)|
     \end{equation}
     for every $i\in\{1,...,m\}$. Recalling that $B\subset A^R_u\subset A^{R_m}_u$, from Proposition  \ref{prop:Properties of derivatives of GBVstar} and from \eqref{eq:properties psiRi} we obtain that 
     \begin{eqnarray*}
        &\displaystyle \frac{dD^c(\psi^{i}_{R_m}\circ u)}{d|D^c(\psi^{i}_{R_m}\circ u)|}=\frac{dD^cu}{d|D^cu|} \quad |D^cu|\text{-a.e. in } A^R_u,\\
        &\displaystyle |D^c(\psi_{R_m}^i\circ u)|=|D^cu| \quad \text{as measures in $A^R_u$},
     \end{eqnarray*}
     for every $i\in\{1,...,m\}$. These equalities, together with \eqref{eq:fine 7.1} and \eqref{eq:fine 7.1 2},  give \eqref{eq:claim Cantor Theorem}, concluding  the proof. \end{proof}

\section{Integrands of the \texorpdfstring{$\Gamma$}{Gamma}-limits}\label{sec:integrands in the limit}

In this section we consider a  sequence  $(E_n)_n$ of functionals in $\E$. In the first part we assume that  $(E_n)_n$ $\Gamma$-converges to some functional $E\in\E_{\rm sc}$ and  we characterise  the bulk and surface integrands $f$ and $g$ of $E$ at a point $x\in\Rd$ by taking first the limit as $n\to+\infty$ of the infima of suitable minimisation problems for $E_n$ on small cubes and taking then the limit as these cubes shrink to $x$.

In the last part we assume that each  functional $E_n$ belongs to $\E^{\alpha,\vartheta}$ and  prove the converse of the previous results: if the limits mentioned above exist and are independent of $x$, then they define two integrands $f$ and $g$ such that the sequence $(E_n)_n$ $\Gamma$-converges to the functional corresponding to $f$ and $g$.

We begin this analysis by showing that it is possible to obtain the function $f$ introduced in \eqref{eq:definition of small f} by means of limits of constrained minimisation problems. 

\begin{lemma}\label{lemma:f come limiti troncati}
Let $E\in\E_{\rm sc}$ and let $f$ be the function defined by \eqref{eq:definition of small f}. Then there exists a $\Ld$-negligible set $N\in\mathcal{B}(\R^d)$, such that for every $x\in\Rd\setminus N$ and for every $\xi\in\Rkd$ we have
\begin{equation}\label{eq:Lemma con troncamento piu eta}
f(x,\xi)=\lim_{m\to+\infty}\limsup_{\rho\to0^+}\frac{m^E_{\rho c_{\xi,m}}(\ell_\xi,Q(x,\rho))}{\rho^{d}},
\end{equation}
where $c_{\xi,m}$ is the constant defined by \eqref{eq:def costante cxi} and $m^E_t(\ell_\xi,Q(x,\rho))$ is given by \eqref{eq:problema di minimo ausiliario troncato} with $t=c_{\xi,m}\rho$.
If, in addition, there exists a function $\hat{f}\colon \Rkd\to[0,+\infty)$ such that $f(x,\xi)=\hat{f}(\xi)$ for $\Ld$-a.e. $x\in\Rd$ and every $\xi \in \Rkd$, then \eqref{eq:Lemma con troncamento piu eta} holds for every $x\in\Rd$ and  $\xi\in \Rkd$.
\end{lemma}
\begin{proof}

It follows immediately from \eqref{eq:problema di minimo ausiliario} and \eqref{eq:problema di minimo ausiliario troncato} that for every $m\in\N$ and $\xi\in\Rdk$ we have
$m^E\leq m^E_{c_{\xi,m}\rho}$. Thus, we only have to prove that 
\begin{equation}\label{eq:claim del lemma troncazione di m per f}
\lim_{m\to +\infty}\limsup_{\rho\to0^+}\frac{m^E_{\rho c_{\xi,m}}(\ell_\xi,Q(x,\rho))}{\rho^d}\leq \limsup_{\rho\to0^+}\frac{m^E(\ell_\xi,Q(x,\rho))} {\rho^d}  
\end{equation}
for $\Ld$-a.e. $x\in \Rd$ and for every $\xi\in\Rkd$.
Thanks to Theorem \ref{thm:rappresentazione}, the equality \eqref{eq: Ea assoltamente continuo fittizio} is satisfied. Hence, by
   Corollary \ref{cor:troncature} there exists a set $N\in\mathcal{B}(\Rd)$, with $\Ld(N)=0$, satisfying the following property: for every $x\in\Rd\setminus N$,  $\xi\in\Rkd$, $m\in\N$ and $\rho>0$ small enough there exists a function $u\in BV(Q(x,\rho);\Rk)$ such that with tr$_{Q(x,\rho)}u=$tr$_{Q(x,\rho)}\ell_\xi$,  $\|u-\ell_\xi\|_{L^\infty(Q(x,\rho);\Rk)}\leq  c_{\xi,m}\rho$, and such that
\begin{equation*}
    m^E_{\rho c_{\xi,m}}(\ell_{\xi},Q(x,\rho))\leq m^E(\ell_\xi,Q(x,\rho))+\frac{C_\xi}{m}\rho^d,
\end{equation*}
where $C_\xi>0$ is the constant defined by \eqref{def: costante resti}.
Dividing this  inequality by $\rho^d$ and letting $\rho\to 0^+$, we conclude that 
\begin{equation*}
\limsup_{\rho\to^0+}\frac{m^E_{\rho c_{\xi,m}}(u,Q(x,\rho))}{\rho^d}\leq \limsup_{\rho\to 0^+}\frac{ m^E(u,Q(x,\rho))}{\rho^d}+\frac{C_\xi}{m}.
\end{equation*}
Taking the limit for $m\to +\infty$, we obtain \eqref{eq:claim del lemma troncazione di m per f}.

To conclude the proof, we note that under the additional hypothesis we have $N=\emptyset$ in Corollary \ref{cor:troncature}.
\end{proof}
The next result is useful to understand the relation between the minima of problems associated with a sequence $(E_n)_n$ on an open set $A'$ and the minimisation of the problem associated with their $\Gamma$-limit $E$, computed on a larger open set $A''$.

\begin{lemma}\label{eq:liminf and set inclusions}
Let $(E_n)_n\subset \E$, $E\in\E_{\rm sc}$, and let $A',A''\in\mathcal{A}_c(\Rd)$  with Lipschitz boundary and such that $A'\subset\subset A''$. Assume  that for every $A\in\mathcal{A}_c(\Rd)$, the sequence $E_n(\cdot,A)$ $\Gamma$-converges to $E(\cdot,A)$ with respect to the topology of $L^0(\Rd;\Rk)$. Then for every $w\in W^{1,1}_{\rm loc}(\Rd;\Rk)$, we have
\begin{equation*}
    m^E(w,A'')\leq \liminf_{n\to+\infty}m^{E_n}(w,A')+c_3k^{1/2}\int_{A\setminus A'}|\nabla w|\, dx+c_4\Ld(A''\setminus A').
\end{equation*}
\begin{proof}
    The proof can be obtained by adapting the arguments of \cite[Proposition 3.1]{DalToa23b}, replacing \cite[Theorem 7.13]{DalToa23} by \cite[Theorem 3.22]{donati2023new}.
\end{proof}
\end{lemma}

We now prove a result that allows to compare the limit of the minima of problems associated with a sequence $(E_n)_n$ on a cube $Q(x,\rho)$ with the minimum of the problem associated with the $\Gamma$-limit $E$, computed on the same cube.
\begin{lemma}\label{lemma:limsup sui cub inequality}
    Let $(E_n)_n\subset\E$, $E\in\E$, $x\in\Rd$, $\xi\in\Rkd$, $m\in\N$, $\rho>0$, and $s>d^{1/2}|\xi|\rho$. Assume that for every $A\in\mathcal{A}_c(\Rd)$ the sequence $E_n(\cdot,A)$ $\Gamma$-converges to $E(\cdot,A)$ with respect to the topology of $L^0(\Rd;\Rk)$. Then 
\begin{equation}
    \displaystyle\label{eq:lemma dati al bordo sul cubo 1}\limsup_{n\to +\infty}m_{t }^{E_n}(\ell_\xi,Q(x,\rho))\leq m_{s}^E(\ell_\xi,Q(x,\rho))+\frac{K_\xi}{m}\rho^d
\end{equation}
where $t:=2\sigma^m(s+d^{1/2}|\xi|\rho)+d^{1/2}|\xi|\rho$ and $K_\xi>0$ is a constant depending on $\xi$, but not on $(f_n)_n$, $(g_n)_n$, $m$, $s$, and $\rho$.
\end{lemma}
\begin{proof}
   Let us fix $0<\eta<1$. Consider a function $u\in BV(Q(x,\rho);\Rk)$, with$\|u-\ell_\xi\|_{L^\infty(Q(x,\rho);\Rk)}\leq s$ and $\text{tr}_{Q(x,\rho)}u=\text{tr}_{Q(x,\rho)}\ell_\xi$, such that 
    \begin{equation}\label{eq:bound 8.4}
        E(u,Q(x,\rho))\leq m^E_s(\ell_\xi,Q(x,\rho))+\eta\leq(c_3k^{1/2}|\xi|+c_4)\rho^d+\eta.
\end{equation}
    By Lemma \ref{lemma:recovery sequence troncate} there exists a sequence of functions $(v_n)_n\subset BV(Q(x,\rho);\Rk)$, with  $\|v_n\|_{L^\infty(Q(x,\rho))}\leq 2\sigma^m(s+d^{1/2}|\xi|\rho)$, such that $v_n\to u$ in $L^1(\Rd;\Rk)$
\begin{equation}\label{eq:bound 8.4 2}
    \limsup_{n\to+\infty}E_n(v_n,Q(x,\rho))\leq E(u,Q(x,\rho))+C\frac{E(u,Q(x,\rho))+\rho^d}{m}.
\end{equation}
Let us fix $0<\eta<1$ and $0<r<\rho$. For every $0<\delta\leq \eta$ we apply Lemma \ref{lemma:fundamental estimate} to the open sets $A=Q(x,r)$ and $U=Q(x,\rho)\setminus \overline{Q(x,r)}$ to obtain a sequence $(u_n)_n\subset BV(Q(x,\rho);\Rk)$ converging to $u$ in $L^1(\Rd;\Rk)$, with $\text{tr}_{Q(x,\rho)}u_n=\text{tr}_{Q(x,\rho)}\ell_\xi$ and $\|u_n-\ell_\xi\|_{L^\infty(Q(x,\rho);\Rk)}\leq  2\sigma^m(s+d^{1/2}|\xi|\rho)+d^{1/2}|\xi|\rho=t$,  such that
 \begin{equation*}
\limsup_{n\to+\infty}E_n(u_n,Q(x,\rho))\leq(1+\delta)\limsup_{n\to+\infty }\Big(E_n(v_n, Q(x,\rho))+E_n(\ell_\xi,Q(x,\rho)\setminus \overline{Q(x,r)})\Big)+\eta 
 \end{equation*} 
Exploiting \eqref{eq:bound 8.4} and \eqref{eq:bound 8.4 2}, from this last inequality we deduce that 
\begin{eqnarray}
\nonumber &\hspace{-2 cm}\displaystyle\limsup_{n\to+\infty}E_n(u_n,Q(x,\rho))\leq m^E_s(u,Q(x,\rho))+\eta+\delta\big((c_3k^{1/2}|\xi|+c_4)\rho^d+\eta\big)\\\label{eq:fine limsup lim}
&\displaystyle\quad\quad\qquad +(1+\delta)C\frac{(c_3k^{1/2}|\xi|+c_4+1)\rho^d+\eta}{m}+(1+\delta)(c_3k^{1/2}|\xi|+c_4)(\rho^d-r^d)+\eta.
\end{eqnarray}
Choosing $r$ so that $2(c_3k^{1/2}|\xi|+c_4)(\rho^d-r^d)\leq \eta$ and $\delta$ such that $\delta((c_3k^{1/2}|\xi|+c_4)\rho^d+\eta)\leq \eta$, recalling that $\text{tr}_{Q(x,\rho)}u_n=\text{tr}_{Q(x,\rho)}\ell_\xi$ and  $\|u_n-\ell_\xi\|_{L^\infty(Q(x,\rho);\Rk)}\leq t$, from \eqref{eq:fine limsup lim} we obtain 
\begin{equation*}
    \limsup_{n\to+\infty} m^{E_n}_{t}(\ell_\xi,Q(x,\rho))\leq m^E_s(u,Q(x,\rho))+ 2C\frac{(c_3k^{1/2}|\xi|+c_4+1)\rho^d+\eta}{m}+4\eta.
\end{equation*}
We conclude the proof by letting $\eta\to 0^+$.
\end{proof}

The next result shows that when $E$ is the $\Gamma$-limit of a sequence of functionals $(E_n)_n$,  the value of its bulk integrand $f$ at $(x,\xi)$ can be obtained by taking first the limit of $m^{E_n}(\ell_\xi,Q(x,\rho))/\rho^d$ as $n\to+\infty$ and then the limit as $\rho\to0^+$. For technical reasons, we need also a similar result where we replace $m^{E_n}(\ell_\xi,Q(x,\rho))$  by its constrained version  $m^{E_n}_{t}(\ell_{\xi},Q(x,\rho))$, for a suitable choice of $t>0$.

Given $m\in\N$ and $\xi\in\Rkd$, the constraint $t$ will be given by $\rho\beta_{\xi,m}$, with
\begin{equation}\label{eq:def betamxi}
\beta_{\xi,m}:=\sigma^m(c_{\xi,m}+d^{1/2}|\xi|)+d^{1/2}|\xi|,
\end{equation}
where $c_{\xi,m}>0$ is the constant defined by \eqref{eq:def costante cxi}.

\begin{proposition}\label{prop: integrands truncated}Let $(E_n)_n\subset \E$, $E\in \E_{\rm sc}$, and let $f$ be the function defined by \eqref{eq:definition of small f}.  Assume  that for every $A\in \mathcal{A}_c(\Rd)$ the sequence $E_n(\cdot,A)$ $\Gamma$-converges to $E(\cdot,A)$ with respect to the topology of $L^0(\Rd;\Rk)$. Then there exist an $\Ld$-negligible set $N\in\mathcal{B}(\Rd)$ such that for every $x\in\Rd\setminus N$ and $\xi\in \Rkd$ we have
\begin{eqnarray}\label{eq:small f as a not trucated limit}
    &&\displaystyle f(x,\xi)=\limsup_{\rho\to0^+}\liminf_{n\to\infty}\frac{m^{E_n}(\ell_{\xi},Q(x,\rho))}{\rho^d}=\limsup_{\rho\to0^+}\limsup_{n\to\infty}\frac{m^{E_n}(\ell_\xi,Q(x,\rho))}{\rho^d},\\
    &&\label{eq:Limite con m}\displaystyle \hspace{-1.2 cm}f(x,\xi)=\!\!\!\!\!\lim_{m\to+\infty}\!\!\limsup_{\rho\to0^+}\liminf_{n\to\infty}\frac{m^{E_n}_{\rho\beta_{\xi,m}}(\ell_{\xi},Q(x,\rho))}{\rho^d}=\!\!\!\!\!\lim_{m\to+\infty}\!\!\limsup_{\rho\to0^+}\limsup_{n\to\infty}\frac{m^{E_n}_{\rho\beta_{\xi,m}}(\ell_\xi,Q(x,\rho))}{\rho^d},
\end{eqnarray}
where $\beta_{\xi,m}$ is defined by \eqref{eq:def betamxi}.

If in addition there exists a function $\hat{f}\colon \Rkd\to[0,+\infty)$ such that $f(x,\xi)=\hat{f}(\xi)$ for  $\Ld$-a.e. $x\in\Rd$ and for every $\xi \in \Rkd$, then  \eqref{eq:small f as a not trucated limit} and \eqref{eq:Limite con m} hold for every $x\in\Rd$ and for every $\xi\in \Rkd$.  
\end{proposition}
\begin{proof}
Let $N\in\mathcal{B}(\Rd)$ be the union of the $\Ld$-negligible sets of Corollary \ref{cor:troncature} and Lemma \ref{lemma:f come limiti troncati}. We fix $m\in\N$, let $\rho>0$, and  set  $r:=\rho+\rho^2$. Using first Lemma \ref{eq:liminf and set inclusions}, then Lemma \ref{lemma:limsup sui cub inequality}, and finally Corollary \ref{cor:troncature}, for every $x\in\Rd\setminus N$ and  $\rho>0$ small enough we have 
\begin{eqnarray}
   \nonumber  m^{E}(\ell_\xi;Q(x,r))&\leq& \liminf_{n\to+\infty}m^{E_n}(\ell_\xi,Q(x,\rho))+(c_3k^{1/2}|\xi|+c_4)(r^d-\rho^d)\\
    \nonumber &\leq& \liminf_{n\to+\infty}m^{E_n}_{\rho\beta_{\xi,m}}(\ell_\xi,Q(x,\rho))+(c_3k^{1/2}|\xi|+c_4)(r^d-\rho^d)\\      
   \nonumber  &\leq& \limsup_{n\to+\infty}m^{E_n}_{\rho\beta_{\xi,m}}(\ell_\xi,Q(x,\rho))+(c_3k^{1/2}|\xi|+c_4)(r^d-\rho^d)\\
  \nonumber   &\leq& m^{E}_{\rho c_{\xi,m}}(\ell_\xi,Q(x,\rho))+(c_3k^{1/2}|\xi|+c_4)(r^d-\rho^d)+\frac{K_\xi}{m}\rho^d\\
 \nonumber   &\leq& m^{E}(\ell_\xi,Q(x,\rho))+(c_3k^{1/2}|\xi|+c_4)(r^d-\rho^d)+\frac{K_\xi+C_\xi}{m}\rho^d,
\end{eqnarray}
where $K_\xi$ is the constant of Lemma \ref{lemma:limsup sui cub inequality} and $C_\xi$ is given by \eqref{def: costante resti}.
We divide all terms of  the previous chain of inequalities by $\rho^d$, and take first the limsup for $\rho\to0^+$  and then the limit for $m\to+\infty$. By \eqref{eq:definition of small f}  in this way we obtain \eqref{eq:Limite con m}, since $(r^d-\rho^d)/\rho^d\to 0$ as $\rho\to0^+$. 

To prove \eqref{eq:small f as a not trucated limit}, one can simply replace the expression in the second line of the previous chain of inequalities by 
\begin{equation*}
    \limsup_{n\to+\infty}m^{E_n}(\ell_\xi,Q(x,\rho))+(c_3k^{1/2}|\xi|+c_4)(r^d-\rho^d).
\end{equation*}

If the additional hypothesis is satisfied, the last lines of Corollary \ref{cor:troncature} and of Lemma \ref{eq:small f as a not trucated limit} ensure that $N=\emptyset$, which concludes the proof.
\end{proof}

In the next proposition we show that an equality similar to \eqref{eq:small f as a not trucated limit} holds also for the surface integrand $g$.
\begin{proposition}\label{prop:jump integrand and gamma limits}
   Let $(E_n)_n\subset \E$, let  $E\in \E_{\rm sc}$, and let $g$ be given by \eqref{eq:definition of small g}. Assume  that for every $A\in \mathcal{A}_c(\Rd)$ the sequence $E_n(\cdot,A)$ $\Gamma$-converges to $E(\cdot,A)$ with respect to the topology of $L^0(\Rd)$. Then for every $x\in\Rd$,  $\zeta\in\Rk$, and  $\nu \in\mathbb{S}^{d-1}$ we have 
    \begin{equation}\label{eq:small g as a limit}
     g(x,\zeta,\nu)=\limsup_{\rho\to 0^+}\limsup_{n\to+\infty}\frac{m^{E_n}(u_{x,\zeta,\nu},Q_\nu(x,\rho))}{\rho^{d-1}}=\limsup_{\rho\to 0^+}\limsup_{n\to+\infty}\frac{m^{E_n}(u_{x,\zeta,\nu},Q_\nu(x,\rho))}{\rho^{d-1}}.
    \end{equation}
\end{proposition}
\begin{proof}
    The proof can be obtained with the same arguments of \cite[Proposition 3.3]{DalToa23b}.
\end{proof}
We conclude this section by stating a fundamental result for the proof of the homogenisation theorem, which will be the object of the next section. We show that for sequences of functions $(f_n)_n\subset\mathcal{F}^\alpha$ and $(g_n)_n\subset\mathcal{G}^\vartheta$, a sufficient conditions for $E^{f_n,g_n}$ to $\Gamma$-converge on every bounded open set is that \eqref{eq:Limite con m} and \eqref{eq:small g as a limit} hold and that the function $f$  is independent of $x$.
\begin{theorem}\label{thm:sufficient for Gammaconvergnece}
    Let $(E_n)_n\subset\E^{\alpha,\vartheta}$. Assume that there exist $\hat{f}\colon \Rkd\to[0,+\infty)$ and $\hat{g}\colon \Rd\times \mathbb{S}^{d-1}\to [0,+\infty)$ such that
    \begin{eqnarray}
    &&\nonumber \displaystyle \label{eq:sufficient 1 per homo} \hspace{-0.8 cm}\hat{f}(\xi)=\lim_{m\to+\infty}\limsup_{\rho\to0^+}\liminf_{n\to+\infty}\frac{m^{E_n}_{\rho \beta_{\xi,m}}(\ell_{\xi},Q(x,\rho))}{\rho^d}=\lim_{m\to+\infty}\limsup_{\rho\to0^+}\limsup_{n\to+\infty}\frac{m^{E_n}_{\rho\beta_{\xi,m}}(\ell_\xi,Q(x,\rho))}{\rho^d},\\
     &&\displaystyle \label{eq:sufficient for gamma convergence 2} \nonumber\hat{g}(x,\zeta,\nu)=\limsup_{\rho\to 0^+}\limsup_{n\to+\infty}\frac{m^{E_n}(u_{x,\zeta,\nu},Q_\nu(x,\rho))}{\rho^{d-1}}=\limsup_{\rho\to 0^+}\liminf_{n\to+\infty}\frac{m^{E_n}(u_{x,\zeta,\nu},Q_\nu(x,\rho))}{\rho^{d-1}},
    \end{eqnarray}
    for every $x\in\Rd$,  $\xi\in\Rkd$,  $\zeta\in\Rk$, and  $\nu\in\mathbb{S}^{d-1}$, where $\beta_{\xi,m}$ is given by \eqref{eq:def betamxi}. Then $\hat{f}\in\mathcal{F}^\alpha$, $\hat{g}\in\mathcal{G}$ and for every $A\in\mathcal{A}_c(\Rd)$  the sequence $E_n(\cdot,A)$ $\Gamma$-converges to $E^{\hat{f},\hat{g}}$ with respect to the topology of $L^0(\Rd;\Rk)$, where $E^{\hat{f},\hat{g}}$ is as in Definition \ref{def:Functionals Efg}.
\end{theorem}
\begin{proof} The proof follows closely the lines of \cite[Theorem 5.4]{DalToa23b}. By Theorem \ref{thm:compactness for E} there exists a subsequence, not relabelled, and a functional $E\in\E$ such that for every $A\in\mathcal{A}_c(\Rd)$ the sequence $E_n(\cdot, A)$ $\Gamma$-converges to $E$ with respect to the topology of $L^0(\Rd;\Rk)$ and by Proposition \ref{prop:limits are in Eweak} $E\in\ClosureE$. 
Thanks to Theorem \ref{thm:rappresentazione}, the functions $f$ and $g$ defined by \eqref{eq:definition of small f} and \eqref{eq:definition of small g}, respectively, satisfy \eqref{eq:rappresentazione ac} and  \eqref{eq:rappresentazione salto}, and $f\in\mathcal{F}^\alpha$ and $g\in\mathcal{G}$. Additionally, by Proposition \ref{prop: integrands truncated} and Proposition \ref{prop:jump integrand and gamma limits}, $f=\hat{f}$ and $g=\hat{g}$, so that \eqref{eq:dependence only on one variable} holds. Hence, by Theorem \ref{thm:Cantor} we have that $E=E^{\hat{f},\hat{g}}$. Since the functions $\hat{f}$ and $\hat{g}$ are independent of the chosen subsequence, for every $A\in\mathcal{A}_c(\Rd)$ by the Urysohn property of $\Gamma$-convergence (see \cite[Proposition 8.3]{DalMaso})  the original sequence $E_n(\cdot,A)$ $\Gamma$-converges to $E^{\hat{f},\hat{g}}(\cdot,A)$, concluding the proof.
\end{proof}
\section{Homogenisation of free-discontinuity functionals}\label{sec:homogenisation}

We are now ready to deal with the homogenisation of functionals in $\E^{\alpha,\vartheta}$. The arguments that we will make use of are based on the method devised by in \cite[Section 6]{DalToa23b}.
The main difference with respect to \cite{DalToa23b}, is the dependence on $m$ of the constant $\beta_{\xi,m}$ appearing in \eqref{thm:sufficient for Gammaconvergnece}, which, in the problem of stochastic homogenisation,  forces us to use the Subadditive Ergodic Theorem at $m$ fixed and only then to pass to the limit for $m\to+\infty$.

Before introducing the stochastic setting, we state a result which shows that, in the case of homogenisation, the sufficient conditions for $\Gamma$-convergence presented in Theorem \ref{thm:sufficient for Gammaconvergnece} can be rewritten in terms of cubes whose side length $r$ tends to $+\infty$. This formulation will be more suitable for the stochastic setting.

Given $f\in\mathcal{F^\alpha}$ and $g\in\mathcal{G}^\vartheta$  for every $\e>0$ we set  $f_\e(x,\xi):=f(x/\e,\xi)$ and $g_\e(x,\zeta,\nu):=g(x/\e,\zeta,\nu)$ for every $x\in\Rd$, $\xi\in\Rkd$, $\zeta\in\Rk$, $\nu\in\mathbb{S}^{d-1}$. We observe that $f_\e\in \mathcal{F}^\alpha$ and $g_\e\in\mathcal{G}^\vartheta$. We set $E_\e:=E^{f_\e,g_\e}$ according to Definition \ref{def:Functionals Efg}. We recall that $f^{\infty}$ and $g^0$ are the functions given by \eqref{eq:defrecession} and \eqref{eq:def g0},  that $f^\infty\in\mathcal{F}^\alpha$, while $g^0\notin \mathcal{G}$, since it does not satisfy (g3). 

The following theorem provides a general condition that guarantees the $\Gamma$-convergence  of $E_\e$ towards an integral functional whose integrands do not depend on $x$. We shall see in Remark \ref{re:periodic} that the hypotheses are satisfied in the case where $f$ and $g$ are periodic with respect to $x$. The advantage of these formulation is that these hypotheses are satisfied almost surely under the standard hypotheses of stochastic homogenisation.

\begin{theorem}\label{thm:cubes with r}
    Assume that there exists a function $g_{\rm hom}\colon\Rk\times\mathbb{S}^{d-1}\to [0,+\infty)$ and that for every $m\in\N$ there exists a function $f_{\rm hom}^m\colon\Rdk\to [0,+\infty)$ such that 
    \begin{eqnarray}
      &\hspace{1.5 cm} \displaystyle\label{eq:sufficient for homogenisation} \hspace{-2 cm}f_{\rm hom}^m(\xi)=\lim_{r\to +\infty}\frac{m^{E^{f,g^0}}_{r\beta_{\xi,m}}(\ell_\xi,Q(rx,r))}{r^d} \quad \text{ for every }x\in\Rd, \, \xi\in\Rkd,\\
       & \displaystyle \label{eq:sufficient g for homogenisation} g_{\rm hom}(\zeta,\nu)=\!\!\!\lim_{r\to +\infty}\frac{m^{E^{f^\infty,g}}(u_{rx,\zeta,\nu},Q_\nu(rx,r))}{r^{d-1}}\quad \text{ for every }x\in\Rd, \, \zeta\in\Rk, \text{ and }\nu\in \Sn^{d-1}.
    \end{eqnarray}
    Let  $f_{\rm hom}\colon\Rkd\to[0,+\infty)$ be the function defined by
    \begin{equation}\label{eq:def hat f}
        f_{\rm hom}(\xi):=\lim_{m\to+\infty}f_{\rm hom}^m(\xi)=\inf_{m\in\N}f_{\rm hom}^m(\xi)
    \end{equation}
  for every $\xi\in\Rkd $. Then $f_{\rm hom}\in\mathcal{F}^\alpha$, $g_{\rm hom }\in\mathcal{G}$, $E^{f_{\rm hom},g_{\rm hom }}\in\ClosureE$, and for every $\e_n\to 0^+$ and for every $A\in\mathcal{A}_c(\Rd)$  the sequence $E_{\e_n}(\cdot,A)$ $\Gamma$-converges to $E^{f_{\rm hom},g_{\rm hom}}(\cdot,A)$ with respect to the topology of $L^0(\Rd;\Rk)$. 
\end{theorem}
\begin{proof} We first observe that the limit in \eqref{eq:def hat f} exists, since $m\mapsto m^{E^{f,g^0}}_{r\beta_{\xi,m}}$ is non-increasing, hence, the same property holds for $m\mapsto f_{\rm hom}^m$.

 Let us fix a sequence $\e_n\to 0^+$ as $n\to +\infty$ and set $E_n:=E^{f_{\e_,},g_{\e_n}}$. To prove the result, it is enough to show that the hypotheses of Theorem \ref{thm:sufficient for Gammaconvergnece} are satisfied by $f_{\rm hom}$ and $g_{\rm hom}$. 
 Since the hypothesis concerning $g_{\rm hom}$ can checked by repeating verbatim the arguments of  \cite[Theorem 6.3]{DalToa23b}, we only prove that the hypothesis concerning $f_{\rm hom}$ are satisfied.

 To this aim, we fix $m\in\N$, $\rho>0$ and set $r_n:=\rho/\e_n$. The same computations performed in \cite[Lemma 6.1]{DalToa23b} show that for every $\e\in (0,1)$, $x\in\Rd$, $\xi\in\Rkd$ we have
 \begin{equation}\label{eq:claim formula di cella ac}
       \big|m^{E_\e}_{\rho\beta_{\xi,m}}(\ell_\xi,Q(x,\rho))-\e^dm_{(\rho/\e)\beta_{\xi,m}}^{E^{f,g^0}}(\ell_\xi,Q(x/\e,\rho/\e))\big|\leq K_\xi\vartheta(2\rho\beta_{\xi,m})\rho^d,
    \end{equation}
 where $K_\xi>0$ is a constant depending on $\xi$, but not on $m$$,\rho$, and $x$. Let us fix $x\in\Rd$ and $\xi\in\Rkd$. Using twice the previous inequality with $\e$ replaced by $\e_n$, by \eqref{eq:sufficient for homogenisation} we get that 
\begin{eqnarray*}
   && \hspace{-0.5 cm}f_{\rm hom}^m(\xi)\rho^d-K_\xi(2\beta_{\xi,m}\rho)\rho^d=\lim_{n\to+\infty}(\rho/r_n)^dm_{r_n\beta_{\xi,m}}^{E^{f,g^0}}(\ell_\xi,Q(r_nx/\rho,r_n))-K_\xi\vartheta(2\beta_{\xi,m}\rho)\rho^d\\&&
    \leq \liminf_{n\to+\infty}m_{\rho\beta_{\xi,m}}^{E_{\e_n}}(\ell_\xi,Q(x,\rho))\leq\limsup_{n\to+\infty}m_{\rho\beta_{\xi,m}}^{E_{\e_n}}(\ell_\xi,Q(x,\rho))\\&&
    \leq\lim_{n\to+\infty}(\rho/r_n)^dm_{r_n\beta_{\xi,m}}^{E^{f,g^0}}(\ell_{\xi},Q(r_nx/\rho,r_n))+K_\xi\vartheta(2\beta_{\xi,m}\rho)\rho^d=f_{\rm hom}^m(\xi)\rho^d+K_\xi\vartheta(2\beta_{\xi,m}\rho)\rho^d.
\end{eqnarray*}
Since $\vartheta$ is continuous and $\vartheta(0)=0$, dividing by $\rho^d$ and taking the limsup for $\rho\to0^+$ we obtain 
\begin{equation*}
    f_{\rm hom}^m(\xi)=\lim_{\rho\to0^+}\liminf_{n\to+\infty}\frac{m^{E_n}_{\rho\beta_{\xi,m}}(\ell_{\xi},Q(x,\rho))}{\rho^d}=\lim_{\rho\to0^+}\limsup_{n\to+\infty}\frac{m^{E_n}_{\rho\beta_{\xi,m}}(\ell_\xi,Q(x,\rho))}{\rho^d}.
\end{equation*}
Taking the limit as  $m\to+\infty$ we obtain that the hypothesis for $f_{\rm hom}$ in Theorem \ref{thm:sufficient for Gammaconvergnece} is satisfied, so that the proof is concluded.
\end{proof}

The following result shows that hypothesis \eqref{eq:sufficient for homogenisation} can be slightly weakend
\begin{lemma}\label{lemma:is enough to check rationals}
    Assume that for every $m\in\N$, $x\in\Rd$, and  $\xi\in\mathbb{Q}^{k\times d}$, the space of $k\times d$ matrices with rational entries, the limit
    \begin{equation*}
        f_{\rm hom}^{m}(\xi):=\lim_{r\to +\infty}\frac{m^{E^{f,g^0}}_{r\beta_{\xi,m}}(\ell_\xi,Q(rx,r))}{r^d}
    \end{equation*}
    exists and is independent of $x$. Then the function $f_{\rm hom}^m$ can be extended to a continuous function, still denoted by $f_{\rm hom}^m$, defined on the whole $\Rkd$ and such that \eqref{eq:sufficient for homogenisation} is satisfied.
\end{lemma}
\begin{proof}
    It is enough to repeat for every $m\in\N$ the arguments of \cite[Lemma 6.4]{DalToa23b}.
\end{proof}

We now introduce the stochastic setting in which we are going to deal with the homogenisation problem.

We fix a probability space  $(\Omega, \mathcal{T}, P )$ and a group $(\tau_z )_{z\in\Z^d}$ of $P$-preserving transformations on
 $(\Omega,\mathcal{T},P)$; that is,  a family $(\tau_z )_{z\in\Z^d}$ of $\mathcal{T}$-measurable bijective maps $\tau_z\colon \Omega\to\Omega$ such that
 \begin{enumerate}
\item [(a)]  for every $E \in\mathcal{T}$ and every $z\in\>\Z^d$ we have $P (\tau_z^{-1} (E))= P(E)$ ;
\item [(b)]  $\tau_0=\text{id}$, the identity map on $\Omega$ and for every $z,z'\in\Z^d$ one has $\tau_z\circ\tau_z'=\tau_{z+z}$.
 \end{enumerate}
A group $(\tau_z)_{z\in\Z^d}$ of $P$-preserving transformations is said to be ergodic if for every set  $E\in\mathcal{T}$ with the property that $\tau_z(E)=E$ for every $z\in\Z^d$, has probability either $0$ or $1$.
 In analogy with \cite{DalToa23b}, we introduce two classes of stochastic integrands.
 \begin{definition}\label{def:stochastic integrands}
      $\mathcal{SF^\alpha}$ is the collection of all  $\mathcal{T}\otimes\mathcal{B}(\Rd\times\Rkd)$-measurable functions $f\colon\Omega\times\Rd\times\Rkd\to[0,+\infty)$ such that for every $\omega\in\Omega$ the function $f(\omega):=f(\omega,\cdot,\cdot)$ belongs to $\mathcal{F}^\alpha$  the following stochastic periodicity holds: for every $\omega\in\Omega$, $z\in\Z^d$, $x\in\Rd$, and $\xi\in\Rkd$ we have
         \begin{equation*}
    f(\omega,x+z,\xi)=f(\tau_z(\omega),x,\xi).
       \end{equation*}
    $\mathcal{SG^\vartheta}$ is 
    the collection of all $\mathcal{T}\otimes\mathcal{B}(\Rd\times\Rk\times\mathbb{S}^{d-1})$-measurable functions $g\colon\Omega\times\Rd\times\Rk\times\mathbb{S}^{d-1}\to[0,+\infty)$ such that for every $\omega\in\Omega$ the function $g(\omega):=g(\omega,\cdot,\cdot,\cdot)$ belongs to $\mathcal{G}^\vartheta$, and the following stochastic periodicity holds:
     for every $\omega\in\Omega$, $z\in\Z^d$, $x\in\Rd$, $\zeta\in\Rk$, and $\nu\in\mathbb{S}^{d-1}$ we have
         \begin{equation*} g(\omega,x+z,\zeta,\nu)=g(\tau_z(\omega),x,\zeta,\nu).
       \end{equation*}
 \end{definition}

We now give the definition of subadditive process.
Before doing this, we introduce $\mathcal{R}$ the collection of rectangles defined by
\begin{equation*}
    \mathcal{R}:=\{R\in \Rd\colon R=[a_1,b_1)\times...\times [a_d,b_d), \text{ for some }a,b\in\Rd \text{ with } a_i<b_i \text{ for }i\in\{1,...,d\}\}.
\end{equation*}
 We also introduce $(\Omega,\widehat{\mathcal{T}},\widehat{P})$ the completion of $(\Omega,\mathcal{T},P)$. It is immediate to see that $(\tau_z)_{z}$ is a group of $P$-preserving transformation on $(\Omega,\widehat{\mathcal{T}},\widehat{P})$. 
\begin{definition}
A function $\mu\colon\Omega\times \mathcal{R}\to\R$ is  said to be a covariant subadditive process with respect to $(\tau_{z})_{z\in\Z^d}$ if the following properties are satisfied
 \begin{enumerate}
\item[(a)]  for every $R\in\mathcal{R}$ the function $\mu(\cdot, R)$  is $\widehat{\mathcal{T}}$-measurable;
\item[(b)]  for every $\omega\in\Omega$, $R \in\mathcal{R}$, and $z \in \Z^d$ we have  $\mu(\omega,R + z)=\mu(\tau_z (\omega), R)$;
\item[(c)] given $R\in\mathcal{R}$ and  a finite partition $(R_i)_{i=1}^n\subset\mathcal{R}$ of $R$, we have
\begin{equation*}
   \mu(\omega,R)\leq \sum_{i=1}^n\mu(\omega,R_i)
\end{equation*}
 for every $\omega\in\Omega$;
\item[(d)] there exists $C>0$ such that $0\leq \mu(\omega,R)\leq C\Ld(R)$ for every $\omega\in\Omega$
and $R\in\mathcal{R}$.
 \end{enumerate} 
\end{definition}

We recall the Subadditive Ergodic Theorem of Ackoglu and Krengel {\cite[Theorem 2.7]{Krengel}}. For the particular version here used we refer the reader to \cite[Proposition 1]{DalMod} (see also  \cite{LichtGerard}).
 \begin{theorem}\label{thm:subadditive ergodic theorem}
 Let $\mu$ be a subadditive process with respect to the group $(\tau_z)_{z\in\Z^d}$. Then there exist a $\mathcal{T}$-measurable set $\Omega'$, with $P(\Omega')=1$, and a function $\varphi\colon\Omega\to[0,+\infty)$ such that
 \begin{equation*}
     \lim_{r\to+\infty} \frac{\mu(\omega,Q(rx,r))}{r^d}=\varphi(\omega)
\end{equation*}
for every $x\in \Rd$ and every $\omega \in \Omega' $. If the group $(\tau_z )_{z\in\Z^d}$ is also ergodic, then $\varphi$ is constant $P$-a.e.
 \end{theorem}

Let $f\in\mathcal{SF}^\alpha$ and $g\in\mathcal{SG}^\vartheta$. For every $\omega\in\Omega$, we set $f^\infty(\omega):=f(\omega)^\infty$ and $g^0(\omega):=g(\omega)^0$.

The following lemma shows that it is possible to define a subadditive process closely related to condition \eqref{eq:sufficient for homogenisation}.
 \begin{lemma}\label{lemma:is a subadditive process}
     Let $f\in\mathcal{SF}^\alpha$, let $g\in\mathcal{SG}^\vartheta$, let $\xi\in\Rkd$, and let $m\in\N$. For every $R\in\mathcal{R}$ let $\rho(R)$ be the length of the longest of its sides.  Then the function $\Phi_{\xi,m}\colon\Omega\times\mathcal{R}\to[0,+\infty)$ defined by 
     \begin{equation}\label{eq:def processo Phi}
         \Phi_{\xi,m}(\omega,R):=m^{E^{f(\omega),g^0(\omega)}}_{\rho(R)\beta_{\xi,m}}(\ell_\xi,R^\circ)
     \end{equation}
     is a covariant subadditive process.
 \end{lemma}
 \begin{proof}
The proof can be obtained by arguing exactly as in \cite[Lemma 6.9]{DalToa23b}.
     \end{proof}
     With this lemma at hand, we are ready to show that   condition \eqref{eq:sufficient for homogenisation} with $f$ replaced by $f(\omega)$ and $g^0$ replaced by $g^0(\omega)$ is satisfied for $P$-a.e $\omega\in\Omega$  and for every $m\in\N$.
\begin{proposition}\label{prop:stochatstic ac integrand}
    Let $f \in \mathcal{SF}^\alpha$ and  $g \in\mathcal{SG}^\vartheta$. Then there exist a $\mathcal{T}$-measurable set $\Omega'$, with $P(\Omega')=1$, such that for every $m\in\N$ there exists a function $f_{\rm hom}^m\colon\Omega\times\Rkd\to [0, +\infty)$, with $f_{\rm hom}^m (\cdot, \xi)$   $\mathcal{T}$-measurable for every $\xi\in\Rkd$, such that 
\begin{equation*}
\lim_{r\to+\infty}\frac{m^{E^{f(\omega),g^0(\omega)}}_{r\beta_{\xi,m}}(\ell_\xi,Q(rx,r))}{r^d}=f_{\rm hom}^m(\omega,\xi)
\end{equation*}
for every $\omega\in\Omega'$, $x\in\Rd$, and $\xi\in\Rkd$. Moreover, the function $f_{\rm hom}\colon\Omega\times\Rkd\to[0,+\infty)$  defined by
\begin{equation}\label{eq:claim parte ac ergodico}
f_{\rm hom}(\omega,\xi)  :=\lim_{m\to+\infty}f_{\rm hom}^m(\omega,\xi)=\inf_{m\in\N}f_{\rm hom}^m(\omega,\xi)
\end{equation}
belongs to $\mathcal{SF}^\alpha$.
\noindent If, in addition,  $(\tau_z)_{z\in\Z^d}$ is ergodic, by choosing $\Omega'$ appropriately, we have that $f_{\rm hom}^m$ and $f_{\rm hom}$ are
independent of $\omega$.
\end{proposition}
\begin{proof}
    By Lemma \ref{lemma:is a subadditive process} for every $m\in\N$ and for every $\xi\in\mathbb{Q}^{k\times d}$ the function $\Phi_{\xi,m}$ defined by \eqref{eq:def processo Phi} is a covariant subadditive process. Hence, by the Subadditive Ergodic Theorem \ref{thm:subadditive ergodic theorem}, there exists a $\mathcal{T}$-measurable set $\Omega'$, with $P(\Omega')=1$, and for every $m\in\N$ a function $f_{\rm hom}^m\colon \Omega\times \mathbb{Q}^{k\times d}\to [0,+\infty)$, with $f_{\rm hom}^m(\cdot,\xi)$  $\mathcal{T}$-measurable for every $\xi\in\mathbb{Q}^{k\times d}$,  such that 
    \begin{equation}\label{eq:def hat f m}
    \lim_{r\to+\infty}\frac{m^{E^{f(\omega),f^0(\omega)}}_{r\beta_{\xi,m}}(\ell_\xi,Q(rx,r))}{r^d}=f_{\rm hom}^m(\omega,\xi)
    \end{equation}
    for every $\omega\in\Omega'$, $m\in\N$, $x\in\Rd$, and $\xi\in\mathbb{Q}^{k\times d}$. Thanks to Lemma \ref{lemma:is enough to check rationals}, the function $f_{\rm hom}^m$ can be extended to a function $f_{\rm hom}^m\colon\Omega\times \Rkd\to[0,+\infty)$, $\mathcal{T}$-measurable with respect to $\omega$ and continuous with respect to $\xi$, such that \eqref{eq:def hat f m} holds for every $\omega\in\Omega'$, $m\in\N$, $x\in\Rd$, and $\xi\in\Rkd$. We now fix $\omega_0\in\Omega$ and we redefine $f_{\rm hom}^m$ on $\Omega\setminus\Omega'\times \Rkd$, by setting $f_{\rm hom}^m(\omega,\xi)=f_{\rm hom}^m(\omega_0,\xi)$for every $\omega\in\Omega$ and $\xi\in\Rkd$. Note that for every $\omega\in\Omega$ and $\xi\in\Rkd$,  the sequence $f^m(\omega,\xi)$ is non-increasing with respect to $m$, which justifies \eqref{eq:claim parte ac ergodico} and that $f_{\rm hom}(\cdot,\xi)$ is $\mathcal{T}$-measurable. By Theorem \ref{thm:cubes with r} for every $\omega\in\Omega$ the function $f_{\rm hom}(\omega)\in\mathcal{F}^\alpha$, hence,  $f_{\rm hom}\in\mathcal{SF}^\alpha$.

   If, in addition,  $(\tau_z)_{z\in\Z^d}$ is  ergodic, then $f_{\rm hom}^m(\cdot,\xi)$ is constant $P$-a.e. for every $m$. This leads to the last sentence of the statement, concluding the proof.
\end{proof}

The following proposition shows that $P$-a.e. in $\Omega$ condition \eqref{eq:sufficient g for homogenisation} is also satisfied.
\begin{proposition}\label{prop:stochastic surface integrand}
   Let $f \in \mathcal{SF}^\alpha$, let  $g \in\mathcal{SG}^\vartheta$. Then  there exists a $\mathcal{T}$-measurable set $\Omega'$, with $P(\Omega')=1$, 
and a $\mathcal{T}\otimes\mathcal{B}(\Rk\times\mathbb{S}^{d-1})$-measurable function $\hat{g}\colon\Omega\times\Rk\times \Sn^{d-1}\to [0, +\infty)$  such that
\begin{equation}\label{eq:def hat g}
            \lim_{r\to+\infty}\frac{m^{E^{f^\infty(\omega),g(\omega)}}(u_{rx,\zeta,\nu},Q_\nu(rx,r))}{r^{d-1}}=\hat{g}(\omega,\zeta,\nu)
\end{equation}
for every $\omega\in\Omega'$, $x\in \Rd$, $\zeta\in\Rk$, and $\nu\in\mathbb{S}^{d-1}$. If, in addition, the group $(\tau_z )_{z\in\Z^d}$ is ergodic, by choosing $\Omega'$ appropriately, we have that 
the function $g$ is independent of $\omega$.
\begin{proof}
    The result is proved by adapting the same arguments used in \cite[Proposition 9.3, Proposition 9.4, Proposition 9.5]{CagnettiGlobal}.
\end{proof}
\end{proposition}

Combining Theorem \ref{thm:sufficient for Gammaconvergnece} with Propositions \ref{prop:stochatstic ac integrand} and Proposition \ref{prop:stochastic surface integrand}, we are finally able to obtain the desired stochastic homogenisation theorem.
\begin{theorem}\label{thm:Stochatsic homogenisation}
    Let $f \in \mathcal{SF}^\alpha$, let  $g \in\mathcal{SG}^\vartheta$, and for every $\e>0$ and $\omega\in\Omega$ let $E^\omega_\e:=E^{f_\e(\omega),g_\e(\omega)}$, according to Definition \ref{def:Functionals Efg}. Let $ f_{\rm hom}\colon\Omega\times \Rkd\to[0,+\infty)$ and $g_{\rm hom}\colon \Omega\times \Rk\times \mathbb{S}^{d-1}\to[0,+\infty)$ be the functions given in Propositions \ref{prop:stochatstic ac integrand} and \ref{prop:stochastic surface integrand}. Then for  the function $f_{\rm hom}(\omega,\cdot)\in\mathcal{F}^\alpha$ and $g_{\rm hom}(\omega,\cdot,\cdot)\in\mathcal{G}$ for every $\omega\in\Omega$ and there exists  $\mathcal{T}$-measurable set $\Omega'$, with $P(\Omega')=1$,  such that
     for every sequence $\e_n\to 0^+$, $\omega\in\Omega'$, and  $A\in\mathcal{A}_c(\Rd)$, the sequence $E^\omega_{\e_n}(\cdot,A)$ $\Gamma$-converges to $E^{f_{\rm hom}(\omega),g_{\rm hom}(\omega)}(\cdot,A)$ in the topology of $L^0(\Rd;\Rk)$.
  
    If, in addition, the group $(\tau_z)_{z\in\Z^d}$ is ergodic,  by choosing $\Omega'$ appropriately,   the functions $f_{\rm hom}$ and $g_{\rm hom}$ are independent of $\omega$.
\end{theorem}
\begin{remark}\label{re:periodic}
The periodic homogenisation in the deterministic case follows immediately from Theorem \ref{thm:Stochatsic homogenisation} once we note that in the case where $\Omega$ consists of a single point and $\tau_z=Id$ for every $z\in\Z^d$, the stochastic periodicity reduces to the 1-periodicity in each variable.
\end{remark}

\noindent {\it Acknowledgments.}
This article is based on work supported by the National Research
Project PRIN 2022J4FYNJ “Variational methods for stationary and evolution problems
with singularities and interfaces” funded by the Italian Ministry of University and Research.
The authors are members of GNAMPA of INdAM.

\bibliographystyle{abbrv}
\bibliography{Sources}

\begin{thebibliography}{10}

\bibitem{Krengel}
M.~A. Akcoglu and U.~Krengel.
\newblock Ergodic theorems for superadditive processes.
\newblock {\em J. Reine Angew. Math.}, 323:53--67, 1981.

\bibitem{alberti_1993}
G.~Alberti.
\newblock Rank one property for derivatives of functions with bounded variation.
\newblock {\em Proc. Roy. Soc. Edinburgh Sect. A}, 123(2):239--274, 1993.

\bibitem{RankOne4thapproach}
G.~Alberti, M.~Cs\"ornyei, and D.~Preiss.
\newblock Structure of null sets in the plane and applications.
\newblock In {\em European {C}ongress of {M}athematics}, pages 3--22. Eur. Math. Soc., Z\"urich, 2005.

\bibitem{AmbrosioComp1}
L.~Ambrosio.
\newblock A compactness theorem for a new class of functions of bounded variation.
\newblock {\em Boll. Un. Mat. Ital. B (7)}, 3(4):857--881, 1989.

\bibitem{AmbroExist}
L.~Ambrosio.
\newblock Existence theory for a new class of variational problems.
\newblock {\em Arch. Ration. Mech. Anal.}, 111(4):291--322, 1990.

\bibitem{AmbrosioANew}
L.~Ambrosio.
\newblock A new proof of the {SBV} compactness theorem.
\newblock {\em Calc. Var. Partial Differential Equations}, 3(1):127--137, 1995.

\bibitem{AmbFuscPall}
L.~Ambrosio, N.~Fusco, and D.~Pallara.
\newblock {\em Functions of bounded variation and free discontinuity problems}.
\newblock Oxford Mathematical Monographs. The Clarendon Press, Oxford University Press, New York, 2000.

\bibitem{Braides1995}
G.~Bouchitt\'e, A.~Braides, and G.~Buttazzo.
\newblock Relaxation results for some free discontinuity problems.
\newblock {\em J. Reine Angew. Math.}, 458:1--18, 1995.

\bibitem{bouchitte1998global}
G.~Bouchitt\'e, I.~Fonseca, and L.~Mascarenhas.
\newblock A global method for relaxation.
\newblock {\em Arch. Ration. Mech. Anal.}, 145(1):51--98, 1998.

\bibitem{BourFrancMari}
B.~Bourdin, G.~A. Francfort, and J.-J. Marigo.
\newblock {\em The variational approach to fracture}.
\newblock Springer, New York, 2008.
\newblock [reprinted from. J. Elasticity 91, 5–148 (2008)].

\bibitem{BraidesApprox}
A.~Braides.
\newblock {\em Approximation of free-discontinuity problems}, volume 1694 of {\em Lecture Notes in Mathematics}.
\newblock Springer-Verlag, Berlin, 1998.

\bibitem{BraDefVit}
A.~Braides, A.~Defranceschi, and E.~Vitali.
\newblock Homogenization of free discontinuity problems.
\newblock {\em Arch. Ration. Mech. Anal.}, 135(4):297--356, 1996.

\bibitem{CagnettiDetFree}
F.~Cagnetti, G.~Dal~Maso, L.~Scardia, and C.~I. Zeppieri.
\newblock {$\Gamma$}-convergence of free-discontinuity problems.
\newblock {\em Ann. Inst. H. Poincar\'{e} C Anal. Non Lin\'{e}aire}, 36(4):1035--1079, 2019.

\bibitem{CagnettiStocFree}
F.~Cagnetti, G.~Dal~Maso, L.~Scardia, and C.~I. Zeppieri.
\newblock Stochastic homogenisation of free-discontinuity problems.
\newblock {\em Arch. Ration. Mech. Anal.}, 233(2):935--974, 2019.

\bibitem{CagnettiGlobal}
F.~Cagnetti, G.~Dal~Maso, L.~Scardia, and C.~I. Zeppieri.
\newblock A global method for deterministic and stochastic homogenisation in {$BV$}.
\newblock {\em Ann. PDE}, 8(1):Paper No. 8, 89, 2022.

\bibitem{CeladaDalM}
P.~Celada and G.~Dal~Maso.
\newblock Further remarks on the lower semicontinuity of polyconvex integrals.
\newblock {\em Ann. Inst. H. Poincar\'e{} C Anal. Non Lin\'eaire}, 11(6):661--691, 1994.

\bibitem{DalMaso}
G.~Dal{ }Maso.
\newblock {\em An Introduction to $\Gamma$-convergence}.
\newblock Birkh\"auser, Basel, 1990.

\bibitem{DalMod}
G.~Dal~Maso and L.~Modica.
\newblock Nonlinear stochastic homogenization and ergodic theory.
\newblock {\em J. Reine Angew. Math.}, 368:28--42, 1986.

\bibitem{DalToa22}
G.~Dal~Maso and R.~Toader.
\newblock A new space of generalised functions with bounded variation motivated by fracture mechanics.
\newblock {\em NoDEA Nonlinear Differential Equations Appl.}, 29(Paper No. 63.), 2022.

\bibitem{DalToa23b}
G.~Dal~Maso and R.~Toader.
\newblock Homogenisation problems for free discontinuity functionals with bounded cohesive surface term.
\newblock \url{http://cvgmt.sns.it/paper/6125/}, 2023.
\newblock cvgmt preprint.

\bibitem{DalToa23}
G.~Dal~Maso and R.~Toader.
\newblock Gamma-convergence and integral representation for a class of free discontinuity functionals.
\newblock {\em J. Convex Anal.}, 31:411--476, 2024.

\bibitem{DeGiorgiLetta}
E.~De~Giorgi and G.~Letta.
\newblock Une notion g\'{e}n\'{e}rale de convergence faible pour des fonctions croissantes d'ensemble.
\newblock {\em Ann. Scuola Norm. Sup. Pisa Cl. Sci. (4)}, 4(1):61--99, 1977.

\bibitem{DePhiRind}
G.~De~Philippis and F.~Rindler.
\newblock On the structure of {$\mathcal{A}$}-free measures and applications.
\newblock {\em Ann. of Math. (2)}, 184(3):1017--1039, 2016.

\bibitem{donati2023new}
D.~Donati.
\newblock A new space of generalised vector-valued functions of bounded variation.
\newblock \url{ https://doi.org/10.48550/arXiv.2310.11538}, 2023.

\bibitem{DonFri23}
A.~F. Donnarumma and M.~Friedrich.
\newblock Stochastic homogenisation for functionals defined on asymptotically piecewise rigid functions.
\newblock \url{https://arxiv.org/abs/2312.12082}, 2023.

\bibitem{Dugdale1960YieldingOS}
D.~S. Dugdale.
\newblock Yielding of steel sheets containing slits.
\newblock {\em Journal of The Mechanics and Physics of Solids}, 8:100--104, 1960.

\bibitem{Evans2015}
L.~C. Evans and R.~F. Gariepy.
\newblock {\em Measure theory and fine properties of functions, revised edition}.
\newblock CRC Press, 2015.

\bibitem{FonsecaMuller}
I.~Fonseca and S.~M\"uller.
\newblock Quasi-convex integrands and lower semicontinuity in {$L^1$}.
\newblock {\em SIAM J. Math. Anal.}, 23(5):1081--1098, 1992.

\bibitem{francfort1998revisiting}
G.~A. Francfort and J.-J. Marigo.
\newblock Revisiting brittle fracture as an energy minimization problem.
\newblock {\em J. Mech. Phys. Solids}, 46(8):1319--1342, 1998.

\bibitem{FriedrichPeruSolo}
M.~Friedrich, M.~Perugini, and F.~Solombrino.
\newblock {$\Gamma$}-convergence for free-discontinuity problems in linear elasticity: homogenization and relaxation.
\newblock {\em Indiana Univ. Math. J.}, 72(5):1949--2023, 2023.

\bibitem{FuscoHutchinson}
N.~Fusco and J.~E. Hutchinson.
\newblock A direct proof for lower semicontinuity of polyconvex functionals.
\newblock {\em Manuscripta Math.}, 87(1):35--50, 1995.

\bibitem{GiacominiPonsiglione}
A.~Giacomini and M.~Ponsiglione.
\newblock A {$\Gamma$}-convergence approach to stability of unilateral minimality properties in fracture mechanics and applications.
\newblock {\em Arch. Ration. Mech. Anal.}, 180(3):399--447, 2006.

\bibitem{giusti1984minimal}
E.~Giusti and G.~H. Williams.
\newblock {\em Minimal surfaces and functions of bounded variation}, volume~80.
\newblock Springer, 1984.

\bibitem{Griffith}
A.~A. Griffith.
\newblock The phenomena of rupture and flow in solids.
\newblock {\em Philosophical Transactions of the Royal Society of London. Series A, Containing Papers of a Mathematical or Physical Character}, 221:163--198, 1921.

\bibitem{LichtGerard}
C.~Licht and G.~Michaille.
\newblock Global-local subadditive ergodic theorems and application to homogenization in elasticity.
\newblock {\em Ann. Math. Blaise Pascal}, 9(1):21--62, 2002.

\bibitem{MainikMielke2009}
A.~Mainik and A.~Mielke.
\newblock Global existence for rate-independent gradient plasticity at finite strain.
\newblock {\em J. Nonlinear Sci.}, 19(3):221--248, 2009.

\bibitem{MassVitto}
A.~Massaccesi and D.~Vittone.
\newblock An elementary proof of the rank-one theorem for {BV} functions.
\newblock {\em J. Eur. Math. Soc. (JEMS)}, 21(10):3255--3258, 2019.

\bibitem{mielke2015rate}
A.~Mielke and T.~Roub{\'\i}{\v{c}}ek.
\newblock Rate-independent systems.
\newblock {\em Applied Mathematical Sciences}, 193, 2015.

\bibitem{NumerApprox}
A.~Mielke and T.~s. Roub{\'\i}{\v{c}}ek.
\newblock Numerical approaches to rate-independent processes and applications in inelasticity.
\newblock {\em M2AN Math. Model. Numer. Anal.}, 43(3):399--428, 2009.

\bibitem{MielkeRoubStef2008}
A.~Mielke, T.~s. Roub{\'\i}{\v{c}}ek, and U.~Stefanelli.
\newblock {$\Gamma$}-limits and relaxations for rate-independent evolutionary problems.
\newblock {\em Calc. Var. Partial Differential Equations}, 31(3):387--416, 2008.

\bibitem{MielTimof}
A.~Mielke and A.~M. Timofte.
\newblock Two-scale homogenization for evolutionary variational inequalities via the energetic formulation.
\newblock {\em SIAM J. Math. Anal.}, 39(2):642--668, 2007.

\end{thebibliography}

\end{document}